\documentclass[10pt,reqno]{article} 
\usepackage[latin1]{inputenc}   % pour les accents
\usepackage{amsfonts,amsmath,layout}
\usepackage{mathrsfs}  
\usepackage{amssymb,mathabx,amsfonts,amsthm,amscd,stmaryrd,dsfont,esint,upgreek,constants,todonotes}
\usepackage{hyperref}
\DeclareFontFamily{OT1}{pzc}{}
\DeclareFontShape{OT1}{pzc}{m}{it}{<-> s * [1.10] pzcmi7t}{}
\DeclareMathAlphabet{\mathpzc}{OT1}{pzc}{m}{it}

\usepackage{color} % temporary

\definecolor{Red}{cmyk}{0,1,1,0.2}
\newcommand{\N}{\mathbb N}

\newcommand{\R}{\mathbb R}

\def\R{\mathbb R}
\def\N{\mathbb N}

\def\E{\mathbb E}

\def\ep{\epsilon}
 
\def\lg{\langle} 
\def\m{\noalign{\medskip}}
\def\dw{{\bf d}_2}
\def\dk{{\bf d}_1}

\newcommand{\be}{\begin{equation}}
\newcommand{\ee}{\end{equation}}
\def\1{{\bf 1}}

\def\dive{{\rm div}}

\def\Pw{{\mathcal P}_2}

\def\ds{\displaystyle}
\newtheorem{Theorem}{Theorem}[section]
\newtheorem{Definition}[Theorem]{Definition}
\newtheorem{Proposition}[Theorem]{Proposition}
\newtheorem{Lemma}[Theorem]{Lemma}
\newtheorem{Corollary}[Theorem]{Corollary}
\newtheorem{Remark}[Theorem]{Remark}

\begin{document}

%\pagestyle{plain}
%\title{MFGf\\ (Draft)}

\title{Splitting methods and short time existence for the master equations in mean field games}
\author{Pierre Cardaliaguet\thanks{Universit\'e Paris-Dauphine, PSL Research University, CNRS, Ceremade, 75016 Paris, France. cardaliaguet@ceremade.dauphine.fr} \and Marco Cirant\thanks{Universit\`a di Parma, Parco Area delle Scienze 53/a, 43124 Parma, Italy. marcoalessandro.cirant@unipr.it} \and Alessio Porretta\thanks{Universit\`a di Roma Tor Vergata, Via della Ricerca Scientifica 1, 00133 Roma, Italy. porretta@mat.uniroma2.it}}

\maketitle

\begin{abstract} We develop a splitting method to prove the well-posedness, in short time, of solutions for two master equations in mean field game (MFG) theory: the second order master equation, describing MFGs with a common noise, and the system of master equations associated with MFGs with a major player. Both problems are infinite dimensional equations stated in the space of probability measures. Our new approach simplifies, shortens and generalizes previous existence results for  second order master equations and provides the first existence result for  systems associated with MFG problems with a major player. 
\end{abstract}

\tableofcontents

%%%%%%%%%%%%%%%%%%%%%%%
\section{Introduction}

The paper is dedicated to a  construction  of a solution of the so-called ``master equations" in mean field game theory (MFG). These equations have been introduced by Lasry and Lions and discussed by  Lions in \cite{LLperso}. Let us recall that mean field games describe the behavior of infinitely many agents in interaction. We consider here two problems: the master equation with common noise and the master equation with a major player. We present a general approach valid for both problems. 

Let us first discuss the master equation with common noise. In this problem, the agents are subject to a common source of randomness. The master equation is then a second order equation in the space of measures and reads as follows: 
\be\label{MasterEqIntro}
\left\{\begin{array}{l} 
\ds - \partial_t U(t,x,m)  - {\rm Tr}((a(t,x)+a^0(t,x))D^2_{xx}U(t,x,m))+H(x,D_xU(t,x,m),m)\\
\ds \qquad -\int_{\R^d}{\rm Tr}((a(t,y)+a^0(t,y))D^2_{ym}U(t,x,m,y))\ dm(y) \\
\ds    \qquad    + \int_{\R^d} D_m U(t,x,m,y)\cdot H_p(y,D_xU(t,y,m),m) \ dm(y) \\
\ds \qquad -2\int_{\R^d} {\rm Tr} \left[\sigma^0(t,y)(\sigma^0(t,x))^TD^2_{xm}U(t,x,m,y)\right]m(dy) \\
\ds    \qquad  
-\int_{\R^{2d}} {\rm Tr}[\sigma^0(t,y)(\sigma^0(t,y'))^TD^2_{mm}U(t,x,m,y,y')]m(dy)m(dy') = 0 \\
 \ds \qquad \qquad  \qquad {\rm in }\; (0,T)\times \R^d\times \Pw\\
 U(T,x,m)= G(x,m) \qquad  {\rm in }\;  \R^d\times \Pw\\
\end{array}\right.
\ee
In the above equation, the unknown $U=U(t,x,m)$ is scalar valued and depends on the time variable $t\in [0,T]$, the space variable $x\in \R^d$ and the distribution of the agents $m\in \Pw$ ($\Pw$ is the space of Borel probability measures with finite second order moment); the derivatives $D_mU$ and $D^2_{mm}U$ refer to the derivative with respect to the probability measure (see subsection \ref{subsec.derivatives}); the maps $H=H(x,p,m)$ and $G=G(x,m)$ reflect the running and terminal costs of the agents. The matrix valued function $a=a(t,x)$ is the volatility term corresponding to idiosyncratic noise of the small players while $a^0=a^0(t,x)= \sigma^0(\sigma^0)^T(t,x)$ is the volatility corresponding to the common noise.

As explained by Lions \cite{LLperso}, the master equation can be understood as a non-linear transport equation in the space of probability measures. When $a^0 =0$ (i.e., in the so-called first order master equation), the characteristics of this transport equation are  given by the MFG system:  if we fix an initial time $t_0$ and an initial probability measure $m_0$ on $\R^d$, and if the pair $(u,m)$ is a solution of the MFG system 
\be\label{eq.MFGsystIntro}
\left\{ \begin{array}{rl}
\ds (i) &  -\partial_t u -{\rm Tr}(a(t,x)D^2u)+H(x,Du,m(t))=0\qquad {\rm in}\; (t_0,T)\times \R^d\\
\ds (ii) &  \partial_t m -\sum_{i,j} D_{ij}(a_{i,j}m) -\dive( m H_p(x,Du,m(t))=0\qquad {\rm in}\; (t_0,T)\times \R^d\\
\ds  (iii)& m(t_0)=m_0, \qquad u(T,x)= G(x,m(T)) \qquad {\rm in}\;\R^d
\end{array}\right.
\ee
then we expect the equality 
\be\label{Rel.Uu}
U(t,x,m(t))= u(t,x) \qquad \forall t\in [t_0,T].
\ee
The interpretation of the MFG system \eqref{eq.MFGsystIntro}  is the following: the map $u$ is the value function of a typical small agent (anticipating the evolution of the population density $(m(t))$) and accordingly solves the Hamilton-Jacobi equation \eqref{eq.MFGsystIntro}-(i). When this agent plays in an optimal way, the drift in the dynamic of its state is given by the term $-H_p(x,Du,m(t))$. By a mean field argument (assuming that the noise of the agents are independent), the resulting evolution of the population density $\tilde m$ satisfies the Kolmogorov equation 
$$
\left\{ \begin{array}{l}
\ds \partial_t \tilde m -\sum_{i,j} D_{ij}(a_{i,j}\tilde m) -\dive(\tilde  m H_p(x,Du,m(t))=0\qquad {\rm in}\; (t_0,T)\times \R^d\\
\ds m(t_0)=m_0 \qquad {\rm in}\;\R^d
\end{array}\right.
$$
In an equilibrium configuration, i.e., when agents anticipate correctly the evolving measure, one has $\tilde m=m$ and therefore the population density $m$ solves \eqref{eq.MFGsystIntro}-(ii). 

The existence/uniqueness of the solution for the MFG system is rather well understood: it relies on Schauder estimates,  fixed point methods  and monotonicity arguments (see, in particular, \cite{LL06cr2, LL07mf}).  From the well-posedness of the MFG system, one can  derive the existence of a solution to the first order master equation ``quite easily": one just needs to define the map $U$ by \eqref{Rel.Uu} with $t=t_0$ and  check that the map $U$ thus defined is a classical solution to the first order master equation. This is the path followed in \cite{GS14-2, May} (when there is no diffusion at all:  $a=a^0\equiv 0$) and in \cite{CgCrDe} (when $a>0$ is constant and $a^0=0$). See also \cite{CDLL} for a similar result (in the torus) using PDE linearization techniques. 

When $a^0\not \equiv 0$ (i.e., for the second order master equation, or master equation with a common noise), the characteristics  are now given by the system of SPDEs (called ``stochastic MFG system"):  
\be\label{eq.MFGsIntro}
\left\{ \begin{array}{l}
\ds 
 d u(t,x) = \bigl[ -{\rm Tr}((a+a^0)(t,x)D^2u(t,x))+ H(x,Du(t,x),m(t)) \\
 \ds \qquad \qquad \qquad\qquad -  \sqrt{2} {\rm Tr}(\sigma^0(t,x) Dv(t,x)) \bigr] dt + v(t,x) \cdot dW_{t} \qquad \qquad {\rm in }\; (0,T)\times \R^d,\\
\ds d m(t,x) = \bigl[ \sum_{i,j} D_{ij}(( (a_{ij})+a^0_{ij})(t,x)m(t,x)) + {\rm div} \bigl( m(,x) D_{p} H(x,D u(t,x),m(t))\bigr) \bigr] dt\\
 \qquad \qquad \qquad \qquad \ds  - {\rm div} ( m(t,x) \sqrt{2}\sigma^0(t,x) dW_{t} \bigr), 
 \qquad \qquad {\rm in }\; (0,T)\times \R^d,
 \\
 \ds u(T,x)=G(x,m(T)), \; m(0)=m_0,\qquad {\rm in }\;  \R^d
\end{array}\right.
\ee
In the above system, $(W_t)$ is the common noise (here a Brownian motion) and   the unknown is the triplet $(u,m,v)$, where  the new variable $v$  (a random vector field in $\R^d$) ensures the solution $u$ of the backward Hamilton-Jacobi (HJ) equation to be adapted to the filtration generated by the common noise $(W_t)$. 
The analysis of this system is much more involved than the deterministic one: Schauder estimates are no longer available and the usual fixed point methods based on compactness arguments can no longer be applied. One has to replace them by continuation methods, which are much heavier to handle, see \cite{CDLL}. Besides the PDE approach we just  mentioned, MFG with common noise can also be handled through a probabilistic formulation: see the pioneering result \cite{CDLcn}, as well as \cite{Ah, LW} and the monograph \cite{CaDebook}. Once the analysis of the stochastic MFG system has been performed, one can proceed with the construction of the second order master equation as in the first order case, defining the map $U$ by \eqref{Rel.Uu} for $t=t_0$, where $u$ is now the $u-$component of the solution of the stochastic MFG system ($u(t_0,\cdot)$ turns out to be deterministic). However, here again, the verification that the map $U$  defined so far is smooth enough to satisfy \eqref{MasterEqIntro} requires a lot of work: see \cite{CDLL} and \cite{CaDebook}. 

Let us finally recall another approach, suggested by P.-L. Lions in the seminar \cite{LLperso2}: it consists in writing the equation for the quantity $D_mU$ as an hyperbolic equation in a Hilbert space of random variables. The construction requires, however,  convexity conditions on the system with respect to the space variable (but no uniform ellipticity for the matrix $a$). \\

We now discuss the second equation considered in this paper: the master  equation corresponding to MFG models with a major player. It reads  as follows: 
\be\label{eq.MmIntro}
\left\{\begin{array}{rl}
\ds (i)& \ds -\partial_t U^{0}-\Delta_{x_0} U^0+ H^0(x_0,D_{x_0}U^{0}, m) -\int_{\R^d} \dive_yD_m U^0(t,x_0,m,y)dm(y)\\
&\ds \qquad + \int_{\R^d} D_mU^0(t,x_0,m,y)\cdot H_p(x_0,y, D_{x}U(t,x_0,y,m), m)dm(y)=0 \\ 
&\ds \hspace{8cm} {\rm in} \; (0,T)\times \R^{d_0}\times \Pw,\\
\ds (ii) & \ds -\partial_t U- \Delta_{x} U-\Delta_{x_0}U+ H(x_0,x,D_{x}U, m) -\int_{\R^d} \dive_y D_mU(t,x_0,x,m,y)dm(y)\\
&\ds \qquad  + D_{x_0}U \cdot H^0_p(x_0,D_{x_0}U^0(t,x_0,m), m)\\
&\ds \qquad   + \int_{\R^d} D_m U(t,x_0,x,m,y)\cdot H_p(x_0,y, D_{x}U(t,x_0,y,m), m)dm(y)=0 \\ 
&\ds \hspace{8cm} {\rm in} \; (0,T)\times \R^{d_0}\times \R^{d}\times \Pw,\\
\ds (iii) & \ds U^{0}(T,x_0,m)= G^0(x_0,m), \qquad  {\rm in} \; \R^{d_0}\times \Pw,\\
\ds (iv) & \ds U(T,x_0, x,m)= G(x_0,x,m)\qquad  {\rm in} \; \R^{d_0}\times \R^{d}\times \Pw.\\
\end{array}\right.
\ee
In the above system, $U^0=U^0(t,x_0,m)$ corresponds to the  payoff at equilibrium for a major player interacting with a crowd  in which each agent  has at equilibrium  a payoff given by $U=U(t,x_0,x, m)$. Here $m$ is the distribution law of the agents. Notice that each agent is influenced by the major player whereas the latter is only influenced by the distribution of the whole population. Mean field games with a major player have been first discussed by Huang in \cite{H10} and several notions of equilibria, in different contexts, have been proposed in the literature since then: see \cite{BCY15, BCY16, BLP14, CK14, CaDebook, CW16, CW17, CZ16, LL18}. The above system has been introduced by Lasry and Lions in \cite{LL18}.  In the companion paper \cite{CCP}, we explain how the above master equation is related to the approach by Carmona and al. \cite{CW16, CW17, CZ16}. Concerning the existence of a solution, \cite{CZ16} shows the existence of an equilibrium in short time for the case of a finite state space, \cite{LL18} proves the existence of a solution to the master equation still in the finite state space framework and notes that the Hilbertian techniques described in \cite{LLperso2} could be adapted to the master equation with a major player \eqref{eq.MmIntro}. \\

The purpose of this paper is to introduce a   different path towards the construction of a solution to the second order master equation and to the master equation with a major player, using as a building block the construction of a solution to the first order master equation. For the second order master equation, we justify this point of view by the fact that the deterministic MFG system and  the   first order master equation are much easier to manipulate than the stochastic MFG system. Our approach allows for instance to build solutions of the second order master equation (in short time) under more general assumptions than in \cite{CDLL, CaDebook}. For the MFG problem with a major player, we prove for the first time the (short time) well-posedness of the associated system of master equations in continuous space. 

Let us first explain our ideas for the master equation with common noise \eqref{MasterEqIntro}. In contrast to previous works, we do not use directly the representation formula \eqref{Rel.Uu} (for $t=t_0$) for the solution of the second order master equation. 
Instead, we somehow decompose the second order master equation as the superposition of the first order master equation: 
\be\label{eq.transport}
\left\{\begin{array}{l} 
\ds - \partial_t U  - {\rm Tr}(a(t,x)D^2_{xx}U) +H(x,D_xU,m) -\int_{\R^d}{\rm Tr}(a(t,y)D^2_{ym}U)\ dm(y) \\
\ds    \qquad   + \int_{\R^d} D_m U\cdot H_p(y,D_xU,m) \ dm(y) =0 
%\\
% \ds \qquad \qquad  
 \qquad {\rm in }\; (0,T)\times \R^d\times \Pw\\
 U(T,x,m)= G(x,m) \qquad  {\rm in }\;  \R^d\times \Pw\\
\end{array}\right.
\ee
and of a linear second order master equation:
\be\label{eq.master1}
\left\{\begin{array}{l} 
\ds - \partial_t U -{\rm Tr}\left[ \sigma^0(\sigma^0)^T(t,x) D^2_{xx} U\right]
%\\
%\ds \qquad  
-\int_{\R^d} {\rm Tr}\left[\sigma^0(\sigma^0)^T(t,y) D^2_{ym}U\right] m(dy)
\\
\ds \qquad 
-2\int_{\R^d} {\rm Tr} \left[\sigma^0(t,y)(\sigma^0(t,x))^TD^2_{xm}U\right]m(dy) 
\\
\ds    \qquad  
-\int_{\R^{2d}} {\rm Tr}[\sigma^0(t,y)(\sigma^0(t,y'))^TD^2_{mm}U]m(dy)m(dy') = 0 \\
%\\
% \ds \qquad \qquad  
 \qquad\qquad\qquad\qquad\qquad\qquad {\rm in }\; (0,T)\times \R^d\times \Pw\\
 U(T,x,m)= G(x,m) \qquad  {\rm in }\;  \R^d\times \Pw\\
\end{array}\right.
\ee
The  solution to this linear second order master equation is just given by a Feynman-Kac formula, and thus it is very easy to handle. Then we use Trotter-Kato formula, alternating the two equations in short time intervals to build in the limit a solution of the full equation \eqref{MasterEqIntro}. If the technique is quite transparent, its actual implementation requires some care. Indeed, one has to check that, at each step of the process, the regularity of the solution does not deteriorate too much, meaning at least in a linear way in time. The aim of Section \ref{Sec.FirstOrdreMaster} is precisely to quantify this deterioration for the solution $U$ of the first order master equation \eqref{eq.transport}, as well as for its derivatives in the measure variable. As the solution of \eqref{eq.transport} is built by using the representation formula \eqref{Rel.Uu} (where $t=t_0$) presented above, one has first to do the analysis on the MFG system \eqref{eq.MFGsystIntro} and this is the aim of Section \ref{sec:ReguMFG}. Note that we are able to control  the regularity of the linear second order equation \eqref{eq.master1} only when the matrix $a^0$ is constant. Hence we only prove the short time existence of a solution to \eqref{MasterEqIntro}  in that case. 

For the problem with a major player, we argue in a  similar way: we view equation \eqref{eq.MmIntro} as the superposition of two systems: the first one is a first order system of master equations (for a fixed $x_0$): 
$$
\left\{\begin{array}{rl}
\ds (i)& \ds -\partial_t U^{0} -\int_{\R^d} \dive_yD_m U^0(t,x_0,m,y)dm(y)\\
&\ds \qquad + \int_{\R^d} D_mU^0(t,x_0,m,y)\cdot H_p(x_0,y, D_{x}U(t,x_0,y,m), m)dm(y)=0 \\ 
%&\ds \hspace{8cm} {\rm in} \; (0,T)\times \Pw \\
\ds (ii) & \ds -\partial_t U- \Delta_{x} U+ H(x_0,x,D_{x}U, m) -\int_{\R^d} \dive_y D_mU(t,x_0,x,m,y)dm(y)\\
&\ds \qquad   + \int_{\R^d} D_m U(t,x_0,x,m,y)\cdot H_p(x_0,y, D_{x}U(t,x_0,y,m), m)dm(y)=0 \\ 
%&\ds \hspace{8cm} {\rm in} \; (0,T)\times \R^d\times \Pw
\end{array}\right.
$$
It turns out that this system can be handled by the method of characteristics. As for the  second one, it is a simple system of HJ equations (for fixed $x,m$):
$$
\left\{\begin{array}{rl}
\ds (i)& \ds -\partial_t U^{0}-\Delta_{x_0} U^0+ H^0(x_0,D_{x_0}U^{0}, m)  =0  \\ 
%&\ds \hspace{8cm} {\rm in} \; (0,T)\times \R^d\\
\ds (ii) & \ds -\partial_t U- \Delta_{x_0}U+ D_{x_0}U \cdot H^0_p(x_0,D_{x_0}U^0(t,x_0,m), m)=0 \,.\\ 
%&\ds \hspace{8cm} {\rm in} \; (0,T)\times \R^d
\end{array}\right.
$$
The idea of splitting time is not completely new in the framework of mean field games. For instance, the construction, given in  \cite{CDLcn},  of (weak) equilibria for MFG problems with common noise relies on a time splitting. The main difference is that it is done at the level of the MFG equilibrium, while we do the construction at the (stronger) level of the master equation. One  consequence is that, with our approach, the construction of a solution to the stochastic MFG system (in short time, though) is straightforward once the solution of the master equation is built, while deriving a solution of the master equation from the stochastic MFG system is much trickier.  Let us also quote the paper in preparation \cite{ALLRS} in which the authors use a splitting technique similar to the one described above to compute numerically the solution  of MFGs with a major player.
\\

Let us finally point out that, in this paper, we do not address at all the problem of the existence of a solution on a large time interval. For the first and second order master equation, this question is related to the Lasry-Lions monotonicity condition \cite{LL06cr2, LL07mf}. The existence of a solution  on a large time interval can be obtained under this condition either  by the Hilbertian approach, as explained in  \cite{LLperso2},  or by a continuation method, as  in \cite{CgCrDe} and  \cite{CaDebook} or even directly by using the long time existence of a solution for the MFG system, as in \cite{CDLL}. Let us recall that, when the monotonicity condition is not fulfilled, the solution to the second order master equation is expected to develop shocks (i.e., discontinuities) in finite time. Note also that a structure condition  similar to the monotonicity condition is not known for MFGs with a major player. \\

The paper is organized in the following way. 
In Section  \ref{sec.notation} we  fix the notation, we recall the definition of derivatives in the space of measures and state our main assumptions. 
The main existence results for the second order master equation (equation \eqref{MasterEqIntro}) and for the system of master equations for MFG with a major player (system \eqref{eq.MmIntro}) are collected in Sections \ref{sec.master2} and \ref{sec.Mm} respectively. Both sections require estimates on the first order master equations. As first order master equations are  built by the method of characteristics involving the solutions of classical MFG systems  \eqref{eq.MFGsystIntro},  Section \ref{sec:ReguMFG} first provides estimates for these systems. Then Section \ref{Sec.FirstOrdreMaster} is devoted to the first order master equations. We complete the paper by appendices in which we prove short-time estimates for the standard Hamilton-Jacobi equations (Section \ref{sec.HJ}) and we discuss several facts on maps defined on the space of measures (differentiability, interpolation and Ascoli Theorem, Section \ref{sec.diffUm}). 
\\

\noindent {\bf Acknowledgement. } The first author was partially supported by the ANR (Agence Nationale de la Recherche) project ANR-16-CE40-0015-01 and by the AFOSR  grant FA9550-18-1-0494. The second author was partially supported by the Fondazione CaRiPaRo Project ``Nonlinear Partial Differential Equations: Asymptotic Problems and Mean-Field Games'' and the Programme ``FIL-Quota Incentivante'' of University of Parma, co-sponsored by Fondazione Cariparma. The third author was partially supported by Indam (Gnampa national project 2018) and by FSMP (Foundation Sciences Math\'ematiques de Paris).

%\newpage

%%%%%%%%%%%%%%%%%%%%%%%%%%%%
%\section{Assumptions and preliminary results}
\section{Notation and assumptions}\label{sec.notation}

\subsection{Notation}

Throughout the paper, we work in the euclidean space $\R^d$ (with $d\in \N$, $d\geq 1$), endowed with the scalar product $(x,y)\to x\cdot y$ and the distance $|\cdot|$. Given $T>0$ and a map 
%$\phi: \R^d\to \R$, we denote by $\partial_{x_i}\phi$ its partial derivative with respect to the $i-$th variable ($i=1,\dots, d$) and by $D\phi$  its gradient. In the same way, if $T>0$ and 
$\phi:(0,T)\times \R^d\to \R$, we denote by $\partial_t\phi$ the derivative of $\phi$ with respect to the time-variable, by  $\partial_{x_i}\phi$ its partial derivative with respect to the $i-$th space variable ($i=1,\dots, d$) and by $D\phi$ the gradient with respect to the space variable. 

For $n\in\N$, we denote by $C^n_b$  the set of maps $\phi:\R^d\to \R$ which are $n-$times differentiable with continuous and bounded derivatives: in particular, $C^0_b$ is the set of continuous and bounded maps.  Given $\phi\in C^n_b$ and a multi-index $k=(k_1,\dots, k_d)\in \N^d$, with length $|k|:=\sum_{i=1}^{d} k_i\leq n$, we denote by $\partial^k \phi= \frac{\partial^{k_1}}{\partial x_1^{k_1}}\dots \frac{\partial^{k_d}}{\partial x_d^{k_d}}\phi$ (or briefly $\phi_k$) the $k-$th derivative of $\phi$.  We also denote by $D^n\phi$ ($n\in \N$, $n\geq 1$) the vector $(\partial^k\phi)_{|k|= n}$.  The norm of $\phi$ in $C^n_b$ is 
$$
\|\phi\|_n:= \sum_{r=0}^n\sup_x \left(\sum_{|\alpha|=r}|\partial^{\alpha}\phi(x)|^2\right)^{1/2} = \sum_{r=0}^n \|D^r \phi\|_\infty. 
$$
For $n=0$, we use indifferently the notation $\|\phi\|_0$ or $\|\phi\|_\infty$.

For $(n_1, \dots, n_k)\in \N^k$ ($k\in \N$, $k\geq 2$), we denote by $C^{n_1, \dots, n_k}_b$ the space of functions $\phi:\R^{d_1}\times \dots\times \R^{d_k}\to \R$ ($d_i\geq 1$) having continuous and bounded derivatives $D^{l_1}_{x_1}\cdots D^{l_k}_{x_k}\phi$ for all $l_1 \leq n_1, \ldots, l_k \leq n_k$, endowed with the norm
$$
\|\phi\|_{n_1,\dots, n_k} = \|\phi(\cdot_{x_1},\dots, \cdot_{x_k})\|_{n_1,\dots, n_k} := \sum_{l_1 \leq n_1, \ldots, l_k \leq n_k} \|D^{l_1}_{x_1}\cdots D^{l_k}_{x_k}\phi\|_\infty,
$$
where now $(x_1, \dots, x_k)$ stands for a generic element $\R^{d_1}\times \R^{d_k}$. 

We denote by $C^{-n}$ the dual space of $C^n_b$, endowed with the usual norm 
$$
\|\rho\|_{-n}:= \sup_{\|\phi\|_n\leq 1} \rho(\phi) \qquad \forall \rho\in C^{-n}.
$$

Finally, when a map $\phi=\phi(t,x)$ depends also on time $t$ belonging to an interval $I$, we often write 
$\sup_{t\in I}\|\phi(t)\|_n$ for $\sup_{t\in I}\|\phi(t,\cdot)\|_n$. We use a corresponding notation for a map $\rho\in C^0([0,T], C^{-k})$. \\

Throughout the paper,  ${\mathcal P}$ stands for the set of Borel probability measures on $\R^d$ and for $k\geq 1$, ${\mathcal P}_k$ stands for the set of measures in ${\mathcal P}$ with finite moment of order $k$: namely,
$$
M_k(m):= \left(\int_{\R^d} |x|^k m(dx)\right)^{1/k} <+\infty \qquad {\rm if} \; m\in {\mathcal P}_k. 
$$
The set ${\mathcal P}_k$ is endowed with the distance  (see for instance \cite{AGS, RaRu98, villani})
$$
{\bf d}_k(m,m') =\inf_{\pi} \left( \int_{\R^d} |x-y|^k\pi(dx,dy)\right)^{1/k}, \qquad \forall m,m'\in {\mathcal P}_k, 
$$
where the infimum is taken over the couplings $\pi$ between $m$ and $m'$, i.e., over the Borel probability measures $\pi$ on $\R^d\times \R^d$ with first marginal $m$ and second marginal $m'$. Note that $\Pw\subset {\mathcal P}_1$ and ${\bf d}_1\leq {\bf d}_2$ by Cauchy-Schwarz inequality. We will often use the fact that, if $\phi:\R^d\to \R$ is Lipschitz continuous with a Lipschitz constant $L\geq 0$, then
$$
\left| \int_{\R^d} \phi(x)(m-m')(dx)\right| \leq L {\bf d}_1(m,m'), \qquad \forall m,m'\in {\mathcal P}_1. 
$$
Moreover,  ${\bf d}_1(m,m')$ is the smallest constant for which the above inequality holds for any $L-$Lipschitz continuous map $\phi$  (see for instance \cite{RaRu98, villani}). Given $m\in {\mathcal P}$ and $\phi\in C^0_b$,  the image $\phi\sharp m$ of $m$ by $\phi$ is the element of ${\mathcal P}$ defined by 
$$
\int_{\R^d} f(x) \ m\sharp \phi (dx)= \int_{\R^d} f(\phi(x)) m(dx)\qquad \forall f\in C^0_b.
$$

%%%%%%%%%
\subsection{Derivatives in the space of measures}\label{subsec.derivatives}

We now define the derivative in the space ${\mathcal P}_2$. For this, we follow mostly the definition and notations introduced in \cite{CDLL} (in a slightly different context) and which are reminiscent of earlier approaches: see \cite{AKR, AGS} and the references in \cite{CaDebook}. We say that a map $U:\Pw\to \R$ is $C^1$ if there exists a {\it continuous and bounded} map $\frac{\delta U}{\delta m}:\Pw\times \R^d\to \R$ such that 
\be\label{zalekrnjdg}
U(m')-U(m)= \int_0^1 \int_{\R^d} \frac{\delta U}{\delta m}((1-s)m+sm',y) (m'-m)(dy)ds \qquad \forall m,m'\in \Pw. 
\ee
Note that the restriction on $\frac{\delta U}{\delta m}$ to be continuous on the entire space $\R^d$ and globally bounded is restrictive: it will however simplify our forthcoming construction. The map $\frac{\delta U}{\delta m}$ is defined only up to an additive constant that we fix with the convention 
\be\label{eq.convention}
\int_{\R^d} \frac{\delta U}{\delta m}(m,y)m(dy)=0 \qquad \forall m\in \Pw.
\ee
We say that the map $U$ is continuously $L-$differentiable (in short: $L-C^1$) if $U$ is $C^1$ and if $y\to \frac{\delta U}{\delta m}(m,y)$ is everywhere differentiable with a  continuous and globally bounded derivative on $\Pw\times \R^d$. We   denote by
$$
D_mU(m,y):= D_y \frac{\delta U}{\delta m}(m,y)
$$
this $L-$derivative. In view of the discussion in \cite{CDLL}, $D_mU$ coincides with the Lions derivative as introduced in \cite{LLperso} and discussed in \cite{CaDebook}. In particular, 
it estimates the Lipschitz regularity of $U$ in $\Pw$ (Remark 5.27 in \cite{CaDebook}): 
 \be\label{ineq.LipU}
 |U(m)-U(m')| \leq {\bf d}_2(m,m') \sup_{\mu\in \Pw} \left( \int_{\R^d} |D_mU(\mu,y)|^2\mu(dy)\right)^{1/2}\qquad \forall m,m'\in \Pw. 
 \ee
Of course one can also  estimate the Lipschitz regularity of $U$ through the ${\bf d}_1$ norm, as
 \be\label{ineq.LipU2}
  |U(m)-U(m')| \leq {\bf d}_1(m,m') \sup_{\mu\in \Pw} \|D_mU(\mu, \cdot)\|_\infty\leq {\bf d}_2(m,m') \sup_{\mu\in \Pw} \|D_mU(\mu, \cdot)\|_\infty\,.
 \ee
Note that, with our boundedness convention, if  $U$ is continuously $L-$differentiable, then  $U$ is automatically globally Lipschitz continuous. 

When $U$ is smooth enough, we often see the map $\frac{\delta U}{\delta m}$ as a linear map on $C^{-k}$ by 
 $$
 \frac{\delta U}{\delta m}(m)(\rho)= \lg \rho , \frac{\delta U}{\delta m}(m,\cdot)\rangle_{C^{-k},C^k} \qquad \forall \rho\in C^{-k}.
 $$

 We say that $U$ is $C^2$ if $\frac{\delta U}{\delta m}$ is $C^1$ in $m$ with a continuous and bounded derivative: namely $\frac{\delta^2 U}{\delta m^2}= \frac{\delta}{\delta m}(\frac{\delta U}{\delta m}):\Pw\times \R^d\times\R^d\to \R$ is continuous in all variables and bounded. We say that $U$ is twice $L-$differentiable if the map $D_m U$ is $L-$differentiable  with respect to $m$ with a second order derivative $D^2_{mm}U=D^2_{mm} U(m, y,y')$ which is continuous and bounded on $\Pw\times \R^d\times \R^d$ with values in $\R^{d\times d}$. 
 
 When a map $U:\R^d\times \Pw\to \R$ is of class $C^n_b$ with respect to the space variable, uniformly with respect to the measure variable, we often set 
\be\label{laezbrebd}
\|U\|_n:=\sup_{m\in \Pw} \|U(\cdot,m)\|_n. 
\ee
We use similar notation for a map $U$ depending on several space variables and on a measure.  
   
When a map $U:\R^d\times \Pw\to \R$ is Lipschitz continuous with respect to $m$, uniformly with respect to the space variable in some $C^n$ norm, we define ${\rm Lip}_{n}(U)$ as the smallest constant $C$ such that 
$$
\|U(\cdot, m_1)-U(\cdot,m_2)\|_n \leq C {\bf d}_2(m_1,m_2)\qquad \forall m,m'\in \Pw.
$$
Namely: 
$$
{\rm Lip}_{n}(U):= \sup_{m_1\neq m_2} \frac{\|U(\cdot, m_1)-U(\cdot,m_2)\|_n}{{\bf d}_2(m_1,m_2)}. 
$$
More generally, if $U:(\R^d)^k\times \Pw\to \R$ (for $k\in \N$, $k\geq 1$) is Lipschitz continuous in the measure variable in some $C^{n_1,\dots, n_k}_b$ norm (where $n_i\in \N$ for $i=1,\dots, k$), then we set 
$$
{\rm Lip}_{n_1, \dots, n_k}(U):= \sup_{m_1\neq m_2} \frac{\|U(\cdot_{x_1},\dots, \cdot_{x_k}, m_1)-U(\cdot_{x_1},\dots, \cdot_{x_k},m_2)\|_{n_1, \dots, n_k}}{{\bf d}_2(m_1,m_2)}. 
$$
We will  typically use this notation for the derivatives of a map $U: \R^d\times \Pw\to \R$: indeed we will often have to estimate quantities of the form 
$$
{\rm Lip}_{n_1,n_2}(D_mU):= \sup_{m_1\neq m_2} \frac{\|D_mU(\cdot_x, m_1,\cdot_y)-D_mU(\cdot_x, m_2,\cdot_y)\|_{n_1,n_2}}{{\bf d}_2(m_1,m_2)}
$$
and 
$$
{\rm Lip}_{n_1, n_2, n_3}(D^2_{mm}U):= \sup_{m_1\neq m_2} \frac{\|D^2_{mm}U(\cdot_{x},m_1, \cdot_{y},\cdot_{y'})-D^2_{mm}U(\cdot_{x},m_2, \cdot_{y},\cdot_{y'}) \|_{n_1, n_2,n_3}}{{\bf d}_2(m_1,m_2)}. 
$$
Concerning the Lipschitz continuity with respect to one of the entries $x_i$, we will use the following notation:
\begin{multline*}
{\rm Lip}^{x_i}_{n_1, \dots, n_{i-1}, n_{i+1}, \ldots, n_k}(U):= \\ \sup_{m, x_i^1\neq x_i^2} \frac{\|U(\cdot_{x_1},\dots, \cdot_{x_{i-1}}, x_i^1, \cdot_{x_{i+1}}, \dots, \cdot_{x_k}, m)-U(\cdot_{x_1},\dots, \cdot_{x_{i-1}}, x_i^2, \cdot_{x_{i+1}}, \dots, \cdot_{x_k}, m)\|_{n_1, \dots, n_{i-1}, n_{i+1}, \ldots, n_k}}{|x_i^1-x_i^2|}.
\end{multline*}

{\bf Further norms:} In order to estimate the $y-$dependence of a derivative with respect to the measure of a map $U=U(x,m)$, we systematically proceed by duality method, testing this derivative against distributions. This yields to the following norms, for  $n,k\in \N$ (note the the subtle difference in notation between $\|\cdot\|_{n,k}$ and $\|\cdot \|_{n;k}$): 
\begin{align*}
& \left\|\frac{\delta U}{\delta m}\right\|_{n;k}:= 
 \sup_{m\in \Pw} \ \sum_{r=0}^n  \ \sup_{ \substack{ x\in \R^{d}, \rho\in C^0_c \\   \ \|\rho\|_{-k} = 1} } 
\left( \sum_{|\alpha| =r } \left|\partial^{\alpha}_{x} \frac{ \delta U}{\delta m} (x,m)(\rho) \right|^2\right)^{1/2}
=\sup_{m\in \Pw} \ \sum_{r=0}^n  \ \sup_{ \substack{ x\in \R^{d}, \rho\in C^0_c, \\ \|\rho\|_{-k} = 1} }
\left|D^r_{x} \frac{ \delta U}{\delta m} (x,m)(\rho) \right| 
,\\
&  \left\|\frac{\delta^2 U}{\delta m^2}\right\|_{n;k,k'}: = 
\sup_{m\in \Pw} \ \sum_{r=1}^n  \ \sup_{ \substack{ x\in \R^{d}, \rho, \rho'\in C^0_c, \\  \|\rho\|_{-k} =\|\rho'\|_{-k'} = 1} } 
\left|D^r_{x}\frac{ \delta^2 U}{\delta m^2} (x,m)(\rho,\rho') \right|.
\end{align*}
For maps $U=U(x_1, x_2,m)$ depending on two (or more) space variables, we use the transparent notation $\|\cdot\|_{n_1,n_2;k}$  (and, if $n_1=0$ (say), we simply set $\|\cdot\|_{n_2, k}=\|\cdot\|_{0,n_2; k}$). Finally, we use similar notation for the Lipschitz norms,  setting, for instance for a map $U=U(x,m)$,  
\begin{align*}
& {\rm Lip}_{n;k,k'} \Big(\frac{\delta^2 U}{\delta m^2}\Big) :=   \sup_{m_1\neq m_2}\dw(m_1,m_2)^{-1}  \  \sum_{r=0}^n\sup_{ \substack{ x\in \R^{d}, \rho, \rho'\in C^0_c\\ \ \|\rho\|_{-k} =\|\rho'\|_{-k'} = 1} }
\left|D^r_{x} \frac{\delta^2 U}{\delta m^2}(x,m_2)(\rho, \rho')- D^r_x \frac{\delta^2 U}{\delta m^2}(x,m_1)(\rho, \rho') \right|. 
\end{align*}

Some comment about the norms we just introduced are now in order. We discuss the norm $\|\cdot \|_{n;k}$ to fix the ideas. With these notations, we have, if $U = U(x,m)$ is smooth enough,  
 $$
\|\frac{\delta U}{\delta m}(\cdot,m)(\rho)\|_{n} \leq  \|\frac{\delta U}{\delta m}\|_{n;k} \|\rho\|_{-k},  
$$ 
for every fixed $m\in \Pw$.
Inequalities of this type are used throughout the text. Next we note that the norms
%\footnote{ if I am not wrong, the first one is defined for a single frozen $m$, while the second is a sup over $m$; if this is so, this sentence is not properly correct and one should write that: "there exist positive constants $c_1,c_2$ such that 
%$$
% c_1\, \|\frac{\delta U}{\delta m}\|_{n;k} \leq  \sup_{m\in \Pw}\, \|\frac{\delta U}{\delta m}(\cdot_x,m,\cdot_y)  \|_{n,k} \leq  c_2\, \|\frac{\delta U}{\delta m}\|_{n;k}
%$$
%} 
$\sup_{m\in \Pw}\|\cdot \|_{n,k}$ and $\|\cdot \|_{n;k}$ are equivalent if we know a priori that $\delta U/\delta m= \delta U/\delta m(x,m,y)$ is in
$C^{n,k}_b$.  In general we do not have this information, but only know that $\delta U/\delta m$ is (at least) continuous. In this case, we use the following remark:

\begin{Lemma}\label{lem.k;k} Let $k\in \N$ with $k\geq 1$ and $u\in C^0$ be such that 
\be\label{lkjknsr,d}
\theta:=\sup_{ \substack{ \rho\in C^0_c, \ \|\rho\|_{-k} = 1} } \int_{\R^d} u(y)\rho(y)dy <+\infty.
\ee
Then $u\in C^{k-1}_b$ with $\|u\|_{k-1}\leq C_k \theta$ (where $C_k$ depends on $d$ and $k$) and, for any $\beta\in \N^d$ with $|\beta|=k-1$, $\partial^\beta u$ is $\theta-$Lipschitz continuous. 
\end{Lemma}

\begin{Remark}\label{rem.k;k} In particular, if $\frac{\delta U}{\delta m} \in C^{n,0}_b$  and $\left\|\frac{\delta U}{\delta m}\right\|_{n;k}$ is finite for some $n,k\in \N$ with $k\geq 1$, then  $\frac{\delta U}{\delta m} \in C^{n,k-1}_b$ and 
$$
\|\delta U/\delta m\|_{n,k-1}\leq C_{n,k}\|\delta U/\delta m\|_{n;k}
$$
for some constant $C_{n,k}$ depending in addition on dimension only. Moreover, the derivatives of the form $\partial_x^\alpha\partial_y^\beta \frac{\delta U}{\delta m}$ for $|\alpha|\leq n$ and $\beta\leq k-1$ are Lipschitz continuous with respect to $y$ and thus---by \eqref{ineq.LipU2}---also with respect to $m$, with a Lipschitz constant bounded by $\left\|\frac{\delta U}{\delta m}\right\|_{n;k}$. 
\end{Remark}

\begin{proof} For $k=1$, we have, approximating Dirac masses by continuous maps with compact support: for any $x,y\in \R^d$,  
$$
|u(x)| \leq \theta \|\delta_x\|_{-1} = \theta 
\qquad {\rm and}\qquad |u(x)-u(y)| \leq \theta \|\delta_x-\delta_y\|_{-1} = \theta |x-y|. 
$$
This proves the claim for $k=1$. Let now assume that \eqref{lkjknsr,d} holds for $k=2$. Then $u$ can be extended to an element $T$ in $(C^{-2})'$ with norm $\|T\|\leq \theta$,  such that $T(\rho)= \int ud\rho$ for any Radon measure  $\rho$. As, for any $v\in \R^d$, 
$$
\lim_{h\to 0, v'\to v} h^{-1} (\delta_{x+hv'}-\delta_x) = -\partial_{v} \delta_x \qquad {\rm in}\; C^{-2}, 
$$
we infer that 
$$
\lim_{h\to 0, v'\to v} h^{-1} (u(x+hv')-u(x))= -T(\partial_{v} \delta_x ). 
$$
The map $(x,v)\to \partial_{v} \delta_x$  being continuous in $C^{-2}$ with $\|\partial_{v} \delta_x\|_{-2} \leq |v|$,  $u$ is in $C^1$ with $\|Du\|\leq \theta$. Then, arguing as for $k=1$, one  can easily check that $Du$ is $\theta-$Lipschitz continuous. So the result also holds for $k=2$. The proof can be completed in the same way for any $k$ by induction. 
\end{proof}

%%%%%%%%%%%%%%%%%%%%%%%%%%%%%
\subsection{Assumptions on the data}\label{subsec.Hyp}

We state here the assumptions needed on $a$, $H$ and $G$ for the existence of a classical solution to the second order master equation \eqref{MasterEqIntro} and to the master equation \eqref{eq.MmIntro} for the MFG problem with a major player. These assumptions are in force throughout the paper. Note that they are common to both problems \eqref{MasterEqIntro} and \eqref{eq.MmIntro} since both require the same kind of estimates on the first order master equation (see Section \ref{Sec.FirstOrdreMaster}).

We assume  that the map $a:[0,T]\times \R^d\to \R^{d\times d}$ can be written as $a=\sigma\sigma^T$  where $\sigma:[0,T]\times \R^d\to \R^d$ ($M\in \N$, \; $M\geq 1$) is bounded in $C^n_b$ with respect to the space variable, uniformly with respect to the time variable, for some  $n\geq 4$. We also assume that the following uniform ellipticity condition holds: 
\be\label{Conda}
a(t,x)\geq C_0^{-1}I_d, \qquad \|Da\|_\infty\leq C_0\,,
\ee
for some $C_0>0$. 

We assume that the map $H:\R^{d_0}\times \R^{d}\times \R^d\times \Pw\to\R$ satisfies the growth condition 
\be\label{CondH}
 \sup_{x_0\in \R^{d_0},\ x\in \R^d, \; m\in \Pw} |D_xH(x_0,x,p,m)|\leq C_0(1+|p|^\gamma),\,\quad \forall p\in \R^d,
\ee
for some $\gamma>1$. We also suppose that, for any $R>0$, the quantities
$$
\|H(\cdot_{x_0},\cdot_x,\cdot_p, m)\|_{3,n,n+1}, \;   \left\|\frac{\delta H}{\delta m}(\cdot_{x_0},\cdot_x,\cdot_p, m, \cdot_y)\right\|_{2,n-1,n,k},
\; \left\|\frac{\delta^2 H}{\delta m^2}(\cdot_{x_0}, \cdot_x,\cdot_p, m, \cdot_y,\cdot_{y'})\right\|_{1, n-2,n-1,k-1,k-1},
$$
and ${\rm Lip}_{1, n-3,n-2,k-1,k-1} (\frac{\delta^2 H}{\delta m^2})$
are bounded for $|p|\leq R$, $m\in \Pw$ and $x_0 \in \R^{d_0}$, for any  $k\in \{2, \dots, n-1\}$. Note that we could also allow for a time dependence for $H$ without changing at all the arguments: we will not do so to simplify a little the notation. For the second order master equation, the Hamiltonian $H$ actually does not depend on $x_0$, but this dependence is important to handle the MFG problem with a major player. 

As for the initial condition $G:\R^{d_0}\times \R^{d}\times \Pw\to \R$, we assume that $G$ is of class $C^2$ with respect to all variables and that the quantities 
\begin{align*}
& \|G(\cdot_{x_0}, \cdot_x, m)\|_{3,n}, \;   \left\|\frac{\delta G}{\delta m}(\cdot_{x_0}, \cdot_x, m,\cdot_y)\right\|_{2, n-1,k},
\; \left\|\frac{\delta^2 G}{\delta m^2}(\cdot_{x_0},\cdot_x, m, \cdot_y, \cdot_{y'})\right\|_{1,n-2,k-1,k-1}, \\ 
& {\rm Lip}_{1,n-3,k-2,k-2} (\frac{\delta^2 G}{\delta m^2})(\cdot_{x_0}, \cdot_x,  m, \cdot_y, \cdot_{y'}),
\end{align*}
are bounded uniformly with respect to  $m\in \Pw$. Here again, for the  second order master equation, the terminal condition $G$  does not depend on $x_0$, but this dependence is needed in the MFG problem with a major player.\\

{\bf Additional assumptions for the MFG problem with a major player.} This problem involves in addition a Hamiltonian $H^0:\R^{d_0} \times \R^{d_0} \times \Pw \to \R$  and a terminal condition $G^0:\R^{d_0}\times \Pw\to \R$. We assume that the map $H^0$ satisfies the growth property
\be\label{CondH0} 
 \sup_{x_0\in \R^{d_0}, \; m\in \Pw}
 |D_{x_0,p}H^0(x_0,p,m)| + |D^2_{x_0,p}H^0(x_0,p,m)|  \leq C_0(|p|^\gamma+1), 
%  |D_{x_0}H^0(x_0,p,m)|\leq C_0(1+|p|^\gamma)\,,
\ee
for some $\gamma>1$. We also suppose that, for any $R>0$, the quantities 
$$
\|H^0(\cdot_{x_0},\cdot_p, m)\|_{3,4}, \;   \left\|\frac{\delta H^0}{\delta m}(\cdot_{x_0},\cdot_p, m, \cdot_y)\right\|_{2,3,k},
\; \left\|\frac{\delta^2 H^0}{\delta m^2}(\cdot_{x_0},\cdot_p, m, \cdot_y,\cdot_{y'})\right\|_{1,2,k-1,k-1},
$$
and ${\rm Lip}_{0,1,k-2,k-2} (\frac{\delta^2 H}{\delta m^2})$
are bounded for $|p|\leq R$, $m\in \Pw$ and $x_0 \in \R^{d_0}$, for any  $k\in \{2, \dots, n-1\}$.

The initial condition $G^0:\R^{d_0}\times \Pw\to \R$ is assumed to be of class $C^2$ with respect to the measure variable and the quantities 
\begin{align*}
& \|G^0(\cdot, m)\|_{3}, \;   \left\|\frac{\delta G^0}{\delta m}(\cdot, m,\cdot)\right\|_{2,k},
\; \left\|\frac{\delta^2 G^0}{\delta m^2}(\cdot_{x_0}, m, \cdot_y, \cdot_{y'})\right\|_{1,k-1,k-1}, \\ 
& {\rm Lip}_{0,k-2,k-2} (\frac{\delta^2 G^0}{\delta m^2})(\cdot_{x_0}, m, \cdot_y, \cdot_{y'}),
\end{align*}
are supposed to be bounded uniformly with respect to  $m\in \Pw$. \\

Throughout the proofs, we assume that the time horizon $T$ is small, say $T\leq 1$. We denote by $C$ and $C_M$  a constant which might change from line to line and which depends only on the data of the problem, i.e., on $a$, $H$ and $H^0$---the dependence in $G$ and $G^0$ being always explicitly written---and, for $C_M$,  on the additional real number $M$. In some proofs, when there is no ambiguity, we drop the $M$ dependence of $C_M$ to simplify the expressions. 

%%%%%%%%%%%%%%%%%%%%%%%%%%%%%%%%%
%%%%%%%%%%%%%%%%%%%%%%%%%%%%%%%%%
%%%%%%%%%%%%%%%%%%%%%%%%%
\section{The second order master equation}\label{sec.master2}

The aim of the section is to show the short-time existence of the second order master equation: 

\be\label{eq.Master2}
\left\{\begin{array}{l} 
\ds - \partial_t U(t,x,m)  - {\rm Tr}((a(t,x)+a^0)D^2_{xx}U(t,x,m))+H(x,D_xU(t,x,m),m)\\
\ds \qquad -\int_{\R^d}{\rm Tr}((a(t,y)+a^0)D^2_{ym}U(t,x,m,y))\ dm(y) \\
\ds    \qquad    + \int_{\R^d} D_m U(t,x,m,y)\cdot H_p(y,D_xU(t,y,m),m) \ dm(y) \\
\ds \qquad -2\int_{\R^d} {\rm Tr} \left[a^0 D^2_{xm}U(t,x,m,y)\right]m(dy) \\
\ds    \qquad  
-\int_{\R^{2d}} {\rm Tr}[a^0 D^2_{mm}U(t,x,m,y,y')]m(dy)m(dy') = 0 \\
 \ds \qquad \qquad  \qquad {\rm in }\; (0,T)\times \R^d\times \Pw\,,\\
 U(T,x,m)= G(x,m) \qquad  {\rm in }\;  \R^d\times \Pw\;,
\end{array}\right.
\ee
where $a^0$ is a symmetric positive definite $d\times d$ matrix (independent of time and space).

\begin{Definition} We say that $U:[0,T]\times \R^d\times \Pw\to \R$ is a classical solution of \eqref{eq.Master2} if $U$ and its derivatives involved in \eqref{eq.Master2} exist, are continuous in all variables and are bounded, and if \eqref{eq.Master2} holds.
\end{Definition}

Our main result is the following short time existence Theorem: 

\begin{Theorem}\label{theo.ShortTime} Under the assumptions of Subsection \ref{subsec.Hyp}, there exists a time $T>0$ such that the second order master equation \eqref{eq.Master2} has a  classical solution $U$ on $[0,T]$ .
% which is, in addition, such that $D_xU$ is uniformly Lipschitz continuous
%\footnote{ isn't this property already contained in the above definition of "classical solution" ? because "the derivatives involved (including $D_{xx}, D_{xm}$, $D_{ym}$) are continuous in all variables and are bounded"... }.
\end{Theorem}

  We shall not prove here the uniqueness of the solution to \eqref{eq.Master2}, which holds under our assumptions: this point has been often discussed in the literature (see \cite{CDLL, CaDebook} for instance). The reader may notice that we cannot handle a second order master equation with a space dependent matrix $a^0=a^0(t,x)$. The reason is that we do not know how to extend the estimate in Proposition  \ref{prop.linear2} to the space dependent case.

The proof of Theorem \ref{theo.ShortTime} is given at the end of the section, after a few preliminary steps. The key idea is to use a Trotter-Kato scheme alternating  the first order master equation as in \eqref{eq.transport} and a linear second order master equation. The analysis of the first order master equation, being quite technical, is postponed to Section \ref{Sec.FirstOrdreMaster} below. We now discuss the linear second order master equation.

%%%%%%%%%%%%%%%%%%%%%%%%%%%%%
\subsection{The linear second order master equation}

In this section we consider the (forward) second order linear master equation
\be\label{linearMforward}
\left\{\begin{array}{l} 
\ds  \partial_t U(t,x,m) -{\rm Tr}\left[ a^0D^2_{xx} U(t,x,m)\right]
%\\
%\ds \qquad  
-\int_{\R^d} {\rm Tr}\left[a^0 D^2_{ym}U(t,x,m,y)\right] m(dy)\\
\ds \qquad -2\int_{\R^d} {\rm Tr} \left[a^0 D^2_{xm}U(t,x,m,y)\right]m(dy) 
\\
\ds    \qquad  
-\int_{\R^{2d}} {\rm Tr}[a^0 D^2_{mm}U(t,x,m,y,y')]m(dy)m(dy') = 0 \\
%\\
% \ds \qquad \qquad  
 \qquad\qquad\qquad\qquad\qquad\qquad {\rm in }\; (0,T)\times \R^d\times \Pw\;,\\
 U(0,x,m)= G(x,m) \qquad  {\rm in }\;  \R^d\times \Pw\,.
\end{array}\right.
\ee
Let $\Gamma$ be the fundamental solution of the equation associated with $a^0$: 
$$
\left\{\begin{array}{l} 
\partial_t \Gamma(t,x) -{\rm Tr}\left[ a^0D^2_{xx} \Gamma(t,x)\right]= 0 \qquad {\rm in}\; (0,+\infty)\times \R^d,\\
\Gamma(0,x)= \delta_0(x)  \qquad {\rm in}\;  \R^d,
\end{array}\right.
$$
and, given a map $G:\R^d\times \Pw\to \R$ of class $C^2$ in $(x,m)$, let us set
%\footnote{ did we ever define the push-forward of a  measure?}
$$
U(t,x,m)=  \int_{\R^d} G(\xi,(id-x+\xi)\sharp m) \Gamma(t,x-\xi)d\xi \qquad \forall (t,x,m)\in [0,T]\times \R^d\times \Pw.
$$

\begin{Proposition}\label{prop.linear2}
The map $U$ is a classical solution to the second order equation \eqref{linearMforward}. Moreover, there exists a constant $C>0$ (depending only on $n$, $k$ and $a^0$), such that 
$$
\sup_{t\in [0,T]} \|U(t)\|_n \leq (1+CT)\sup_{m\in \Pw}\|G\|_n
$$
and, for $k\in \{2, \dots, n-1\}$, 
$$
\sup_{t\in [0,T]}  \left\|\frac{\delta U}{\delta m}(t)\right\|_{n-1;k} \leq 
(1+CT)  \left\|\frac{\delta G}{\delta m}\right\|_{n-1;k},
$$
$$
\sup_{t\in [0,T]} \left\|\frac{\delta^2 U}{\delta m^2}(t)\right\|_{n-2;k-1,k-1} \leq 
(1+ CT) \left\|\frac{\delta^2 G}{\delta m^2}\right\|_{n-2;k-1,k-1}
$$
and 
$$
\sup_{t\in [0,T]}  {\rm Lip}_{n-3;k-2,k-2} (\frac{\delta^2 U}{\delta m^2}(t))
\leq (1+ CT){\rm Lip}_{n-3;k-2,k-2} (\frac{\delta^2 G}{\delta m^2})\, . 
$$
\end{Proposition}

\begin{Remark}\label{kjhbaezls} If we assume that, for some constant $M$,  
$$
\|G\|_{n}+ \left\|\frac{\delta G}{\delta m}\right\|_{n-1;k}+ \left\|\frac{\delta^2 G}{\delta m^2}\right\|_{n-2;k-1,k-1}+{\rm Lip}_{n-3;k-2,k-2} (\frac{\delta^2 G}{\delta m^2}) \leq M,
$$
then the above estimates can be rewritten in the form:
\begin{align*}
\sup_{t\in [0,T]}(\|U(t)\|_n + \left\|\frac{\delta U}{\delta m}(t)\right\|_{n-1;k}+  \left\|\frac{\delta^2 U}{\delta m^2}(t)\right\|_{n-2;k-1,k-1} 
+ {\rm Lip}_{n-3;k-2,k-2} (\frac{\delta^2 U}{\delta m^2}(t)))\\
\leq \|G\|_n+
\left\|\frac{\delta G}{\delta m}\right\|_{n-1;k}+ \left\|\frac{\delta^2 G}{\delta m^2}\right\|_{n-2;k-1,k-1}+{\rm Lip}_{n-3;k-2,k-2} (\frac{\delta^2 G}{\delta m^2})+C_MT,  
\end{align*}
for some constant $C_M$ depending on  $n$, $k$, $a^0$ and $M$. 
\end{Remark}

In order to prove this Proposition, we need two Lemmas, the proof of which are easy and left to the reader. 

\begin{Lemma}\label{lem1} Let $U:\Pw\to \R$ be $L-C^1$ and let $\phi:\R^d\to \R^d$ be of class $C^1$ with bounded derivative. Let us set $V(m)=U(\phi\sharp m)$. Then $V$ is   $L-C^1$ with
$$
D_mV(m,y)= (D\phi(y))^T D_mU(\phi\sharp m, \phi(y)).
$$
\end{Lemma}

\begin{Lemma}\label{lem2}  Let $U:\Pw\to \R$ be $L-C^1$ and let $V(x,m)= U((id+x)\sharp m)$. Then $V$ is of class $C^1$ with 
$$
D_xV(x,m)=  \int_{\R^d} D_mU((id+x)\sharp m, x+y)dm(y).
$$
\end{Lemma}

\begin{proof}[Proof of Proposition \ref{prop.linear2}] Let us first note that 
$$
U(t,x,m)=  \int_{\R^d} G(\xi,(id-x+\xi)\sharp m) \Gamma(t,x-\xi)d\xi =  \int_{\R^d} G(x-z,(id-z)\sharp m) \Gamma(t,z)dz.  
$$
In particular, $U$ is $C^1$ in $t$, $C^2$ in $x$ and has second order derivatives which are $C^2$ in the space variables with, in view of Lemma  \ref{lem1} and Lemma \ref{lem2}, 
$$
D_x U(t,x,m)=   \int_{\R^d} D_xG(x-y,(id-y)\sharp m) \Gamma(t,y)dy,   
$$
$$ 
D^2_x U(t,x,m)=   \int_{\R^d} D^2_xG(x-y,(id-y)\sharp m) \Gamma(t,y)dy,   
$$
$$
D_mU(t,x,m,y)=  \int_{\R^d} D_mG(x-z, (id-z)\sharp m, y-z)\Gamma (t,z)dz, 
$$
and
$$
D^2_mU(t,x,m,y,y')=  \int_{\R^d} D^2_mG(x-z, (id-z)\sharp m, y-z,y'-z)\Gamma (t,z)dz. 
$$
This easily implies the estimates on $U$ and its derivatives. 

On the other hand, since $(id-w)\sharp [(id-z)\sharp m]= (id-z-w)\sharp m$, we have, for any $t\in (0,T)$ and $h\in (0,T-t)$, 
\begin{align*}
& \int_{\R^d} U(t,x-z,(id-z)\sharp m)\Gamma(h,z)dz \\
& \qquad =   \int_{\R^d} \int_{\R^d}  G(x-z-w,(id-z-w)\sharp m)\Gamma(h,z)\Gamma(t,w)dwdz\\ 
& \qquad =    \int_{\R^d}  G(x-u,(id-u)\sharp m)\left(  \int_{\R^d} \Gamma(h,u-w)\Gamma(t,w)dw\right)du\\
& \qquad  =  \int_{\R^d}  G(x-u,(id-u)\sharp m)\Gamma(t+h,u)du = U(t+h,x,m). 
\end{align*}
So, taking the derivative with respect to $h>0$ in the above expression:
\begin{align*}
& \partial_t U(t+h,x,m) =  \int_{\R^d} U(t,x-z,(id-z)\sharp m)\partial_t \Gamma(h,z)dz. 
\end{align*}
Integrating by parts and using Lemma \ref{lem1} and Lemma \ref{lem2}:
\begin{align*}
& \partial_t U(t+h,x,m)   =  \int_{\R^d} U(t,x-z,(id-z)\sharp m) \Bigl( {\rm Tr}\left[ a^0D^2_{zz} \Gamma(h,z)\right] \Bigr) dz \\ 
& =  \int_{\R^d}  \Bigl( {\rm Tr}\left[ a^0D^2_{xx}U(t,x-z,(id-z)\sharp m)\right] + 2 \int_{\R^d} {\rm Tr}\left[ a^0D^2_{xm}U(t,x-z,(id-z)\sharp m,y-z)\right]m(dy)\\
&  \qquad +  \int_{\R^d} {\rm Tr}\left[ a^0D^2_{ym}U(t,x-z,(id-z)\sharp m,y-z)\right]m(dy) \\
&  \qquad +  \int_{\R^d} \int_{\R^d} {\rm Tr}\left[ a^0D^2_{mm} U(t,x-z,(id-z)\sharp m,y-z,y'-z)\right]m(dy)m(dy') \Bigr) \Gamma(h,z) ) dz . 
\end{align*}
Letting $h\to 0$ we obtain 
\begin{align*}
& \partial_t U(t,x,m) =  {\rm Tr}\left[ a^0D^2_{xx}U(t,x,m)\right] + 2 \int_{\R^d}  {\rm Tr}\left[ a^0D^2_{xm}U(t,x,m,y)\right]m(dy)\\
& \qquad  +  \int_{\R^d} {\rm Tr}\left[ a^0D^2_{ym}U(t,x,m,y)\right]m(dy) 
%\\
%& \qquad \qquad 
+   \int_{\R^d} \int_{\R^d}  {\rm Tr}\left[ a^0D^2_{mm} U(t,x,m,y,y')\right]m(dy)m(dy'). 
\end{align*}
So $U$ is a solution to \eqref{linearMforward}. 
\end{proof}

%%%%%%%%%%%%%%%%%%%%%%%%%%
\subsection{Existence of a solution}

%%%%%%%%%%%%%%%%%%%%%%%%%
\subsubsection{Definition of the semi-discrete scheme}

Let us fix some horizon $T>0$ (small) and a step-size $\tau:= T/(2N)$ (where $N\in \N$, $N\geq 1$). We set $t_k= kT/(2N)$, $k\in \{0, 2N\}$. We define by backward induction a continuous map $U^N=U^N(t,x,m)$, with $U^N:[0,T]\times \R^d\times \Pw\to \R$, as follows: we require that 

\begin{itemize}

\item[(i)] $U^N$ satisfies the terminal condition  
$$
U^N(T,x,m)= G(x,m)\qquad \forall (x,m)\in \R^d\times \Pw, 
$$

\item[(ii)]  $U^N$ solves the backward linear second order master equation 
\be\label{bla}
\begin{array}{l} 
\ds - \partial_t U^N -2{\rm Tr}\left[ a^0D^2_{xx} U^N\right] -2\int_{\R^d} {\rm Tr}\left[a^0 D^2_{ym}U^N \right] m(dy)\\
\ds \qquad -4\int_{\R^d} {\rm Tr} \left[a^0 D^2_{xm}U^N \right]m(dy) 
%\\
%\ds    \qquad  
-2\int_{\R^{2d}} {\rm Tr}[a^0 D^2_{mm}U^N ]m(dy)m(dy') 
 = 0 \\
\end{array}
\ee
on time intervals of the form $(t_{2j+1}, t_{2j+2})$, for $j=0,\ldots,N-1$, 

\item[(iii)] $U^N$ solves the first order master equation 
\be\label{blabla}
\begin{array}{l} 
\ds - \partial_t U^N  - 2{\rm Tr}(aD^2_{xx}U^N) +2H(x,D_xU^N,m) -2\int_{\R^d}{\rm Tr}(aD^2_{ym}U^N)\ dm(y) \\
\ds    \qquad   + 2\int_{\R^d} D_m U^N\cdot H_p(y,D_xU^N,m) \ dm(y) =0 
\end{array}
\ee
on time intervals of the form $(t_{2j}, t_{2j+1})$,  for $j=0,\ldots,N-1$. 
\end{itemize}

Our aim is to show that, if the time horizon is short enough, $U^N$ converges to a solution of the second order master equation as $N\to+\infty$. 

%%%%%%%%%%%%%%%%%%%%%%%%%
\subsubsection{Estimates of $U^N$}

For $n\geq 4$
and $k\in \{3, \dots, n-1\}$, let 
\be\label{eq.defM}
M:= \|G\|_{n}+ \left\|\frac{\delta G}{\delta m}\right\|_{n-1;k}+ \left\|\frac{\delta^2 G}{\delta m^2}\right\|_{n-2;k-1,k-1}+{\rm Lip}_{n-3;k-2,k-2} (\frac{\delta^2 G}{\delta m^2})+1.
\ee

\begin{Lemma}\label{lem.estiUN} There exists $T_{M}>0$ such that, for any $T\in (0,T_{M}]$ and $N\geq 1$, we have 
$$
\sup_{t\in [0,T]}( \|U^N(t)\|_{n}+ \|\frac{\delta U^N}{\delta m}(t)\|_{n-1;k} + \left\|\frac{\delta^2  U^N}{\delta m^2}(t)\right\|_{n-2;k-1,k-1}+{\rm Lip}_{n-3;k-2,k-2} \left(\frac{\delta^2 U^N}{\delta m^2}(t)\right))\leq M.
$$
Moreover: 
\begin{itemize}
\item The maps $U^N$, $D_xU^N$, $D^2_{xx}U^N$ are globally Lipschitz  continuous in $(t,x,m)$, uniformly with respect to $N$.
\item The maps $D_mU$, $D_m D_xU^N$, $D_yD_m U^N$ are Holder continuous in $(t,x,m,y)$, uniformly with respect to $N$, in any set of the form 
\be\label{def.setofformR}
\{(t,x,m,y)\in [0,T]\times \R^d\times \Pw \times \R^d, \; M_2(m)\leq R, \; |y|\leq R\}\,,
\ee
(where $M_2(m)= (\int_{\R^d} |y|^2m(dy))^{1/2}$). 
\item The map $D^2_m U^N$ is Holder continuous in $(t,x,m,y,y')$, uniformly with respect to $N$, in any set of the form 
\be\label{def.setofformR2}
\{(t,x,m,y,y')\in [0,T]\times \R^d\times \Pw \times \R^d\times \R^d, \; M_2(m)\leq R, \; |y|, |y'|\leq R\}\,.
\ee
\end{itemize}
\end{Lemma}

\begin{proof}   In order to prove the estimate, we use Proposition  \ref{prop.linear2} as well as Propositions \ref{Prop.DerivU}, \ref{Prop.DerivU2}, \ref{prop:rlipdelta2U} (in Section \ref{Sec.FirstOrdreMaster} below). Let $T_M$ be the smallest positive constant associated with  these Propositions. Let also  $C_M$ be the largest constant in Propositions  \ref{prop.linear2}, \ref{Prop.DerivU}, \ref{Prop.DerivU2} and \ref{prop:rlipdelta2U}. We assume without loss of generality that $T_M <1/(2C_M)$ and we fix $T\in (0,T_M]$.

%{\color{green} Let $T^0_{M}$ and $L_{M}$ be the constants associated with Proposition \ref{prop.esti1MFG-0}, then we fix $K_M= M+ L_M T^0_M$. Let now $T_M^1, C_M$ be given by Propositions \ref{Prop.DerivU}, \ref{Prop.DerivU2}, \ref{prop:rlipdelta2U} (in Section \ref{Sec.FirstOrdreMaster} below) and Proposition \ref{prop.linear2}, where the regularity of the Hamiltonian $H(x,p,m)$ is used for $|p|\leq K_M$. }

 We define  the sequence $(\theta_{k})_{k=0,\dots,2N}$ by 
$$
 \theta_{2j}= M-1 + C_{M}\frac{T}{N}(N-j)\,, \qquad  j=0,\ldots,N.  
$$
As $T_{M}\leq  1/(2C_M)$, we have $\theta_{2j} \leq M$ for any $T\in (0,T_{M}]$ and $N\geq 1$.

Now, using  Propositions \ref{Prop.DerivU}, \ref{Prop.DerivU2}, \ref{prop:rlipdelta2U} and \ref{prop.linear2} one  checks by backward induction that  
\be\label{induc}
\begin{array}{l}
\ds \sup_{t\in [t_{2j},t_{2j+2}]} \left\{ \|U^N(t)\|_{n}+ \|\frac{\delta U^N}{\delta m}(t)\|_{n-1;k} + \left\|\frac{\delta^2  U^N}{\delta m^2}(t)\right\|_{n-2;k-1,k-1}\right.\\
\qquad \qquad \qquad \qquad  \ds \left.+{\rm Lip}_{n-3;k-2,k-2} \left(\frac{\delta^2 U^N}{\delta m^2}(t)\right)\right\}\leq \theta_{2j} \leq M\qquad \forall j=0,\ldots,N-1. 
%\;\leq \; M.
\end{array}
\ee
Indeed, assume that this is true for $j+1$;  Proposition \ref{prop.linear2} (see also Remark \ref{kjhbaezls}), applied in the interval $[t_{2j+1},t_{2j+2}]$ and with the terminal condition $U^N(t_{2j+2}, \cdot, \cdot)$ which satisfies \eqref{induc} by assumption,  implies that
$$
\begin{array}{l}
\ds \sup_{t\in [t_{2j+1},t_{2j+2}]} \left\{ \|U^N(t)\|_{n}+ \|\frac{\delta U^N}{\delta m}(t)\|_{n-1;k} + \left\|\frac{\delta^2  U^N}{\delta m^2}(t)\right\|_{n-2;k-1,k-1}\right.\\
\m
\qquad \qquad \qquad \qquad \qquad \ds \left.+{\rm Lip}_{n-3;k-2,k-2} \left(\frac{\delta^2 U^N}{\delta m^2}(t)\right)\right\}\leq \theta_{2j+2}+\frac {C_M T}{2N}\,.
\end{array}
$$
Then using Propositions \ref{Prop.DerivU}, \ref{Prop.DerivU2}, \ref{prop:rlipdelta2U}  for the interval $[t_{2j},t_{2j+1}]$ and the terminal condition $U^N(t_{2j+1}, \cdot, \cdot)$ for which  \eqref{induc} now holds, one gets:
$$
\begin{array}{l}
\ds \sup_{t\in [t_{2j},t_{2j+1}]} \left\{ \|U^N(t)\|_{n}+ \|\frac{\delta U^N}{\delta m}(t)\|_{n-1;k} + \left\|\frac{\delta^2  U^N}{\delta m^2}(t)\right\|_{n-2;k-1,k-1}\right.\\
\m
\qquad \qquad \qquad\qquad \qquad  \qquad \qquad \ds \left.+{\rm Lip}_{n-3;k-2,k-2} \left(\frac{\delta^2 U^N}{\delta m^2}(t)\right)\right\}\leq  \theta_{2j+2}+ \frac {C_M T}{N}
= \theta_{2j}\,,
\end{array}
$$
so \eqref{induc} holds for $j$.
Since the first step ($j=N-1$) can be proved similarly using the very definition of $M$ in \eqref{eq.defM}, we can conclude that \eqref{induc} holds for every $j= 0, \dots,N-1$.

We now prove the second part of the Lemma. As $U^N$ solves \eqref{bla} on the time intervals $(t_{2j+1}, t_{2j+2})$ and \eqref{blabla} on time intervals $(t_{2j},t_{2j+1})$,  we obtain directly,  by the space estimates proved above:
\be\label{ineq.partialtUN}
\sup_{t,m}\left\| \partial_t U(t,\cdot,m) \right\|_{n-2}\leq C_M,
\ee
where $C_M$ does not depend on $N$. 

Let now $l\in \N^d$ with $|l|\leq 2$. By \eqref{ineq.partialtUN} and the fact that $\|U^N\|_n$ is bounded for $n>|l|$,  $D^lU^N$ is uniformly Lipschitz continuous in $t$ and $x$. Moreover, since $\|\delta U^N/\delta m\|_{n-1;k}$ is bounded (for $k\geq 1$), $D^l U^N$ is uniformly Lipschitz continuous in $m$ as well by Remark \ref{rem.k;k} since  $|l|\leq n-1$.

Next we prove the uniform continuity of $D^l_xD^r_y D_mU^N$ for $|l|,|r|\leq 1$. First we recall that $\|\delta U^N/\delta m\|_{n-1;k}$ is bounded, so that $\|D_m U^N\|_{n-1;k-1}$ is bounded, with $n-1\geq 2$ and $k- 1\geq 2$.
 Therefore $D^l_xD^r_yD_m U^N$ is uniformly Lipschitz continuous in $(x,y)$ (for $y$, this is Remark \ref{rem.k;k}). Second, recall that $\|\delta^2U^N/\delta m^2\|_{n-2;k-1,k-1}$ is bounded,  so that $\|\frac{\delta}{\delta m} D_mU^N\|_{n-2;k-2,k-1}$ is bounded as well, 
 with $n\geq 3$ and $k\geq 3$: therefore $D^l_xD^r_y D_mU^N$ is uniformly Lipschitz continuous in $m$. As we have already proved that $U^N$ is uniformly Lipschitz continuous in $t$, we can deduce from Lemma \ref{lem.interp} below applied to $U^N$ that $D_mU^N$ is also Holder continuous in time in any set of the form \eqref{def.setofformR}. 

Finally we consider $D^2_{mm}U^N= D^2_{mm}U^N(t,x,m,y,y')$. Since $\|\delta^2U^N/\delta m\|_{n-2;k-1,k-1}$ and\\ ${\rm Lip}_{n-3;k-2,k-2} \left(\frac{\delta^2 U^N}{\delta m^2}\right)$ are bounded, with $n\geq 4$ and $k\geq 3$, $D^2_{mm}U^N$ is uniformly Lipschitz continuous in $(x,m,y,y')$. Applying Lemma \ref{lem.interp} to the map $D_mU^N$, which is Holder continuous in time in sets of the form \eqref{def.setofformR} (as we have seen above) and such that $D^2_{mm}U$ is uniformly Lipschitz in $(m,y,y')$, we deduce that  $D^2_{mm}U^N$ is also Holder continuous in time, uniformly in $N$, in sets of the form \eqref{def.setofformR2}. So we conclude that $D^2_{mm}U^N$ is uniformly Holder continuous in all variables.
\end{proof}

%%%%%%%%%%%%%%%%%%%%%%%%%
\subsubsection{Proof of Theorem \ref{theo.ShortTime}}

\begin{proof}[Proof of Theorem \ref{theo.ShortTime}]
In view of Lemma \ref{lem.estiUN}, the maps $U^N$, $D_xU^N$, $D^2_{xx}U^N$, $D_mU^N$, $D_m D_xU^N$, $D_yD_m U^N$ and  $D^2_m U^N$ are locally Holder continuous in all variables, uniformly with respect to $N$. So, by a version of Arzela-Ascoli Theorem (see Lemma \ref{lem.Arzela-Ascoli} below), there is a subsequence denoted in the same way such that $U^N$, $D_xU^N$, $D^2_{xx}U^N$, $D_mU^N$,  $D_m D_xU^N$, $D_yD_m U^N$ and  $D^2_m U^N$ converge pointwisely in $m$ and locally uniformly in time-space to some maps $U$, $D_xU$, $D^2_{xx}U$, $V$, $D_x V$, $D_yV$ and  $W$. Moreover, using the integral formula \eqref{zalekrnjdg}, is  easy to check that $V=D_mU$ and $W=D^2_mU$. 

By the equation satisfied by $U^N$, we have, for any $0\leq s<t\leq T$, 
\begin{align*}
& U^N(t,x,m)-U^N(s,x,m)  \\
& \qquad = -\sum_{k=0}^{N-1}\int_{t_{2k+1}}^{t_{2k+2}} 2 \Bigl\{ \ds  {\rm Tr}\left[ a^0D^2_{xx} U^N\right] + \int_{\R^d} {\rm Tr}\left[a^0 D^2_{ym}U^N\right] m(dy)
\\
& \qquad \qquad \ds  
+2\int_{\R^d} {\rm Tr} \left[a^0 D^2_{xm}U^N\right] m(dy)
+ \int_{\R^{2d}} {\rm Tr}\left[a^0 D^2_{mm}U^N\right]  m(dy)m(dy')  
\Bigr\}{\bf 1}_{[s,t]}(\tau)d\tau\\
 &\qquad \qquad   -\sum_{k=0}^{N-1}\int_{t_{2k}}^{t_{2k+1}}  2\Bigl\{ 
  {\rm Tr}(aD^2_{xx}U^N) -  H(x,D_xU^N,m) \\
 &  \qquad  \qquad\qquad+ \int_{\R^d}{\rm Tr}(aD^2_{ym}U^N)\  m(dy)  -  \int_{\R^d} D_m U^N\cdot H_p(y,D_xU^N,m) \  m(dy) \Bigr\}{\bf 1}_{[s,t]}(\tau)d\tau.
\end{align*}
Since, as $N$ tends to infinity, the maps 
$$
t \to \sum_{k=0}^{N-1} {\bf 1}_{[t_{2k+1},t_{2k+2}]}(t)\qquad {\rm and }\qquad t\to  \sum_{k=0}^{N-1} {\bf 1}_{[t_{2k},t_{2k+1}]}(t)
$$
weakly converge to the constant  $1/2$ and since the space integrals in the above equation converge pointwisely to the corresponding quantities for the limit $U$, we obtain by the dominated convergence Theorem:
\begin{align*}
& U(t,x,m)-U(s,x,m)  \\
& \qquad = - \int_{s}^{t}   \Bigl(\ds {\rm Tr}\left[ a^0D^2_{xx} U\right] +\int_{\R^d} {\rm Tr}\left[a^0 D^2_{ym}U\right] dm
\\
& \qquad \qquad \ds  
+2\int_{\R^d} {\rm Tr} \left[a^0 D^2_{xm}U\right]dm
+\int_{\R^{2d}} {\rm Tr}\left[a^0 D^2_{mm}U\right] dm\otimes dm \\
 &\qquad \qquad   +
 {\rm Tr}(a D^2_{xx}U) - H(x,D_xU,m) \\
 &  \qquad  \qquad+\int_{\R^d}{\rm Tr}(aD^2_{ym}U)\ dm   - \int_{\R^d} D_m U\cdot H_p(y,D_xU,m) \ dm \Bigr)
 d\tau,
\end{align*}
so that $U$ is a classical solution to \eqref{eq.Master2}. 
\end{proof}

\subsection{Existence of the solution to the stochastic MFG system}

An easy consequence of the existence of a solution to the master equation is the well-posedness of the stochastic MFG system: 
\be\label{eq.MFGs}
\left\{ \begin{array}{l}
\ds 
 d u(t,x) = \bigl[ -{\rm Tr}((a+a^0)(t,x)D^2u(t,x))+ H(x,Du(t,x),m(t)) \\
 \ds \qquad \qquad \qquad\qquad -  \sqrt{2} {\rm Tr}(\sigma^0  Dv(t,x)) \bigr] dt + v(t,x) \cdot dW_{t} \qquad \qquad {\rm in }\; (0,T)\times \R^d,\\
\ds d m(t,x) = \bigl[ \sum_{i,j} D_{ij}(( (a_{ij})+a^0_{ij})(t,x)m(t,x)) + {\rm div} \bigl( m(,x) H_p(x,D u(t,x),m(t))\bigr) \bigr] dt\\
 \qquad \qquad \qquad \qquad \ds  - {\rm div} ( m(t,x) \sqrt{2}\sigma^0  dW_{t} \bigr), 
 \qquad \qquad {\rm in }\; (0,T)\times \R^d,
 \\
 \ds u(T,x)=G(x,m(T)), \; m(0)=m_0,\qquad {\rm in }\;  \R^d
\end{array}\right.
\ee
We say that $(u,m,v)$ is a classical solution to \eqref{eq.MFGs} if $u$, $m$ and $v$ are random with values in $C^0([0,T], C^2_b)$, $C^0([0,T],\Pw)$ and $C^0([0,T], C^1_b(\R^d, \R^d))$ respectively and  adapted to the filtration generated by $W$ and if
the backward HJ equation is satisfied in a classical sense: 
\begin{align*}
u(t,x) & = G(x,m(T))-\int_t^T \bigl( -{\rm Tr}((a+a^0)(s,x)D^2u(s,x))+ H(x,Du(s,x),m(s))\\ 
& \qquad\qquad\qquad\qquad  -  \sqrt{2} {\rm Tr}(\sigma^0  Dv(s,x)) \bigr) ds-\int_t^T v(s,x) \cdot dW_{s} 
\end{align*}
while the Fokker-Planck equation is satisfied in the sense of distributions: for any $\phi\in C^\infty_c([0,T)\times \R^d)$, 
\begin{align*}
& 0= \int_{\R^d} \phi(0,x)m_{0}(dx)\\
& \qquad + \int_0^T \int_{\R^d} \Bigl( {\rm Tr}((a+a^0)(s,x)D^2\phi(s,x)-D\phi(s,x)\cdot H_p(x,D u(s,x),m(s))\Bigr)m(s,dx)ds\\
& \qquad  + \sqrt{2}  \int_0^T \int_{\R^d}  ((\sigma^0)^TD\phi(s,x)m(s,dx) \cdot dW_s.
\end{align*}

\begin{Theorem} Under the assumptions of Theorem \ref{theo.ShortTime},  there exists a time $T>0$ for which the stochastic MFG system \eqref{eq.MFGs} has a classical solution $(u,m,v)$ in $[0,T]$. Moreover,  
\be\label{eq.defv}
v(t,x)= \sqrt{2}\int_{\R^d} (\sigma^0)^T D_mU(t,x,m(t),y)m(t,dy),
\ee
where $U$ is the solution to the second order master equation \eqref{eq.Master2}. 
\end{Theorem}

\begin{proof} Let $m$ be the solution to the stochastic McKean-Vlasov equation: 
\be\label{stochastic McKean-Vlasov equation}
\left\{ \begin{array}{l}
\ds d m(t,x) = \bigl[ \sum_{i,j} D_{ij}((a_{i,j}+a^0_{i,j})(t,x) m(t,x)) + {\rm div} \bigl( m(t,x) H_p(x,D U(t,x,m(t)),m(t)) 
\bigr) \bigr] dt\\ 
 \qquad \qquad \qquad \qquad \ds  - {\rm div} ( m(t,x) \sqrt{2}\sigma^0  dW_{t} \bigr), 
  \qquad \qquad {\rm in }\; (0,T)\times \R^d, \\
 \ds m(0,dx)=m_{0},\qquad {\rm in }\;  \R^d
\end{array}\right.
\ee
Existence of a solution for this system can be obtained, for instance, as the mean field limit of the SDE
$$
\left\{ \begin{array}{l}
dX^{N,i}_s= -H_p(X^{N,i}_s,D_x U(t,X^{N,i}_s,m^N_{{\bf X}^N_s}),m^N_{{\bf X}^N_s})ds +\sqrt{2} \sigma(s,X^{N,i}_s)dB^i_s +  \sqrt{2}
\sigma^0(s,X^{N,i}_s) dW_s \\
X^{N,i}_0= \bar X^{N,i}_0
\end{array}\right.
$$
where $\bar X^{N,i}_0$ is a family of i.i.d. r.v. of law $m_0$ and where $\ds m^N_{{\bf X}^N_s}=\frac{1}{N} \sum_{i=1}^N \delta_{X^{N,i}_s}$. Indeed, one can show that the family of laws of $(m^N_{{\bf X}^N_s})$ is tight in $C^0([0,T], \Pw)$ and that its limit is a solution to \eqref{stochastic McKean-Vlasov equation}. Uniqueness for \eqref{stochastic McKean-Vlasov equation} comes from the regularity of $U$ and Gronwall's Lemma. 

Then one can use the Itô's formula  in \cite[Theorem A.1]{CDLL} (see also \cite[Theorem 11.13]{CaDebook}) to derive that $u(t,x):= U(t,x,m(t))$ solves the backward stochastic HJ equation
$$
\left\{ \begin{array}{l}
\ds 
 d u(t,x) = \bigl[ -{\rm Tr}((a+a^0)(t,x)D^2u(t,x))+ H(x,Du(t,x),m(t)) \\
 \ds \qquad \qquad \qquad\qquad -  \sqrt{2} {\rm Tr}(\sigma^0  Dv(t,x)) \bigr] dt + v(t,x) \cdot dW_{t} \qquad \qquad {\rm in }\; (0,T)\times \R^d,\\
 \ds u(T,x)=G(x,m(T)) \qquad {\rm in }\;  \R^d
\end{array}\right.
$$
where $v$ is given by \eqref{eq.defv}. Note that, by the regularity of $U$, $u$ and $v$ have the required regularity. 
\end{proof}

%%%%%%%%%%%%%%%%%%%%%%%%%%%%%%%%%%%%%%%%%%%
%%%%%%%%%%%%%%%%%%%%%%%%%%%%%%%%%%%%%%%%%%%
%%%%%%%%%%%%%%%%%%%%%%%%%%%%%%%%%%%%%%%%%%%
%%%%%%%%%%%%%%%%%%%%%%%%%%%%%%%%%%%%%%%%%%%

\section{The  master equation for MFGs with a major player}\label{sec.Mm}

In this section we investigate the well-posedness of the master equation associated with the MFG problem with a major player. The unknown  $(U^0, U)$ solves the system of master equations:  
\be\label{eq.MasterMM}
\left\{\begin{array}{rl}
\ds (i)& \ds -\partial_t U^{0}(t,x_0,m)-\Delta_{x_0} U^0(t,x_0,m)+ H^0(x_0,D_{x_0}U^{0}(t,x_0,m), m) \\
\ds & \ds \qquad -\int_{\R^d} \dive_yD_m U^0(t,x_0,m,y)dm(y)\\
&\ds \qquad + \int_{\R^d} D_mU^0(t,x_0,m,y)\cdot H_p(x_0,y, D_{x}U(t,x_0,y,m), m)dm(y)=0 \\ 
&\ds \hspace{8cm} {\rm in} \; (0,T)\times \R^{d_0}\times \Pw,\\
\ds (ii) & \ds -\partial_t U(t,x_0,x,m)- \Delta_{x} U(t,x_0,x,m)-\Delta_{x_0}U(t,x_0,x,m)\\
&\ds \qquad  + H(x_0,x,D_{x}U(t,x_0,x,m), m)\\
&\ds \qquad  -\int_{\R^d} \dive_y D_mU(t,x_0,x,m,y)dm(y)
 + D_{x_0}U \cdot H^0_p(x_0,D_{x_0}U^0(t,x_0,m), m)\\
&\ds \qquad   + \int_{\R^d} D_m U(t,x_0,x,m,y)\cdot H_p(x_0,y, D_{x}U(t,x_0,y,m), m)dm(y)=0 \\ 
&\ds \hspace{8cm} {\rm in} \; (0,T)\times \R^{d_0}\times \R^d\times \Pw,\\
\ds (iii) & \ds U^{0}(T,x_0,m)= G^0(x_0,m), \qquad  {\rm in} \; \R^{d_0}\times \Pw,\\
\ds (iv) & \ds U(T,x_0, x,m)= G(x_0,x,m)\qquad  {\rm in} \;  \R^{d_0}\times \R^d\times \Pw.\\
\end{array}\right.
\ee

\begin{Definition} Let  $U^0:[0,T]\times \R^{d_0}\times \Pw\to \R$ and $U:[0,T]\times\R^{d_0}\times \R^d\times \Pw\to \R$. We say that $(U^0,U)$ is a classical solution of \eqref{eq.MasterMM} if $U^0$ and $U$ and their derivatives involved in \eqref{eq.MasterMM} exist, are continuous in all variables and are bounded, and if \eqref{eq.MasterMM} holds. 
\end{Definition}

Throughout this part, assumptions in Subsection \ref{subsec.Hyp} are in force.  Our main result is the following: 

\begin{Theorem}\label{thm.Mm} Under the assumptions of Subsection \ref{subsec.Hyp}, there exists a time $T>0$ and a classical solution $(U^0, U)$ to \eqref{eq.MasterMM} on the time interval $[0,T]$, which is, in addition, such that $D_{x_0}U^0$ and $D_{x_0,x}U$ are uniformly Lipschitz continuous in the space and measure variables.
\end{Theorem}

The result can be easily extended to non constant diffusions. We work here with a constant diffusion to simplify the notation. \\ 

The idea of the proof follows a  similar splitting method as we did in Section \ref{sec.master2}, by dividing the time interval $[0,T]$ into $[t_{2k},t_{2k+1})$ and $[t_{2k+1}, t_{2k+2})$,   where $t_k= kT/(2N)$, $k\in \{0, 2N\}$. This time we alternate the  two following 
problems: in  $[t_{2k+1},t_{2k+2})$  we solve, for a fixed $x_0\in \R^{d_0}$,  the first order system of master equations in $\R^d\times\Pw$:
\be\label{syst.1}
\left\{\begin{array}{rl}
\ds (i)& \ds -\partial_t U^{0} -2\int_{\R^d} \dive_yD_m U^0(t,x_0,m,y)dm(y)\\
&\ds \qquad +2\int_{\R^d} D_mU^0(t,x_0,m,y)\cdot H_p(x_0,y, D_{x}U(t,x_0,y,m), m)dm(y)=0\,, \\ 
%&\ds \hspace{8cm} {\rm in} \; (0,T)\times \Pw \\
\ds (ii) & \ds -\partial_t U- 2\Delta_{x} U+ 2H(x_0,x,D_{x}U, m) -2\int_{\R^d} \dive_y D_mU(t,x_0,x,m,y)dm(y)\\
&\ds \qquad   + 2\int_{\R^d} D_m U(t,x_0,x,m,y)\cdot H_p(x_0,y, D_{x}U(t,x_0,y,m), m)dm(y)=0\,, \\ 
%&\ds \hspace{8cm} {\rm in} \; (0,T)\times \R^d\times \Pw
\end{array}\right.
\ee
while on  $[t_{2k}, t_{2k+1})$ we solve for a fixed $(x,m)\in \R^d\times \Pw$ the system of HJ equations in $\R^{d_0}$:
\be\label{syst.2}
\left\{\begin{array}{rl}
\ds (i)& \ds -\partial_t U^{0}-2\Delta_{x_0} U^0+ 2H^0(x_0,D_{x_0}U^{0}, m)  =0 \,, \\ 
%&\ds \hspace{8cm} {\rm in} \; (0,T)\times \R^d\\
\ds (ii) & \ds -\partial_t U- 2\Delta_{x_0}U+ 2D_{x_0}U \cdot H^0_p(x_0,D_{x_0}U^0(t,x_0,m), m)=0\,. \\ 
%&\ds \hspace{8cm} {\rm in} \; (0,T)\times \R^d
\end{array}\right.
\ee

The (technical) analysis of System \eqref{syst.1} is postponed to Section \ref{Sec.FirstOrdreMaster}. 
We now concentrate on System \eqref{syst.2}. In order to  write the estimate, we need to treat the pair of maps $(U^0,U)$ simultaneously: this requires specific notation that we discuss first.

%%%%%%%%%%%%%%%%%%%%%
\subsection{Analysis of the simple system of HJ equations}\label{s:sshj}

In this section we consider the system 
\be\label{syst.2bis}
\left\{\begin{array}{cl}
\ds (i)& \ds -\partial_t U^{0}(t,x_0;m)-\Delta_{x_0} U^0(t,x_0;m)+ H^0(x_0,D_{x_0}U^{0}(t,x_0;m), m)  =0  
%\\ 
%&\ds \hspace{8cm} 
\; {\rm in} \; (0,T)\times \R^{d_0}\\
\ds (ii) & \ds -\partial_t U(t,x_0;x,m)- \Delta_{x_0}U(t,x_0;x,m)
\\ 
&\ds 
%\hspace{8cm} 
\qquad + D_{x_0}U(t,x_0;x,m) \cdot H^0_p(x_0,D_{x_0}U^0(t,x_0,m), m)=0\qquad  {\rm in} \; (0,T)\times \R^{d_0} \\
\ds (iii) & \ds U^0(T,x_0;m)= G^0(x_0,m) \; {\rm in} \; \R^{d_0},\; U(T,x_0; x,m)= G(x_0,x,m) \; {\rm in} \; \R^{d_0}, 
\end{array}\right.
\ee
where $(x,m)\in \R^d\times \Pw$ are fixed. 
The main part of this subsection consists in proving estimates on the solution $(U^0,U)$ to \eqref{syst.2bis}.

\subsubsection{Notation for the norms}

In this section, we are dealing with pairs of maps $(V^0,V)=(V^0(x_0,m),V(x_0, x,m))$ which might also depend on time $t$, not indicated here. The way we compute the norms is crucial in order to match all the estimates. We use the following norms:
\begin{align*}
\left\| (V^0,V)\right\|_n :=  &  \sup_{m \in \Pw}  \ \sum_{r=0}^n \ \sup_{x_0\in \R^{d_0}, x \in \R^{d}} \left(|V^0(x_0,m)|^2+ |D^r_x V(x_0,x,m)|^2\right)^{1/2}, \\
\left\|\frac{\delta (V^0,V)}{\delta m }\right\|_{n;k} := & \sup_{m \in \Pw} \ \sum_{r=0}^n  \ \sup_{ \substack{ x_0\in \R^{d_0}, x \in \R^{d}, \\ \rho\in C^0_b, \|\rho\|_{-k} = 1} }\left( \left|\frac{ \delta V^0}{\delta m} (x_0,m)(\rho) \right|^2 +  \left|D^r_x\frac{ \delta V}{\delta m}(x_0,x,m)(\rho) \right|^2 \right)^{1/2}, \\
\left\|\frac{\delta^2 (V^0,V)}{\delta m^2 }\right\|_{n;k,k} :=  & \sup_{m \in \Pw} \ \sum_{r=0}^n \ \sup_{ \substack{ x_0\in \R^{d_0}, x \in \R^{d}, \\ \rho,\rho'\in C^0_b, \|\rho\|_{-k} = \|\rho'\|_{-k} = 1 } } \left( \left|\frac{ \delta^2 V^0}{\delta m^2} (x_0,m)(\rho, \rho') \right|^2 +  \left| D^r_x \frac{ \delta^2 V}{\delta m^2} (x_0,x,m)(\rho,\rho') \right|^2 \right)^{1/2}
\end{align*}
and 
\begin{align*}
&  {\rm Lip}_{n;k,k} \Big(\frac{\delta^2 (V^0,V)}{\delta m^2}\Big) := \sup_{m_1\neq m_2} \dw(m_1,m_2)^{-1} \left\|\frac{\delta^2}{\delta m^2 } \big(V^0(m_2)-V^0(m_1),V(m_2)-V(m_1)\big)\right\|_{n;k,k} \\ 
& \qquad= \sup_{m_1\neq m_2} \dw(m_1,m_2)^{-1}  \sum_{r=0}^n  
\sup_{\scriptsize \begin{array}{c} x_0\in \R^{d_0}, x \in \R^{d},\\ \rho,\rho'\in C^0_b, \|\rho\|_{-k} = \|\rho'\|_{-k} = 1 \end{array}  }\\
& \; \left( \left| \frac{\delta^2 V^0}{\delta m^2}(x_0,m_2)(\rho, \rho')-\frac{\delta^2 V^0}{\delta m^2}(x_0,m_1)(\rho, \rho') \right|^2+
\left| D^r_x \frac{\delta^2 V}{\delta m^2}(x_0,x,m_2)(\rho, \rho')- D^r_x \frac{\delta^2 V}{\delta m^2}(x_0,x,m_1)(\rho, \rho') \right|^2\right)^{1/2}. 
\end{align*}
We define in a similar way the quantities ${\rm Lip}_{n}^{x_0} (D^2_{x_0} V^0, D^2_{x_0} V) $, 
${\rm Lip}_{n;k} (\frac{\delta V^0_{x_0}}{\delta m}, \frac{\delta V_{x_0}}{\delta m})$  and ${\rm Lip}_{n} (D^2_{x_0}V^0, D^2_{x_0} U). $

Note that arguing as in Remark \ref{rem.k;k}, a control on $\left\|\frac{\delta (V^0,V)}{\delta m }\right\|_{n;k}$ yields a control on $\left\|\frac{\delta V^0}{\delta m }\right\|_{n, k-1}$ and $\left\|\frac{\delta V}{\delta m }\right\|_{n, k-1}$, and similarly for $\left\|\frac{\delta^2 (V^0,V)}{\delta m^2 }\right\|_{n;k,k}$, ${\rm Lip}_{n;k,k} \Big(\frac{\delta^2 (V^0,V)}{\delta m^2}\Big)$, ...

%%%%%%%%%%%%%%%%%%%%%%%%
\subsubsection{Basic regularity of $(U^0,U)$} 

 We recall that $H^0$ satisfies the assumptions of Section \ref{subsec.Hyp},  in particular  condition \eqref{CondH0} is in force.

\begin{Proposition}\label{prop.p0} Fix $M>0$ and $n\geq 3$. There are constants $K_M,T_M>0$, depending on $M$, $C_0$ and $\gamma$,  and a constant $C_M>0$ depending on 
$$
{  \sup_{|p|\leq K_M} \sup_{m\in \Pw} \sum_{k=0}^3 \|D^k_{(x_0,p)}H^0(\cdot,p,m)\|_\infty} + \sum_{k=0}^3 \|D^k_{(x_0,p)}H_p^0(\cdot,p,m)\|_\infty,
$$
 such that, if $ \sum_{k = 0}^2 \big(\|D^k G^0\|_\infty+ \|D_{x_0}^k G\|_{0,n-k} \big)+ \Bigl({\rm Lip}_{n-3}^{x_0} (D^2_{x_0} G^0, D^2_{x_0} G)    \Bigr) \leq M$, then, for any $T\in (0,T_M)$, we have
\begin{align*}
& \sup_{t} \Bigl(\Bigl\|(U^0,U)(t)\Bigr\|_n +  \Bigl\|D_{x_0}(U^0,U)(t)\Bigr\|_{n-1} 
+  \Bigl\|D^2_{x_0}(U^0,U)(t)\Bigr\|_{n-2} 
+({\rm Lip}_{n-3}^{x_0} (D^2_{x_0} (U^0,  U)(t)))  \Bigr) \\
& \qquad \leq  
\Bigl\|(G^0,G)\Bigr\|_n + \Bigl\|D_{x_0}(G^0,G)\Bigr\|_{n-1} 
+ \Bigl\|D^2_{x_0}(G^0,G))\Bigr\|_{n-2} 
+ ({\rm Lip}_{n-3}^{x_0} (D^2_{x_0} (G^0,G)))
+ C_MT. 
\end{align*}
\end{Proposition}

\begin{proof}  To estimate $\|(U^0,U)\|_n$ it suffices to apply successively Proposition \ref{prop.HighSyst} with $r=0$ and $l \le n$, and to sum over $l$. The argument to estimate first and higher order derivatives with respect to $x_0$  is identical: apply successively Proposition \ref{prop.HighSyst} with $r=1$ and $l \le n-1$ (for $\|(D_{x_0}U^0,D_{x_0}U)\|_{n-1}$), with $r=2$ and $l \le n-2$ (for $\|(D^2_{x_0}U^0,D^2_{x_0}U)\|_{n-2}$) and finally with $r=3$ and $l \le n-3$ (for the Lipschitz bound in $x_0$ of $D^2_{x_0}(U^0,U)$). 
\end{proof}

%%%%%%%%%%%%%%%%%%%%%%%%%%%%%%%%
\subsubsection{First order differentiability in $m$}

\begin{Proposition}\label{Prop.1Mm} Under the assumptions of Proposition \ref{prop.p0}, the pair $(U^0,U)$ is of class $C^1$ with respect to $m$, as well as its derivatives with respect to $x$ appearing below, and, for any fixed  $(x,m,\rho)\in \R^d\times \Pw\times C^{-k}$ the derivative 
$$
(v^0,v)= \Big(\frac{\delta U^0}{\delta m} (t,x_0;m) (\rho), \ \frac{\delta U}{\delta m}(t,x_0;x,m)(\rho) \Big)
$$ 
solves,  
\be\label{Prop.derivM}
\left\{\begin{array}{rl}
\ds (i)& \ds -\partial_t v^{0}-\Delta_{x_0} v^0+  \frac{\delta H^0}{\delta m}(x_0,D_{x_0}U^{0}, m)(\rho)  \\
&\ds \qquad \qquad\qquad\qquad + H^0_p(x_0,D_{x_0}U^{0}, m)\cdot D_{x_0}v^0 =0  \qquad  {\rm in} \; (0,T)\times \R^{d_0},\\
\ds (ii) & \ds -\partial_t v- \Delta_{x_0}v+ D_{x_0}v \cdot H^0_p(x_0,D_{x_0}U^0, m)\\
&\ds \   +D_{x_0}U \cdot \left(  \frac{\delta  H^0_p}{\delta m}(x_0,D_{x_0}U^0, m)(\rho)  +  H^0_{pp}(x_0,D_{x_0}U^0, m)D_{x_0}v^0\right)=0 
\ {\rm in} \; (0,T)\times \R^{d_0}, \\
\ds (iii) & \ds v^0(T,x_0; m)= \frac{\delta G^0}{\delta m}(x_0,m)(\rho) , \;  v(T,x_0 ,x;m)= \frac{\delta G}{\delta m}(x_0,x,m)(\rho) \qquad {\rm in}\; \R^{d_0}.
\end{array}\right.
\ee
Suppose in addition that, for $k \ge 2$,
\[
\sup_{x_0,m} \Bigl(\|\frac{\delta G^0}{\delta m}\|_{k}+ \|\frac{\delta G}{\delta m}\|_{n-1;k} +
\|\frac{\delta G^0_{x_0}}{\delta m}\|_{k-1}+ \|\frac{\delta G_{x_0}}{\delta m}\|_{n-2;k-1} +
\big( {\rm Lip}^{x_0}_{n-3;k-2} (\frac{\delta G^0_{x_0}}{\delta m}, \frac{\delta G_{x_0}}{\delta m})  \big) \Bigr) \le M.
\]
Then there exists $T_M,C_M>0$ such that, for any $T\in (0, T_M)$, we have
\begin{align*}
&
\sup_{t} \Bigl( \Bigl\|\frac{\delta (U^0,U)}{\delta m}(t)\Bigr\|_{n-1;k}+ 
 \Bigl\|\frac{\delta (U^0_{x_0},U_{x_0})}{\delta m}(t)\Bigr\|_{n-2;k-1}
+ ({\rm Lip}^{x_0}_{n-3;k-2} (\frac{\delta U^0_{x_0}}{\delta m}, \frac{\delta U_{x_0}}{\delta m})(t)) \Bigr) \\
&\qquad   \le   \Bigl( \Big\|\frac{\delta (G^0,G)}{\delta m}\Bigr\|_{n-1;k}
+  \Bigl\|\frac{\delta (G^0_{x_0},G_{x_0})}{\delta m}\Bigr\|_{n-2;k-1} 
+ {\rm Lip}^{x_0}_{n-3;k-2} (\frac{\delta G^0_{x_0}}{\delta m}, \frac{\delta G_{x_0}}{\delta m})  \Bigr) + C_MT ,
\end{align*}
where $C_M$ depends on $M$, $r$, $n$, $k$ and on the regularity of $H^0$. 
\end{Proposition}

\begin{proof} 
In order to show that $U^0$ is $C^1$ with respect to $m$, let us define 
$$
\hat U^0(t,x_0; s,m,y):=U^0(t,x_0, (1-s)m+s\delta_y).
$$ 
Then, as $\hat H^0:=H^0(x_0,p, (1-s)m+s\delta_y)$ and $\hat g^0:= G^0(x_0, (1-s)m+s\delta_y)$ are of class $C^1$ with respect to the parameter $s\in[0,1]$, the map $\hat U^0$ is $C^1$ in $s$ and  its derivative $\hat v^0(t,x_0;m,y):= (d\hat U^0/ds)(t,x_0; 0,m,y)$ solves the linearized equation
$$
\left\{\begin{array}{l}
\ds \ds -\partial_t \hat v^{0}-\Delta_{x_0} \hat v^0+ \frac{\delta H^0}{\delta m}(x_0,D_{x_0}U^{0}, m,y) 
+ H^0_p(x_0,D_{x_0}U^{0}, m)\cdot D_{x_0}\hat v^0 =0  \;  {\rm in} \; (0,T)\times \R^{d_0},\\
\ds \hat v^0(T,x_0)= \frac{\delta G^0}{\delta m}(x_0,m,y)\qquad {\rm in}\; \R^{d_0}.
\end{array}\right.
$$
By uniqueness and parabolic regularity, the solution to this equation depends continuously of the parameters $(m,y)$. 
Hence Lemma \ref{lem.condC1} states that $U^0$ is $C^1$ in $m$ with $\delta U^0/\delta m(t, x_0,m,y)= \hat v^0(t, x_0;m,y)$. 

Next we consider the linear equation satisfied by $U$.  By our previous discussion on $U^0$, the vector field
$$
(t,x_0)\to H^0_p(x_0, D_{x_0}U^0(t,x_0, m),m)   
$$
is $C^1$ with respect to $m$. 
For $(s,m,y)\in [0,1]\times \Pw\times \R^{d_0}$, the map  $\hat U(t,x_0;s,x,m,y):= U(t,x_0,x, (1-s)m+s\delta_y)$ solves a linear equation in which the vector field $$
\hat V(t,x_0; s,m,y):=H^0_p(x_0, D_{x_0}U^0(t,x_0; (1-s)m+s\delta_y), (1-s)m+s\delta_y)
$$ 
and the terminal condition $\ds
\hat g(x_0; x,s,m,y):= G(x_0,x,  (1-s)m+s\delta_y)$
are $C^1$ in $s$. Then  $\hat U$ is $C^1$ in $s$ and   its derivative 
$ {  \hat v}(t,x_0; x,m,y):=(d/ds)\hat U(t,x_0;0,x,m,y)$ solves the linear equation 
$$
\left\{\begin{array}{rl}
\ds & \ds -\partial_t {  \hat v}- \Delta_{x_0}{  \hat v}+ D_{x_0}{  \hat v} \cdot H^0_p(x_0,D_{x_0}U^0, m)\\
&\ds \qquad  +D_{x_0}U \cdot \left( \frac{\delta H_p}{\delta m}(x_0,D_{x_0}U^0, m,y)+ H_{pp}(x_0,D_{x_0}U^0, m)D_{x_0}{  \hat v}^0\right)=0 
\qquad {\rm in} \; (0,T)\times \R^{d_0} ,\\
& \ds {  \hat v}(T,x_0; x,m,y)=\frac{\delta G}{\delta m}(x_0,x,m,y) \qquad {\rm in}\; \R^{d_0}.
\end{array}\right.
$$
As the  solution to this equation depends continuously of the parameters $(m,y)$, Lemma \ref{lem.condC1} states that $U$ is $C^1$ in $m$ with $\delta U/\delta m(t, x_0, x,m,y)= {  \hat v}(t, x_0;x,m,y)$. This proves that  the derivative ${  (\hat v^0,\hat v)}= (\delta U^0/\delta m, \delta U/\delta m)(t, x_0,x, m,y)$ solves \eqref{Prop.derivM}  with $\rho = \delta_y$.

Hence, for any $\rho \in C^0_b$, the pair $(v^0,v) = (\frac{\delta U^0}{\delta m} (t,x_0;m) (\rho), \frac{\delta U}{\delta m}(t,x_0;x,m)(\rho) )$ solves a linear system of the form \eqref{eq.HJsystL} in which the drifts 
$$
V^0(t,x^0;m):= H^0_p(x_0,D_{x_0}U^0(t,x_0,m), m) 
$$ 
and
$$
V(t,x^0;x,m):=  H^0_{pp}(x_0,D_{x_0}U^0(t,x_0,m), m)D_{x_0}U(t,x_0;x)
$$
are bounded in class $C_b^{1}$ and $C_b^{0,n-1} \cap C_b^{1,n-2}$ respectively, while the source terms 
$$
f^0(t,x^0;m):=  \frac{\delta H^0}{\delta m}(x_0,D_{x_0}U^{0}, m)(\rho) 
$$
and 
\begin{align*}
f(t,x_0;x,m)& := D_{x_0}U(t,x_0;x) \cdot  \frac{\delta  H^0_p}{\delta m}(x_0,D_{x_0}U^0, m)(\rho) 
\end{align*}
are in $C_b^{1}$ and $C_b^{0,n-1} \cap C_b^{1,n-2}$ respectively, thanks to Proposition \ref{prop.p0}. We then use Proposition \ref{prop.prop.HighSystL} successively to obtain the estimates:  first with $r=0$ and $l \le n-1$, we get
\begin{multline*}
\Bigl( |\frac{\delta U^0}{\delta m} (t,x_0;m) (\rho)|^2 + | D^l_x \frac{\delta U}{\delta m}(t,x_0;x,m)(\rho) |^2 \Bigr)^{1/2} \leq \\ (1+CT)\sup_{x_0,x} \Big(|\frac{\delta G^0}{\delta m}(x_0; m)(\rho)|^2+ |D^l_x \frac{\delta G}{\delta m}(x_0; x,m)(\rho)|^2\Big)^{1/2}  +CT.
\end{multline*} 
Then by taking the supremum over $\|\rho\|_{-k} = 1, x_0, x$ and summing over $l \le n-1$ we find the estimate for $\|\frac{\delta (U^0,U)}{\delta m}\|_{n-1;k}$. An analogous application of Proposition \ref{prop.prop.HighSystL} with $r=1$ and $l \le n-2$ provides the bound for $\|\frac{\delta (U^0_{x_0}, U_{x_0})}{\delta m}\|_{n-2;k-1}$, while the Lipschitz estimate in $x_0$ for $(\frac{\delta U^0_{x_0}}{\delta m}, \frac{\delta U_{x_0}}{\delta m})$ is obtained similarly with $r=2$, and $l \le n-3$. 
\end{proof}

%%%%%%%%%%%%%%%%%%%%%%%%%%%%%
\subsubsection{Second order differentiability with respect to $m$}

\begin{Proposition}\label{Prop.2Mm} Under the assumptions of Proposition \ref{Prop.1Mm}, $k\geq 3$, the pair $(U^0,U)$ (and its derivatives with respect to $x$) is of class $C^2$ with respect to $m$ and, for any fixed $(x,m,\rho,\rho')\in \R^d\times \Pw\times C^{-(k-1)}\times C^{-(k-1)}$ the derivative $(w^0,w)= (\delta^2 U^0/\delta m^2(t,x_0;m)(\rho, \rho'), \delta^2 U/\delta m^2(t,x_0; x,m)(\rho, \rho') \, )$ solves  
\be\label{Prop.derivM2}
\left\{\begin{array}{rl}
\ds (i)& \ds -\partial_t w^{0}-\Delta_{x_0} w^0+ H^0_p(x_0,D_{x_0}U^{0}, m)\cdot D_{x_0}w^0  + \frac{\delta^2 H^0}{\delta m^2}(x_0,D_{x_0}U^{0}, m)(\rho, \rho') \\
& \ds \qquad +  H^0_{pp}(x_0,D_{x_0}U^{0}, m)D_{x_0}v^0 \cdot D_{x_0}(v')^0+ {  \frac{\delta H^0_p}{\delta m}}(x_0,D_{x_0}U^{0}, m)(\rho) \cdot D_{x_0} (v')^0 \\
&\ds  \qquad 
+{  \frac{\delta H^0_p}{\delta m}}(x_0,D_{x_0}U^{0}, m)(\rho')\cdot D_{x_0} v^0 =0
\qquad  {\rm in} \; (0,T)\times \R^{d_0}\\
\ds (ii) & \ds -\partial_t w- \Delta_{x_0}w+  H^0_p(x_0,D_{x_0}U^0, m)\cdot D_{x_0}w\\
& \ds +D_{x_0}v \cdot \left( \frac{\delta  H^0_p}{\delta m}(x_0,D_{x_0}U^0, m)(\rho')+  H^0_{pp}(x_0,D_{x_0}U^0, m)D_{x_0}(v')^0 \right)\\
& \ds +D_{x_0}v' \cdot \left( \frac{\delta  H^0_p}{\delta m}(x_0,D_{x_0}U^0, m)(\rho) +  H^0_{pp}(x_0,D_{x_0}U^0, m)D_{x_0}v^0\right)\\
&\ds \   +D_{x_0}U \cdot \Bigl(  \frac{\delta  H^0_{pp}}{\delta m}(x_0,D_{x_0}U^0, m)(\rho) D_{x_0}(v')^0+\frac{\delta^2  H^0_p}{\delta m^2}(x_0,D_{x_0}U^0, m)(\rho, \rho')\\
& \ds \qquad  + H^0_{ppp}(x_0,D_{x_0}U^0, m)D_{x_0}v^0D_{x_0}(v')^0+ \frac{\delta H^0_{pp}}{\delta m} (x_0,D_{x_0}U^0, m)(\rho')D_{x_0}v^0\\
% & +  H^0_{pp}(x_0,D_{x_0}u^0, m)D_{x_0}v^0(y)D_{x_0}v^0(y') \\
& \ds  \qquad \qquad + H^0_{pp}(x_0,D_{x_0}U^0, m)D_{x_0}w^0\Bigr)=0 
\ {\rm in} \; (0,T)\times \R^{d_0} \\
\ds (iii) & \ds w^0(T,x_0; m)=\frac{\delta^2 G^0}{\delta m^2}(x_0,m)(\rho, \rho'), \;  w(T,x_0; x,m)=\frac{\delta^2 G}{\delta m^2}(x_0,x,m)(\rho,\rho') \; {\rm in}\; \R^{d_0},
\end{array}\right.
\ee
where $(v^0,v), ((v')^0, v')$ are the solutions to \eqref{Prop.derivM} associated with $\rho$ and $\rho'$ respectively. 
Moreover, if 
$$
\left\|\frac{\delta^2 (G^0,G)}{\delta m^2}\right\|_{n-2;k-1,k-1}
+{\rm Lip}_{n-3;k-2,k-2}^{x_0} (\frac{\delta^2 G^0}{\delta m^2}, \frac{\delta^2 G}{\delta m^2})  \leq M, 
$$
then there exists $T_M,C_M>0$  such that, for any $T\in (0, T_M)$, 
\begin{align*}
& \sup_t \Bigl(  \left\|\frac{\delta^2 (U^0,U)}{\delta m^2}(t)\right\|_{n-2;k-1,k-1}
+({\rm Lip}_{n-3;k-2,k-2}^{x_0} (\frac{\delta^2 U^0}{\delta m^2}, \frac{\delta^2 U}{\delta m^2})(t)) \Bigr)  \\ 
&\qquad   \leq \Bigl( \left\|\frac{\delta^2 (G^0,G)}{\delta m^2}\right\|_{n-2;k-1,k-1}
+{\rm Lip}_{n-3;k-2,k-2}^{x_0} (\frac{\delta^2 G^0}{\delta m^2}, \frac{\delta^2 G}{\delta m^2})  \Bigr)+C_MT.
\end{align*}
\end{Proposition}

\begin{proof} The differentiability of $\delta U^0/\delta m$ and of $\delta U/\delta m$ and the representation formula \eqref{Prop.derivM2} can be established as for $U^0$ and $U$ in Proposition \ref{Prop.1Mm}. To prove the estimate, we use Proposition \ref{prop.prop.HighSystL} with 
$$
V^0(t,x^0;m):= H^0_p(x_0,D_{x_0}U^0(t,x_0,m), m) 
$$ 
and
$$
V(t,x^0x;m):=  H^0_{pp}(x_0,D_{x_0}U^0(t,x_0,m), m)D_{x_0}U(t,x_0,x),
$$
which are bounded in class $C_b^{1}$ and $C_b^{0,n-1} \cap C_b^{1,n-2}$  respectively, while the source terms 
\begin{align*} 
f^0(t,x^0;m):= &  \frac{\delta^2 H^0}{\delta m^2}(x_0,D_{x_0}U^{0}, m)(\rho, \rho')  +  H^0_{pp}(x_0,D_{x_0}U^{0}, m)D_{x_0}v^0 \cdot D_{x_0}(v')^0 \\
& \ds + {  \frac{\delta H^0_p}{\delta m}}(x_0,D_{x_0}U^{0}, m)(\rho) \cdot D_{x_0} (v')^0 +{  \frac{\delta H^0_p}{\delta m}}(x_0,D_{x_0}U^{0}, m)(\rho')\cdot D_{x_0} v^0
\end{align*}
and 
\begin{align*}
f(t,x_0, x;m)& := 
D_{x_0}v \cdot \left( \frac{\delta  H^0_p}{\delta m}(x_0,D_{x_0}U^0, m)(\rho')+  H^0_{pp}(x_0,D_{x_0}U^0, m)D_{x_0}(v')^0 \right)\\
& \ds +D_{x_0}v' \cdot \left( \frac{\delta  H^0_p}{\delta m}(x_0,D_{x_0}U^0, m)(\rho) +  H^0_{pp}(x_0,D_{x_0}U^0, m)D_{x_0}v^0\right)\\
&\ds \   +D_{x_0}U \cdot \Bigl(  \frac{\delta  H^0_{pp}}{\delta m}(x_0,D_{x_0}U^0, m)(\rho) D_{x_0}(v')^0+\frac{\delta^2  H^0_p}{\delta m^2}(x_0,D_{x_0}U^0, m)(\rho, \rho')\\
& \ds \qquad  + H^0_{ppp}(x_0,D_{x_0}U^0, m)D_{x_0}v^0D_{x_0}(v')^0+ \frac{\delta H^0_{pp}}{\delta m} (x_0,D_{x_0}U^0, m)(\rho')D_{x_0}v^0\Bigr)
\end{align*} 
are in  $C_b^{0}$ and $C_b^{0,n-2}$ respectively, thanks to Propositions \ref{prop.p0} and \ref{Prop.1Mm}. By Proposition \ref{prop.prop.HighSystL}, with $r=0$ and $n-2$ we obtain the estimates for $\|\frac{\delta^2 (U^0,U)}{\delta m^2} \|_{n-2;k-1,k-1}$. The Lipschitz bound in $x_0$ of $(\frac{\delta^2 U^0}{\delta m^2},\frac{\delta^2 U}{\delta m^2})$ follows analogously.
\end{proof}

\subsubsection{Lipschitz regularity of the second order derivatives} 

We finally address the Lipschitz regularity of second order derivatives of $U^0$ and $U$ with respect to $m$ and $x_0$.

\begin{Proposition}\label{Prop.3Mm} Under the assumptions of Proposition \ref{Prop.2Mm} and if, in addition,\\
$
{\rm Lip}_{n-3;k-2,k-2} (\frac{\delta^2 G^0}{\delta m^2}, \frac{\delta^2 G}{\delta m^2}) \leq M, 
$
then we have, 
\begin{align*}
& \sup_t ({\rm Lip}_{n-3;k-2,k-2} (\frac{\delta^2 U^0}{\delta m^2}, \frac{\delta^2 U}{\delta m^2})(t))  \leq 
 {\rm Lip}_{n-3;k-2,k-2} (\frac{\delta^2 G^0}{\delta m^2}, \frac{\delta^2 G}{\delta m^2})   +C_MT, 
\end{align*} 
where the constant $C_M$ depend on the regularity of $H$ and $H^0$ and on $M$. 
\end{Proposition}

\begin{proof} Let $(x,\rho,\rho') \in \R^d \times C^{-(k-2)} \times C^{-(k-2)}$, 
$m^1,m^2\in \Pw$, $(U^{0,1},U^1)$ be the solution to \eqref{syst.2bis} associated with $(x, m^1)$ and $(U^{0,2}, U^2)$ be the solution associated with $(x, m^2)$. We denote by $(v^{0,1},v^1)$, $((v')^{0,1},(v')^1)$ (resp. $(v^{0,2},v^2)$, $((v')^{0,2},(v')^2)$) the corresponding solutions to the first order linearized system \eqref{Prop.derivM} associated with $\rho$ and $\rho'$, and by $(w^{0,1},w^1)$ (resp. $(w^{0,2},w^2)$) the  corresponding  solution of the second order linearized system \eqref{Prop.derivM2}. We want to estimate the difference $(z^0,z):= (w^{0,2}-w^{0,1},w^2-w^1)$. We have 
$$
\left\{\begin{array}{l}
\ds  -\partial_t z^{0}-\Delta_{x_0} z^0 +  H^0_p(x_0,D_{x_0}U^{0,1}(t,x_0,m^1), m^1)\cdot D_{x_0}z^0+f^0=0, \\
\ds -\partial_t z- \Delta_{x_0}z+ D_{x_0}z \cdot H^0_p(x_0,D_{x_0}U^{0,1}, m^1)
-H^0_{pp}(x_0,D_{x_0}U^{0,1}, m)D_{x_0}U^1 \cdot D_{x_0}z^{0} +f =0\\
\ds z^0(T)= \frac{\delta^2 G^0}{\delta m^2}(x_0,m^2)(\rho, \rho')-\frac{\delta^2 G^0}{\delta m^2}(x_0,m^1)(\rho, \rho'), \\ 
\ds z(T)= \frac{\delta^2 G}{\delta m^2}(x_0,x,m^2)(\rho, \rho')-\frac{\delta^2 G}{\delta m^2}(x_0,x,m^1)(\rho, \rho')\,,
\end{array}\right.
$$
where
\begin{align*}
f^0:=& (H^0_p(x_0,D_{x_0}U^{0,2}, m^2)
- H^0_p(x_0,D_{x_0}U^{0,1}, m^1))\cdot D_{x_0}w^{0,2} \\
&  + \frac{\delta^2 H^0}{\delta m^2}(x_0,D_{x_0}U^{0,2}, m^2)(\rho, \rho') -\frac{\delta^2 H^0}{\delta m^2}(x_0,D_{x_0}U^{0,1}, m^1)(\rho, \rho')  \\
&   +  H^0_{pp}(x_0,D_{x_0}U^{0,2}, m^2)D_{x_0}v^{0,2}\cdot D_{x_0}(v')^{0,2} - H^0_{pp}(x_0,D_{x_0}U^{0,1}, m^1)D_{x_0}v^{0,1}\cdot D_{x_0}(v')^{0,1} \\
&  +{  \frac{\delta H^0_p}{\delta m }}(x_0,D_{x_0}U^{0,2}, m^2)(\rho)\cdot D_{x_0} (v')^{0,2}-{  \frac{\delta H^0_p}{\delta m }}(x_0,D_{x_0}U^{0,1}, m^1)(\rho)\cdot D_{x_0} (v')^{0,1} \\
&   
+{  \frac{\delta H^0_p}{\delta m }}(x_0,D_{x_0}U^{0,2}, m^2)(\rho')\cdot D_{x_0} v^{0,2}-{  \frac{\delta H^0_p}{\delta m }}(x_0,D_{x_0}U^{0,1}, m^1)(\rho')\cdot D_{x_0} v^{0,1} 
\end{align*}
and
\begin{align*}
f:= & D_{x_0}w^2 \cdot \left( H^0_p(x_0,D_{x_0}U^{0,2}, m^2)- H^0_p(x_0,D_{x_0}U^{0,1}, m^1)\right)\\ 
&  +D_{x_0}v^2 \cdot \left( \frac{\delta  H^0_p}{\delta m}(x_0,D_{x_0}U^{0,2}, m^2)(\rho')+  H^0_{pp}(x_0,D_{x_0}U^{0,2}, m^2)D_{x_0}(v')^{0,2}\right)\\
&  -D_{x_0}v^1 \cdot \left( \frac{\delta  H^0_p}{\delta m}(x_0,D_{x_0}U^{0,1}, m^1)(\rho')+  H^0_{pp}(x_0,D_{x_0}U^{0,1}, m^1)D_{x_0}(v')^{0,1}\right)\\
&  +D_{x_0}(v')^2 \cdot \left( \frac{\delta  H^0_p}{\delta m}(x_0,D_{x_0}U^{0,2}, m^2)(\rho)+  H^0_{pp}(x_0,D_{x_0}U^{0,2}, m^2)D_{x_0}v^{0,2}\right)\\
&  -D_{x_0}(v')^1 \cdot \left( \frac{\delta  H^0_p}{\delta m}(x_0,D_{x_0}U^{0,1}, m^1)(\rho)+  H^0_{pp}(x_0,D_{x_0}U^{0,1}, m^1)D_{x_0}v^{0,1}\right)\\
&    +D_{x_0}U^2 \cdot \Bigl(  \frac{\delta  H^0_{pp}}{\delta m}(x_0,D_{x_0}U^{0,2}, m^2)(\rho)D_{x_0}(v')^{0,2}+\frac{\delta^2  H^0_p}{\delta m^2}(x_0,D_{x_0}U^{0,2}, m^2)(\rho, \rho')\\
&  \qquad  + H^0_{ppp}(x_0,D_{x_0}U^{0,2}, m^2)D_{x_0}v^{0,2}D_{x_0}(v')^{0,2}+ \frac{\delta H^0_{pp}}{\delta m} (x_0,D_{x_0}U^{0,2}, m^2)(\rho')D_{x_0}v^{0,2}\Bigr) \\
&    +\Bigl( H^0_{pp}(x_0,D_{x_0}U^{0,2}, m)D_{x_0}U^2 - H^0_{pp}(x_0,D_{x_0}U^{0,1}, m)D_{x_0}U^1\Bigr)\cdot D_{x_0}w^{0,2} \\
&   -D_{x_0}U^1 \cdot \Bigl(  \frac{\delta  H^0_{pp}}{\delta m}(x_0,D_{x_0}U^{0,1}, m^1)(\rho)D_{x_0}(v')^{0,1}+\frac{\delta^2  H^0_p}{\delta m^2}(x_0,D_{x_0}U^{0,1}, m^1)(\rho, \rho')\\
&  \qquad  + H^0_{ppp}(x_0,D_{x_0}U^{0,1}, m^1)D_{x_0}v^{0,1}D_{x_0}(v')^{0,1}+ \frac{\delta H^0_{pp}}{\delta m} (x_0,D_{x_0}U^{0,1}, m^1)(\rho')D_{x_0}v^{0,1}\Bigr)\,. 
\end{align*}
Proposition \ref{Prop.1Mm} (for the representation of the $(v^{0,i}, v^i)$) and Proposition \ref{Prop.2Mm} (for their Lipschitz regularity in $m$  and in $x_0$) imply in particular
that
$$
\sup_t \big(\|D_{x_0} (v^{0,2}-v^{0,1})\|_\infty + \|D_{x_0} (v^{2}-v^{1})\|_{0,n-3} \big)  \leq C\dw(m_1,m_2)
$$
and hence we have, using also Proposition \ref{Prop.2Mm},
\begin{align*}
 \sup_t \big( \|f^0\|_\infty +\|f\|_{0,n-3} \big) \leq C \dw (m^1, m^2). 
\end{align*}
Using Proposition \ref{prop.prop.HighSystL} (with $r = 0$), we obtain, for any $ l \le n-3$, 
 \begin{align*}
& \sup_{t,x_0, x} (|z^0(t,x_0)|^2+ |D^l_x z(t,x_0,x)|^2)^{1/2}\\
& \; \leq  (1+CT)\sup_{x_0, x} \Bigl( \left|  \frac{\delta^2 G^0}{\delta m^2}(x_0,m^2)(\rho, \rho')-\frac{\delta^2 G^0}{\delta m^2}(x_0,m^1)(\rho, \rho') \right|^2\\
& \qquad +
\left| D^l_x  \frac{\delta^2 G}{\delta m^2}(x_0,x,m^2)(\rho, \rho') -  D^l_x \frac{\delta^2 G}{\delta m^2}(x_0,x,m^1)(\rho, \rho')  \right|^2\Bigr)^{1/2}+ CT\dw(m^1,m^2),
\end{align*}
which gives the claim. 
\end{proof}

We complete this section by stating similar estimates on the Lipschitz regularity of the other second order derivatives: 
\begin{Proposition}\label{Prop.3MmBIS} Under the assumptions of Proposition \ref{Prop.2Mm} and if, in addition, 
${\rm Lip}_{n-3;k-2} (\frac{\delta G^0_{x_0}}{\delta m}, \frac{\delta G_{x_0}}{\delta m}) +  {\rm Lip}_{n-3} (D^2_{x_0} G^0, D^2_{x_0}G)  \leq M$,
then we have, 
\begin{align*}
& \sup_t ({\rm Lip}_{n-3;k-2} (\frac{\delta U^0_{x_0}}{\delta m}, \frac{\delta U_{x_0}}{\delta m})(t))  \leq 
{\rm Lip}_{n-3;k-2} (\frac{\delta G^0_{x_0}}{\delta m}, \frac{\delta G_{x_0}}{\delta m})   +C_MT
\end{align*} 
and 
\begin{align*}
& \sup_t ({\rm Lip}_{n-3} (D^2_{x_0}U^0, D^2_{x_0} U)(t))  \leq 
 {\rm Lip}_{n-3} (D^2_{x_0} G^0, D^2_{x_0}G)   +C_MT, 
\end{align*} 
where the constant $C_M$ depend on the regularity of $H$ and $H^0$ and on $M$. 
\end{Proposition}

As the proof is completely similar to the proof of Proposition \ref{Prop.3Mm}, we omit it. 

%%%%%%%%%%%%%%%%%%%
\subsection{Existence of a solution} 

%%%%%%%%%%%%%%%%%%%%%%%%%
\subsubsection{Definition of the semi-discrete scheme}

Let us fix some horizon $T>0$ (small) and a step-size $\tau:= T/(2N)$ (where $N\in \N$, $N\geq 1$). We set $t_k= kT/(2N)$, $k\in \{0, 2N\}$. We define by backward induction the continuous maps  $U^{0,N}=U^{0,N}(t,x_0,m)$ and $U^N=U^N(t,x_0,x,m)$ as follows: we require that 

\begin{itemize}

\item[(i)] $(U^{0,N},U^N)$ satisfies the terminal condition:  
$$
U^{0,N}(T,x_0,m)= G^0(x_0,m), \;  U^N(T,x_0,x,m)= G(x_0,x,m)\qquad \forall (x_0,x,m)\in \R^d\times \R^{d_0}\times \Pw, 
$$

\item[(ii)]  for $x_0\in \R^{d_0}$ fixed, $(U^{0,N},U^N)$ solves the backward system of first order master equations:
\be\label{blaMajor}
\left\{\begin{array}{rl}
\ds (i)& \ds -\partial_t U^{0} -2\int_{\R^d} \dive_yD_m U^0(t,x_0,m,y)dm(y)\\
&\ds \qquad + 2\int_{\R^d} D_mU^0(t,x_0,m,y)\cdot H_p(x_0,y, D_{x}U(t,x_0,y,m), m)dm(y)=0 \\ 
\ds (ii) & \ds -\partial_t U- 2\Delta_{x} U+ 2H(x_0,x,D_{x}U, m) -2\int_{\R^d} \dive_y D_mU(t,x_0,x,m,y)dm(y)\\
&\ds \qquad   + 2\int_{\R^d} D_m U(t,x_0,x,m,y)\cdot H_p(x_0,y, D_{x}U(t,x_0,y,m), m)dm(y)=0 
\end{array}\right.
\ee
on time intervals of the form $(t_{2j+1}, t_{2j+2})$, for $j=0,\ldots,N-1$, 

\item[(iii)] for $(x,m)\in \R^d\times \Pw$ fixed, $(U^{0,N}, U^N)$ solves the backward system of HJ equations: 
\be\label{blablaMajor}
\left\{\begin{array}{rl}
\ds (i)& \ds -\partial_t U^{0}-2\Delta_{x_0} U^0+ 2H^0(x_0,D_{x_0}U^{0}, m)  =0  \\ 
\ds (ii) & \ds -\partial_t U- 2\Delta_{x_0}U+ 2D_{x_0}U \cdot H^0_p(x_0,D_{x_0}U^0(t,x_0,m), m)=0 \\ 
\end{array}\right.
\ee
on time intervals of the form $(t_{2j}, t_{2j+1})$,  for $j=0,\ldots,N-1$. 
\end{itemize}

Our aim is to show that, if the time horizon is short enough, $(U^{0,N},U^N)$ converges to a solution of the master equation for MFGs with a major player as $N\to+\infty$.

%%%%%%%%%%%%%%%%%%%%%%%%%
\subsubsection{Proof of the existence of a solution}

  For $n\geq 4$ and $k\in \{3, \dots, n-1\}$, let 
\begin{align*}
M &:= 1+\Bigl\|(G^0,G)\Bigr\|_n + \Bigl\|D_{x_0}(G^0,G)\Bigr\|_{n-1} 
+ \Bigl\|D^2_{x_0}(G^0,G)\Bigr\|_{n-2} 
+ ({\rm Lip}_{n-3}^{x_0} (D^2_{x_0} G^0, D^2_{x_0} G))\\
& + \Big\|\frac{\delta (G^0,G)}{\delta m}\Bigr\|_{n-1;k}
+  \Bigl\|\frac{\delta (G^0_{x_0},G_{x_0})}{\delta m}\Bigr\|_{n-2;k-1} 
+ ({\rm Lip}^{x_0}_{n-3;k-2} (\frac{\delta G^0_{x_0}}{\delta m}, \frac{\delta G_{x_0}}{\delta m}))\\
& + \left\|\frac{\delta^2 (G^0,G)}{\delta m^2}\right\|_{n-2;k-1,k-1}
+({\rm Lip}_{n-3;k-2,k-2}^{x_0} (\frac{\delta^2 G^0}{\delta m^2}, \frac{\delta^2 G}{\delta m^2})) \\
& + ({\rm Lip}_{n-3;k-2,k-2} (\frac{\delta^2 G^0}{\delta m^2}, \frac{\delta^2 G}{\delta m^2})) + ({\rm Lip}_{n-3;k-2} (\frac{\delta G^0_{x_0}}{\delta m}, \frac{\delta G_{x_0}}{\delta m}) )+  ({\rm Lip}_{n-3} (D^2_{x_0} G^0, D^2_{x_0}G)) .
\end{align*}

\begin{Lemma}\label{lem.estiUNMm} There exists $T_{M}>0$, depending on the regularity of $H^0$, $H$ and on $M$, such that, for any $T\in (0,T_{M}]$ and $N\geq 1$, we have, for any $t\in [0,T]$,    
\begin{align}\label{estiti}
& \Bigl\|(U^0,U)(t)\Bigr\|_n + \Bigl\|D_{x_0}(U^0,U)(t)\Bigr\|_{n-1} 
+ \Bigl\|D^2_{x_0}(U^0,U)(t)\Bigr\|_{n-2} 
+ ({\rm Lip}_{n-3}^{x_0} (D^2_{x_0} U^0, D^2_{x_0} U)(t)) \notag\\
& + \Big\|\frac{\delta (U^0,U)}{\delta m}(t)\Bigr\|_{n-1;k}
+  \Bigl\|\frac{\delta (U^0_{x_0},U_{x_0})}{\delta m}(t)\Bigr\|_{n-2;k-1} 
+ ({\rm Lip}^{x_0}_{n-3;k-2} (\frac{\delta U^0_{x_0}}{\delta m}, \frac{\delta U_{x_0}}{\delta m})(t)) \\
& + \left\|\frac{\delta^2 (U^0,U)}{\delta m^2}(t)\right\|_{n-2;k-1,k-1}
+({\rm Lip}_{n-3;k-2,k-2}^{x_0} (\frac{\delta^2 U^0}{\delta m^2}, \frac{\delta^2 U}{\delta m^2})(t))  \notag\\
& + ({\rm Lip}_{n-3;k-2,k-2} (\frac{\delta^2 U^0}{\delta m^2}, \frac{\delta^2 U}{\delta m^2})(t)) + ({\rm Lip}_{n-3;k-2} (\frac{\delta U^0_{x_0}}{\delta m}, \frac{\delta U_{x_0}}{\delta m}) (t))+  ({\rm Lip}_{n-3} (D^2_{x_0} U^0, D^2_{x_0}U)(t)) \;   \leq \; M. \notag
\end{align}

Moreover: 
\begin{itemize}
\item The maps $U^{0,N}$ and $U^N$ are globally Lipschitz continuous in all variables and their first and second space derivatives  are globally Holder continuous in all variables, uniformly with respect to $N$.
\item The maps $D_mU^{0,N}$ and $D_mU^N$  are Holder continuous in $(t,x_0,m,y)$ and $(t,x_0,x,m,y)$ respectively, uniformly with respect to $N$, in any set of the form
\begin{align} \label{def.setofformRMn}
& \{(t,x_0, m,y)\in [0,T]\times \R^{d_0} \times \Pw \times \R^d, \; M_2(m)\leq R, \; |y|\leq R\}\notag\\
&  {\rm and}\; 
\{(t,x_0, x,m,y)\in [0,T]\times \R^{d_0}\times \R^d\times \Pw \times \R^d, \; M_2(m)\leq R, \; |y|\leq R\}\,
\end{align}
respectively (where $M_2(m)= (\int_{\R^d} |y|^2m(dy))^{1/2}$). 
\end{itemize}
\end{Lemma}

\begin{proof} We only sketch the proof, since it is exactly the same as for the second order master equation (see Lemma \ref{lem.estiUN}). The proof of \eqref{estiti} can be established by collecting the estimates in Propositions \ref{Prop.DerivUL},  \ref{Prop.DerivU02} and \ref{prop:rlipdelta2U0} in Section \ref{Sec.FirstOrdreMaster} below, which provide the bounds on intervals of the form $(t_{2j+1}, t_{2j+2})$, and, for the intervals of the form $(t_{2j}, t_{2j+1})$,  by  Proposition \ref{prop.p0}, \ref{Prop.1Mm}, \ref{Prop.2Mm} and \ref{Prop.3Mm}. 

The Lipschitz regularity in space of $U^{0,N}$ and $U^N$ and of their first and second order space derivatives follows immediately from \eqref{estiti}. As $D_mU^{0,N}$ and $D_mU^N$ are bounded {  according to \eqref{estiti}}, $U^{0,N}$ and $U^N$ and their first and second order space derivatives are also Lipschitz continuous in $m$. Finally, since $U^{0,N}$ and $U^N$ satisfy the equations \eqref{blaMajor} and \eqref{blablaMajor},  the bounds in \eqref{estiti} show that $\partial_t U^{0,N}$ and $\partial_t U^N$ are bounded and therefore that $U^{0,N}$ and $U^N$ are also Lipschitz continuous in time. The global Holder regularity of the first and second space derivatives of $U^{0,N}$ and $U^N$  then follows by interpolation (Lemma \ref{lem.interp0}). 

The Lipschitz regularity in space and in measure of $D_mU^{0,N}$ and $D_mU^N$ is a consequence of \eqref{estiti} while the Holder regularity in time in sets of the form \eqref{def.setofformRMn} comes from interpolation (Lemma \ref{lem.interp}). 
\end{proof}

\begin{proof}[Proof of Theorem \ref{thm.Mm}] It relies exactly on the same argument as in the proof of Theorem \ref{theo.ShortTime} and we omit it. 
\end{proof}

%%%%%%%%%%%%%%%%%%%%%%%%%%%%%%%%%
\subsection{Uniqueness of the solution}

We finally address the uniqueness of the solution of the master equation for MFGs with a major player: 

\begin{Theorem} Let $(U^{0,1}, U^1)$ and $(U^{0,2},U^2)$ be two classical solutions to \eqref{eq.MasterMM} defined on the time interval $[0,T]$ and such that $D_{x_0}U^{0,1}$ and  $D_{x_0,x}U^1$ are uniformly Lipschitz continuous in the space and measure variables. Then $(U^{0,1}, U^1)=(U^{0,2},U^2)$. 
\end{Theorem}

\begin{proof} Let $(t_0, \bar x_0, \bar m_0)\in [0,T)\times \R^{d_0}\times \Pw$ be an initial condition, $Z$ be a random variable with law $\bar m_0$  and let $(X^0_t, m_t, X_t)$ be the solution to 
$$
\left\{\begin{array}{l}
dX^0_t= -  H^0_p(X^0_t, D_{x_0}U^{0,1}(t,X^0_t,m_t),m_t) dt + \sqrt{2} dW^0_t\; {\rm in }\; (0,T)\\ 
dm_t= \left(\Delta m_t+ \dive(m_t H_p(X^0_t, x, D_xU^1(t,X^0_t, x, m_t), m_t))\right)dt \; {\rm in }\; (0,T)\times \R^d\\ 
dX_t= -H_p(X^0_t, X_t, D_xU^1(t,X^0_t, X_t, m_t)dt +\sqrt{2} dW_t \; {\rm in }\; (0,T)\\ 
X^0_{t_0}=\bar x_0, \; m_{t_0}= \bar m_0, \; X_{t_0}=Z
\end{array}\right. 
$$
where $(W^0_t)$ and $(W_t)$ are  Brownian motions, $(W^0_t)$, $(W_t)$ and $Z$ being independent. As $D_xU^{0,1}$ and  $D_xU^1$ are globally Lipschitz continuous, the above   system has a unique solution. Note that $m_t$ is the conditional law of $X_t$ given $(W^0_s)_{s\leq t}$. 

We compute the variation of $U^{0,1}$ along $(t,X^0_t, m_t)$:
\begin{align*}
& dU^{0,1}(t,X^0_t, m_t) =  \Bigl( \partial_t U^{0,1}+  \Delta_{x_0} U^{0,1}- H^0_p(X^0_t, D_{x_0}U^{0,1}, m_t)\cdot D_{x_0}U^{0,1} \\ 
&  - \int_{\R^d} D_mU^{0,1}\cdot H_p(X^0_t, y, D_xU^1(t,X^0_t,y,m_t),m_t) m_t(dy) + \int_{\R^d} \dive_y D_mU^{0,1}m_t(dy)\Bigr)dt \\ 
& +\sqrt{2} D_{x_0}U^{0,1}\cdot dW^0_t,
\end{align*}
where, unless specified otherwise,  $U^{0,1}$ and its space derivatives are computed at $(t,X^0_t,m_t)$ while $D_mU^{0,1}$ and its space derivatives are computed at $(t,X^0_t,m_t,y)$. In view of the equation satisfied by $U^{0,1}$, we find 
\begin{align*}
dU^{0,1}(t,X^0_t, m_t)  = & \Bigl( H^0(X^0_t, D_{x_0}U^{0,1},m_t) - H^0_p(X^0_t, D_{x_0}U^{0,1}, m_t)\cdot D_{x_0}U^{0,1}\Bigr)dt\\
& \qquad  +\sqrt{2} D_{x_0}U^{0,1}\cdot dW^0_t\,.
\end{align*}
We proceed in the same way for $U^{0,2}$ and obtain, in view of the equation satisfied by $U^{0,2}$: 
\begin{align*}
& dU^{0,2}(t,X^0_t, m_t) =  \Bigl(H^0(X^0_t, D_{x_0}U^{0,2},m_t) - H^0_p(X^0_t, D_{x_0}U^{0,1}, m_t)\cdot D_{x_0}U^{0,2} \\ 
&  + \int_{\R^d} D_mU^{0,2}\cdot (H_p(X^0_t, y, D_xU^2(t,X^0_t,y,m_t),m_t) -H_p(X^0_t, y, D_xU^1(t,X^0_t,y,m_t)),m_t) m_t(dy) \Bigr)dt \\ 
& +\sqrt{2} D_{x_0}U^{0,2}\cdot dW^0_t,
\end{align*}
where, unless specified otherwise,  $U^{0,2}$ and its space derivatives are computed at $(t,X^0_t,m_t)$ while $D_mU^{0,2}$ and its space derivatives are computed at $(t,X^0_t,m_t,y)$. Therefore 
\begin{align*}
& d(U^{0,2}-U^{0,1})^2 =  2(U^{0,2}-U^{0,1}) \Bigl(H^0(X^0_t, D_{x_0}U^{0,2},m_t)- H^0(X^0_t, D_{x_0}U^{0,1},m_t) \\ 
& \qquad  - H^0_p(X^0_t, D_{x_0}U^{0,1}, m_t)\cdot (D_{x_0}U^{0,2}-D_{x_0}U^{0,1}) \\ 
& \qquad  + \int_{\R^d} D_mU^{0,2}\cdot (H_p(X^0_t, y, D_xU^2(t,X^0_t,y,m_t),m_t) -H_p(X^0_t, y, D_xU^1(t,X^0_t,y,m_t),m_t)) m_t(dy)\Bigr)dt  \\ 
& \qquad + 2(D_{x_0}U^{0,2}-D_{x_0}U^{0,1})^2dt +2\sqrt{2} (U^{0,2}-U^{0,1})(D_{x_0}U^{0,2}-D_{x_0}U^{0,1})\cdot dW^0_t.
\end{align*}
Let us set $U^{0,i}_t= U^{0,i}(t,X^0_t,m_t)$ (for $i=1, 2$). We integrate in time between $s\in [t_0,T]$ and $T$, take expectation and use the fact that $U^{0,1}_T=U^{0,2}_T= G^0(X^0_T,m_T)$: 
\begin{align*}
0= & \E\Bigl[(U^{0,2}_s-U^{0,1}_s)^2 +\int_s^T   2(U^{0,2}_t-U^{0,1}_t) \Bigl(H^0(X^0_t, D_{x_0}U^{0,2},m_t)- H^0(X^0_t, D_{x_0}U^{0,1},m_t) \\ 
& \qquad  - H^0_p(X^0_t, D_{x_0}U^{0,1}, m_t)\cdot (D_{x_0}U^{0,2}-D_{x_0}U^{0,1}) \\ 
& \qquad  + \int_{\R^d} D_mU^{0,2}\cdot (H_p(X^0_t, y, D_xU^2(t,X^0_t,y,m_t),m_t) -H_p(X^0_t, y, D_xU^1(t,X^0_t,y,m_t),m_t)) m_t(dy)\Bigr)dt  \\ 
& \qquad + 2\int_s^T |D_{x_0}U^{0,2}-D_{x_0}U^{0,1}|^2dt \ \Bigr].
\end{align*}
Thanks to the regularity of the solutions, we have by Cauchy-Schwarz inequality and for any $\ep>0$: 
\begin{align*}
0\geq & \E\Bigl[(U^{0,2}_s-U^{0,1}_s)^2 - \int_s^T  \Bigl( C_\ep (U^{0,2}_t-U^{0,1}_t)^2 +\ep |D_{x_0}(U^{0,2}-U^{0,1})|^2 \\ 
& \qquad  + \ep \int_{\R^d} |D_x(U^2-U^1)(t,X^0_t,y,m_t)|^2 m_t(dy) \Bigr)dt+ 2\int_s^T |D_{x_0}(U^{0,2}-U^{0,1})|^2\ dt \ \Bigr].
\end{align*}
So, for $\ep$ small enough, we obtain 
\begin{align*}
0\geq & \E\Bigl[(U^{0,2}_s-U^{0,1}_s)^2 - \int_s^T  \Bigl( C_\ep (U^{0,2}_t-U^{0,1}_t)^2   + \ep \int_{\R^d} |D_x(U^2-U^1)(t,X^0_t,y,m_t)|^2 m_t(dy)\Bigr)dt  \\ 
& \qquad + \int_s^T |D_{x_0}(U^{0,2}-U^{0,1})|^2dt \ \Bigr].
\end{align*}
We argue in the same way for $U^i_t:=U^i(t,X^0_t,X_t,m_t)$ ($i=1,2$) and find: 
\begin{align*}
0\geq & \E\Bigl[(U^{2}_s-U^{1}_s)^2 - \int_s^T   \Bigl( C_\ep (U^{2}_t-U^{1}_t)^2   + \ep |D_{x_0}(U^{0,2}-U^{0,1})|^2 \\
&  + \ep \int_{\R^d} |D_x(U^2-U^1)(t,X^0_t,y,m_t)|^2 m_t(dy)\Bigr)dt  
%\\ 
%& \qquad 
+  \int_s^T |D_{x_0}(U^{2}-U^{1})|^2 + |D_{x}(U^{2}-U^{1})|^2 \ dt \ \Bigr].
\end{align*}
We add the last two inequalities to obtain: 
\begin{align}
0\geq & \E\Bigl[(U^{0,2}_s-U^{0,1}_s)^2 + (U^{2}_s-U^{1}_s)^2 - \int_s^T   \Bigl( C_\ep ((U^{0,2}_t-U^{0,1}_t)^2+(U^{2}_t-U^{1}_t)^2)\notag \\ 
&    + \ep |D_{x_0}(U^{0,2}-U^{0,1})|^2 + 2 \ep \int_{\R^d} |D_x(U^2-U^1)(t,X^0_t,y,m_t)|^2 m_t(dy)\Bigr)dt  \label{lenf}\\ 
&  +\int_s^T  \Bigl( |D_{x_0}(U^{0,2}-U^{0,1})|^2+  |D_{x_0}(U^{2}-U^{1})|^2 + |D_{x}(U^{2}-U^{1})|^2 \Bigr)\ dt \ \Bigr].\notag
\end{align}
Note that, as $m_t$ is the conditional law of $X_t$ given $(W^0_u)_{u\leq t}$, we have 
\begin{align*}
& \E \Bigl[   |D_x(U^2-U^1)(t,X^0_t,X_t,m_t)|^2 \Bigr] = \E \Bigl[  \E\Bigl[  |D_x(U^2-U^1)(t,X^0_t,X_t,m_t)|^2\ |\  (W^0_u)_{u\leq t}\Bigr]\Bigr] \\
& \qquad = \E \Bigl[ \int_{\R^d} |D_x(U^2-U^1)(t,X^0_t,y,m_t)|^2m_t(dy)\Bigr]
\end{align*}
since $X^0_t$ and $X_t$ are adapted to $(W^0_u)_{u\leq t}$. Plugging this relation into \eqref{lenf} we find  therefore, for $\ep>0$ small enough,  
\begin{align*}
0\geq & \E\Bigl[(U^{0,2}_s-U^{0,1}_s)^2 + (U^{2}_s-U^{1}_s)^2 - \int_s^T   \Bigl( C_\ep ((U^{0,2}_t-U^{0,1}_t)^2+(U^{2}_t-U^{1}_t)^2)dt \\ 
& \qquad +(1/2) \int_s^T  |D_{x_0}(U^{0,2}-U^{0,1})|^2+  |D_{x_0}(U^{2}-U^{1})|^2 + |D_{x}(U^{2}-U^{1})|^2 \ dt \ \Bigr].
\end{align*}
We conclude by Gronwall's inequality that, for any $t\in [t_0,T]$,  
$$
 \E\Bigl[(U^{0,2}(t,X^0_t,m_t)-U^{0,1}(t,X^0_t,m_t))^2 + (U^{2}(t,X^0_t,X_t,m_t)-U^{1}(t,X^0_t,X_tm_t))^2\Bigr] = 0. 
$$
For $t=t_0$, we have therefore $U^{0,2}(t_0,\bar x_0, \bar m_0)=U^{0,1}(t_0,\bar x_0, \bar m_0)$ and 
$$
U^{1}(t_0,\bar x_0, Z, \bar m_0)
=U^{2}(t_0,\bar x_0,Z, \bar m_0)\; \; a.s. 
$$ 
If $\bar m_0$ has a positive density, the fact that the law of $Z$ is $\bar m_0$ easily implies the equality of $U^1$ and $U^2$ at any point $(t_0,\bar x_0, x, \bar m_0)$ for $x\in \R^d$. We conclude by density of such laws and by continuity of the $U^i$'s. 
\end{proof}

%%%%%%%%%%%%%%%%%%%%%%%%%%%%
\section{Estimates on the MFG system}\label{sec:ReguMFG}

We are now left to prove the estimates on the first order master equations considered in the two previous sections. As the solutions of these equations are built by a method of characteristics, where the characteristics are the solutions of the MFG system, we first need to discuss the well-posedness and the regularity properties of this system: 
\be\label{eq.MFGcoupled}
\left\{ \begin{array}{rl}
(i) &  \ds -\partial_t u(t,x) -{\rm Tr}(a(t,x)D^2u(t,x))+H(x_0,x,Du(t,x),m(t))=0\qquad {\rm in}\; (t_0,T)\times \R^d,\\
(ii) &  \ds \partial_t m(t,x) -\sum_{i,j} D_{ij}(a_{i,j}(t,x)m(t,x)) -\dive( m(t,x) H_p(x_0,x,Du(t,x),m(t)))=0\\
& \ds \qquad \qquad \qquad \qquad \qquad \qquad \qquad \qquad  {\rm in}\; (t_0,T)\times \R^d,\\
(iii) & \ds m(t_0)=m_0, \qquad u(T,x)= G(x_0,x,m(T)) \qquad {\rm in}\;\R^d.
\end{array}\right.
\ee
Here $x_0\in \R^{d_0}$ is treated as a fixed parameter.
We also present similar results for the corresponding linearized systems. These estimates are motivated by the construction  and the regularity of the first order master equation in the next section. 

Let us first explain the notion of solution to \eqref{eq.MFGcoupled}.
Fix $(t_0,m_0)\in [0,T]\times \Pw$ and $x_0 \in \R^{d_0}$. We say that $(u,m)$ is a  solution to \eqref{eq.MFGcoupled} if $u\in C^0([t_0,T],C^2_b)$ satisfies 
$$
u(t,x)= G(x_0,x,m(T))  +\int_t^T ({\rm Tr}(a(s,x)D^2u(s,x))-H(x_0,x,Du(s,x),m(s)))ds \qquad \forall t\in [t_0,T]
$$
and if $m\in C^0([t_0,T], \Pw)$ solves the Fokker-Planck equation in the sense of distributions: for any $\phi\in C^\infty_c([0,T)\times \R^d)$, 
\begin{align*}
0& = \int_{\R^d} \phi(0,x) m_0(dx)\\
& \qquad +  \int_0^T \int_{\R^d} ({\rm Tr}(a(s,x)D^2 \phi(s,x)) - D\phi(s,x) \cdot H_p(x_0,x,Du(s,x),m(s)) ) m(s,dx)ds\,.
\end{align*}

The assumptions on $a$, $H$ and $G$ given in Subsection \ref{subsec.Hyp} are in force throughout the section. 

%%%%%%%%%%%%%%%%%%%%%%%%
\subsection{Well-posedness and regularity of the MFG system}

We discuss here the well-posedness of the MFG system \eqref{eq.MFGcoupled} and provide several estimates. Let us start with the Hamilton-Jacobi (HJ) equation (general estimates on this equation are given in the Appendix \ref{sec.HJ}). 

\begin{Proposition}\label{prop.esti1MFG-0} For any $M>0$, there exist  $T_M>0$,  $L_M>0$, depending on $C_0$ and $\gamma$ given in assumptions \eqref{Conda} and \eqref{CondH}, such that, if $ \sup_{x_0, m}\|G(x_0,\cdot,m)\|_1\leq M$, then, for any $T\in (0,T_M)$ and any $m\in C^0([0,T], \Pw)$, the solution $u$ to the HJ equation
\be\label{eq.HJ-0}
\left\{ \begin{array}{l}
 \ds -\partial_t u(t,x) -{\rm Tr}(a(t,x)D^2u(t,x))+H(x_0,x,Du(t,x),m(t))=0\qquad {\rm in}\; (t_0,T)\times \R^d\\
\ds  u(T,x)=  G(x_0,x,m(T)) \qquad {\rm in}\;\R^d
\end{array}\right.
\ee
 satisfies
$$
\sup_{t\in [t_0,T]} \|u\|_1 \leq   \sup_{x_0, m}\|G(x_0, \cdot, m)\|_1+ L_MT\,. 
$$
 Henceforth, we set 
$
K_M:= \sup_{x_0, m}\|G(x_0, \cdot, m)\|_1+ L_MT_M$.
\vskip0.3em
 If, in addition, $\sup_{x_0, m}\|G(x_0,\cdot,m)\|_n\leq M$, then there exists $C_M>0$, depending on $n$,  $C_0$, $\gamma$ and
$$
\sup_{t\in [0,T_M]} \|a(t)\|_n + \sup_{|p|\leq K_M, x_0\in \R^{d_0}, m\in \Pw} \sum_{k=0}^n \|D^k_{(x,p)}H(x_0,\cdot,p,m)\|_\infty,
$$
such that  $u$ satisfies, for any $T\in (0,T_M)$, $x_0\in \R^{d_0}$ and  $r \le n$, 
$$
\sup_{t\in [t_0,T], x\in \R^d} |D^r_x u(t, x)| \leq  \sup_{x\in \R^d} |D^r_x G(x_0, x, m(T))|+ C_MT.
$$
Therefore, for any $x_0\in \R^{d_0}$,  
\be\label{norman-0}
\sup_{t\in [t_0,T]} \|u(t)\|_n \leq \sup_m \|G(x_0, \cdot, m)\|_n+ C_MT.
\ee
\end{Proposition}

\begin{proof} Use Propositions \ref{prop.LipEstiBernstein} and \ref{prop.High}. 
\end{proof}

Next we discuss the dependence of the solution $u$ of \eqref{eq.HJ-0} with respect to $(m(t))_{t\in [t_0,T]}$ and $x_0 \in \R^{d_0}$.  We stress that, hereafter, we currently use the preliminary gradient estimate $\sup_{t\in [t_0,T_M]} \|u(t)\|_1 \leq K_M$ which is obtained as a first step in Proposition \ref{prop.esti1MFG-0}. In particular, the Hamiltonian $H(x_0,x,p,m)$ will be systematically estimated for $|p|\leq K_M$.

\begin{Proposition} \label{Prop.LipDep-00}  If the assumptions of Proposition \ref{prop.esti1MFG-0} are satisfied so that \eqref{norman-0} holds true, then there exists  $T_M>0$ such that, for $T\in (0, T_M)$ and any $t_0\in [0,T]$, for any $m^1, m^2\in C^0([0,T], \Pw)$ and any $x_0^1,x_0^2\in \R^{d_0}$,  if  $u^1$ and $u^2$ are the corresponding solutions to the HJ equation \eqref{eq.HJ-0},  then we have, for $n\geq 2$,
\begin{align*}
 \sup_{t\in [t_0,T]} & \|u^1(t)- u^2(t)\|_{n-1}  \leq  C_M\, T \big( \sup_{t\in [t_0,T]} \dw(m^1(t),m^2(t)) + |x_0^1-x_0^2|\big) \\
& + (1+ C_M T)\big\{[{\rm Lip}_{0,n-1}(G)]\dw(m^1(T),m^2(T)) + [{\rm Lip}^{x_0}_{n-1}(G)]|x_0^1-x_0^2|\big\}
\end{align*}
 where $C_M$ depends on the same quantities as in Proposition \ref{prop.esti1MFG-0} as well as on ${\rm Lip}_{n-1,n}(H(x_0, \cdot, \cdot, m))$, ${\rm Lip}^{x_0}_{n-1,n}(H(x_0, \cdot, \cdot, m))$ (for $x\in \R^d$ and $|p|\leq K_M$) and  $\sup_{x_0, m}\|G(x_0, \cdot, m)\|_n$.
\end{Proposition}

\begin{proof} The map $v:=u^1-u^2$ satisfies 
$$
\left\{ \begin{array}{l}
\ds -\partial_t v -{\rm Tr}(a(t,x)D^2v)+ V(t,x)\cdot Dv+f(t,x)=0\\
\ds v(T,x)= G(x_0^1,x,m^1(T))-G(x_0^2,x,m^2(T)) 
\end{array}\right.
$$
where $\ds V(t,x)= \int_0^1 H_p(x, x_0^2, sDu^1(t,x)+(1-s)Du^2(t,x),m^2(t))ds$ and  
$$
f(t,x):= H(x_0^1,x,Du^1(t,x),m^1(t))-H(x_0^2,x,Du^1(t,x),m^2(t)).
$$
By Proposition \ref{prop.highS} (applied with  $k=1$ and $n-1$), we have 
\begin{align*}
\sup_{t\in [0,T]}  & \|u^1(t)-u^2(t)\|_{n-1}  \leq (1+CT)\|G( x_0^1, \cdot, m^1(T))-G( x_0^2, \cdot, m^2(T))\|_{n-1}+ CT\sup_{t\in [t_0,T]} \|f(t)\|_{n-1}\\
& \leq (1+CT)\big\{[{\rm Lip}_{0,n-1}(G)]\dw(m^1(T),m^2(T)) + [{\rm Lip}^{x_0}_{n-1}(G)]|x_0^1-x_0^2|\big\}\\
& \qquad + CT \big( \sup_{t\in [t_0,T]} \dw(m^1(t),m^2(t)) + |x_0^1-x_0^2|\big)\,,
\end{align*}
 where the constant $C$ depends on $H$ and on $\sup_{t\in [0,T]}\|V(t)\|_{n-1}$, hence   on $\sup_{t\in [0,T]}[\|u^1(t)\|_{n}$, $\sup_{t\in [0,T]}\|u^2(t)\|_{n}]$, which are estimated thanks to Proposition \ref{prop.esti1MFG-0}.
\end{proof}

In our next step, we consider the solution to the Fokker-Planck equation
\be\label{eq.FP-0}
\left\{ \begin{array}{l}
 \ds \partial_t \tilde m(t,x) -\sum_{i,j} D_{ij}(a_{i,j}(t,x)\tilde m(t,x)) -\dive( \tilde m(t,x) H_p(x_0,x,Du(t,x), m(t))=0\; {\rm in}\; (t_0,T)\times \R^d\\
\ds \tilde m(t_0)=m_0 \qquad {\rm in}\;\R^d
\end{array}\right.
\ee
where $(m(t))_{t\in [t_0,T]}$ is given and $u$ satisfies \eqref{eq.HJ-0}. Let us recall that, under the assumptions of Proposition \ref{prop.esti1MFG-0}, there exists a unique weak solution $\tilde m\in C^0([t_0,T],\Pw)$ to \eqref{eq.FP-0}. 

\begin{Proposition} \label{Prop.LipDep-0}   Assume that 
\begin{equation}\label{eqglipboh}
\|D_xG\|_\infty\leq M\,,\quad \|D^2_{xx}G\|_\infty\leq M\,, \quad {\rm Lip}_{0,1}(G) + {\rm Lip}^{x_0}_{1}(G) \leq M\,.
\end{equation}
Then  there exists a constant $C_M>0$,  only depending on $M$, $\|a\|_2$ and the regularity of $H$, such that, for any $m^1, m^2\in C^0([0,T], \Pw)$, $x_0^1, x_0^2 \in \R^{d_0}$ and $m_0^1,m_0^2\in \Pw$, if  $u^1$ and $u^2$ are the corresponding solutions to the HJ equation \eqref{eq.HJ-0} with $x_0 = x_0^i$ and if $\tilde m_1,\tilde m_2$ are the corresponding solutions to \eqref{eq.FP-0} starting from $m_0^1$ and $m_0^2$ respectively, then 
$$
\sup_{t\in [t_0,T]} \dw^2(\tilde m^1(t),\tilde m^2(t)) \leq (1+C_MT) \dw^2(m^1_0,m^2_0)+ C_MT\left(\sup_{t\in [t_0,T]} \dw^2(m^1(t),m^2(t)) + |x_0^1 - x_0^2|^2 \right).
$$
\end{Proposition}

\begin{proof} 
We can represent $\tilde m^i(t)$ as the law of $X^i_t$ where $\E[|X^1_0-X_0^2|^2]= \dw^2(m^1_0,m^2_0)$ and $X^i$ solves  
$$
X^i_t = X^i_0-\int_0^t H_p( x_0^i,X^i_s,Du^i(s,X^i_s),m^i(s)) ds + \sqrt{2} \int_0^t \sigma(s,X^i_s)dB_s, 
$$
so that
\begin{align*}
&\E\left[ |X^1_t-X^2_t|^2\right]  \leq  \E\left[|X^1_0-X^2_0|^2\right] \\
& \qquad  +2\E\left[\int_0^t (X^1_s-X^2_s)\cdot \left(H_p( x_0^1, X^1_s, Du^1,m^1(t))-H_p( x_0^2, X^2_s, Du^2,m^2(t))\right) ds\right] \\
& \qquad + \E\left[  \int_0^t {\rm Tr} \left((\sigma(s,X^1_s)-\sigma(s,X^2_s))(\sigma(s,X^1_s)-\sigma(s,X^2_s))^*\right)ds|\right]  \leq  \E\left[|X^1_0-X^2_0|^2\right] \\
&  \qquad\qquad +C_M \E\left[\int_0^t  (|X^1_s-X^2_s|^2 + |D(u^1-u^2)(s,X^1_s)|^2+ \dw^2(m^1(s),m^2(s)) + |x_0^1 - x_0^2|^2) ds\right] 
\end{align*}
where $C_M$ depends  on the Lipschitz regularity of $H_p$ in $\R^{d_0} \times \R^d \times B(K_M)\times \Pw$ (where $K_M$ is defined in Proposition \ref{prop.esti1MFG-0}), on  $\sup_t \|u^1(t)\|_2$,  and on the Lipschitz regularity of $\sigma$. We infer from Gronwall's Lemma that
\begin{align*}
& \E\left[|X^1_t-X^2_t|^2\right] \leq   (1+C_MT) \E\left[|X^1_0-X^2_0|^2\right]\\ &
\qquad \qquad \qquad  + C_MT\left( \sup_t \|D(u^1-u^2)(t)\|_\infty^2+\sup_{t\in[t_0,T]}\dw^2(m^1(t),m^2(t)) + |x_0^1 - x_0^2|^2\right).
\end{align*}
As $\E[|X^1_0-X_0^2|^2]= \dw^2(m^1_0,m^2_0)$ and $\dw^2(\tilde m^1(t),\tilde m^2(t))\leq \E\left[|X^1_t-X^2_t|^2\right]$, we obtain: 
\begin{align*}
\sup_{t\in [t_0,T]} \dw^2(\tilde m^1(t),\tilde m^2(t)) & \leq (1+C_MT) \dw^2(m^1_0,m^2_0) \\
& \qquad + C_MT\left(\sup_t \|D(u^1-u^2)(t)\|_\infty^2+\sup_{t\in[t_0,T]}\dw^2(m^1(t),m^2(t)) + |x_0^1 - x_0^2|^2 \right).
\end{align*}
We estimate the term $\sup_t \|D(u^1-u^2)(t)\|_\infty^2$ by Proposition \ref{Prop.LipDep-00}  (with $n=2$):  since ${\rm Lip}_{0,1}(G)$ and ${\rm Lip}^{x_0}_{1}(G)$ are estimated by \eqref{eqglipboh}, we deduce, for some (possibly different) constant $C_M$:
$$
\sup_{t\in [t_0,T]} \dw^2(\tilde m^1(t),\tilde m^2(t)) \leq (1+C_MT) \dw^2(m^1_0,m^2_0)+ C_MT \left(\sup_{t\in [t_0,T]} \dw^2(m^1(t),m^2(t)) + |x_0^1 - x_0^2|^2 \right).
$$
\end{proof}

Collecting the estimates in Propositions \ref{prop.esti1MFG-0}, \ref{Prop.LipDep-00} and \ref{Prop.LipDep-0}  yields the well-posedness of the MFG system and estimates on the solution: % \Mm{state boundedness of $m$ in $C(\Pw)$?}

\begin{Proposition}\label{Prop.LipDep} Fix $M>0$ and assume that
\eqref{eqglipboh} holds true and that $\|G\|_n\leq M$ holds. Then there exists $T_M, C_M>0$, depending on $M$, $n$,  $C_0$, $\gamma$ and
$$
\sup_{t\in [0,T_M]} \|a(t)\|_n + \sup_{|p|\leq K_M, x_0\in \R^{d_0}, m\in \Pw} \sum_{k=0}^n \|D^k_{(x,p)}H(x_0,\cdot,p,m)\|_\infty,
$$
(where $K_M$ is given in Proposition \ref{prop.esti1MFG-0}) such that, for any $T\in (0,T_M)$, for any $(t_0,m_0)\in [0,T]\times \Pw$, there exists a unique solution to the MFG system \eqref{eq.MFGcoupled}. This solution satisfies 
$$
 \sup_{t\in [t_0,T]} \|u(t)\|_n \leq \|G(x_0, \cdot, m(T))\|_n+ C_MT.
$$

Moreover, if $(t_0,m_0^1)$ and $(t_0,m_0^2)$ are two initial conditions in $[0,T]\times \Pw$ and $x_0^1, x_0^2 \in \R^{d_0}$, and if $(u^1,m^1)$ and $(u^2,m^2)$ are the corresponding solutions to the MFG system \eqref{eq.MFGcoupled} with $x_0 = x_0^1$ and $x_0 = x_0^2$ respectively, then
$$
 \sup_{t\in [t_0,T]} \dw(m^1(t),m^2(t)) \leq (1+C_MT) \dw(m_0^1, m_0^2) + C_MT |x_0^1 - x_0^2|\,,
 $$
and
\begin{align*}
 \sup_{t\in [t_0,T]} & \|u^1(t)- u^2(t)\|_{n-1}  \leq C_M T \big( \dw(m^1_0,m^2_0) + |x_0^1-x_0^2|\big) \\
& + (1+C_M T)\big\{[{\rm Lip}_{0, n-1}(G)](\dw(m^1_0,m^2_0) + |x_0^1-x_0^2|) + [{\rm Lip}^{x_0}_{n-1}(G)] |x_0^1-x_0^2|\big\}.
\end{align*}
%where now $C_M$ may also depend on ${\rm Lip}_{n-1,n}(H(x_0, \cdot, \cdot, m))$ and ${\rm Lip}^{x_0}_{n-1,n}(H(x_0, \cdot, \cdot, m))$ (for $x\in \R^d$ and $|p|\leq K_M$) as well as on $\sup_{x_0, m}\|G(x_0, \cdot, m)\|_n$.
\end{Proposition} 

\begin{proof} The existence and uniqueness result  come from a standard fixed point argument on $C^0([t_0,T],\Pw)$ for $T$ small enough (say $T\leq T_M$ where $C_MT_M\leq 1/2$,  $C_M$ being given by the previous Propositions). For the stability with respect to the initial condition, one first uses the estimate in Proposition \ref{Prop.LipDep-0} with $\tilde m^i=m^i$:
$$
\sup_{t\in [t_0,T]} \dw^2( m^1(t), m^2(t)) \leq (1+C_MT) \dw^2(m^1_0,m^2_0)+ C_MT\left(\sup_{t\in [t_0,T]} \dw^2(m^1(t),m^2(t)) + |x_0^1 - x_0^2|^2 \right).
$$
Thus, as $C_MT\leq 1/2$, one obtains  
$$
\sup_{t\in [t_0,T]} \dw( m^1(t), m^2(t)) \leq (1+C_MT) \dw(m^1_0,m^2_0) + C_M T |x_0^1 - x_0^2|,
$$
modifying $C_M$ if necessary. 
Plugging this estimate into the estimate for the $u^i$ in Proposition \ref{Prop.LipDep-00} gives the result.
\end{proof}

%\begin{Remark} We point out that the existence time $T_M$ actually depends only on the bound $ \|D_xG\|_\infty\leq M\,,\|D^2_{xx}G\|_\infty \leq M$, hence it only depends on the $C^2$ bound of the function $G$ with respect to $x$. 
%\end{Remark}

%%%%%%%%%%%%%%%%%%%%
\subsection{The first order  linearized system}

Next we consider the linearized system
\be\label{MFG2lkjenze}
\left\{ \begin{array}{rl}
(i) & \ds -\partial_t v -{\rm Tr}(a(t,x)D^2v)+H_p(x_0,x,Du,m(t))\cdot Dv+ \frac{\delta H}{\delta m}(x_0,x,Du,m(t))(\rho(t))=R_1(t,x) \\
 & \ds \hspace{10cm} \qquad {\rm in}\; (t_0,T)\times \R^d\\
(ii) & \ds \partial_t \rho -\sum_{i,j} D_{ij}(a_{i,j}\rho) -\dive( \rho H_p(x_0,x,Du,m(t)))-\dive(mH_{pp}Dv)\\
& \ds \hspace{4cm} \qquad -\dive(m \frac{\delta H_p}{\delta m}(\rho))=\dive(R_2(t,x)) \qquad {\rm in}\; (t_0,T)\times \R^d\\
(iii) & \ds \rho(t_0)=\rho_0, \qquad v(T,x)= \frac{\delta G}{\delta m}(x_0,x,m(T))(\rho(T)) +R_3(x) \qquad {\rm in}\;\R^d
\end{array}\right.
\ee
where $(u,m)$ solves \eqref{eq.MFGcoupled} and $H$ and its derivatives are evaluated at $(x_0,x,Du(t,x),m(t))$. In this section, we work under the  conditions given in Proposition \ref{Prop.LipDep} so that \eqref{eq.MFGcoupled} admits a unique solution, in particular we always assume that $T\leq T_M$, where $T_M$ is   given by  Proposition \ref{Prop.LipDep}.   Our goal now is to establish estimates for $(v,\rho)$ in dependence of $G$ and $u$; we implicitly assume that $G(x_0,\cdot, m)$ is sufficiently regular (say, $C_b^n$) so that $u$ inherits the same order of regularity (from \eqref{norman-0}).

The data of equation \eqref{MFG2lkjenze} are $x_0\in \R^{d_0}$, $\rho_0\in C^{-k}$,  $R_1\in C^0([0,T], C^{n-1}_b)$, $R_2\in C^0([0,T], C^{-(k-1)})$ and   $R_3\in C^{n-1}_b$. Here $n\geq 2$ and $k\geq 1$. 
By a solution to \eqref{MFG2lkjenze}, we mean a pair $(v,\rho)$ such that  $v\in C^0([0,T],C^{n-1}_b)$ satisfies \eqref{MFG2lkjenze}-(i) (integrated in time) with terminal condition  $v(T,\cdot)= \frac{\delta G}{\delta m}(x_0,\cdot, m(T))(\rho(T)) +R_3(\cdot)$ and $\rho\in C^0([0,T],  C^{-(k-1)}_b)$ is a solution in the sense of distributions to \eqref{MFG2lkjenze}-(ii) with  initial condition $\rho(t_0)=\rho_0$.

\begin{Proposition}\label{prop.LipschDeltaU} Let us fix $M>0$,   $n\geq 2$ and $k\geq 1$. Under the assumptions of Proposition \ref{Prop.LipDep}, and if 
\be\label{cond.deltaG}
 \|\frac{\delta G}{\delta m}\|_{1;k}\leq M, 
\ee
then there exist  constants $T_M,C_M>0$, depending  on  $M$, $n$, $k$, $\sup_{t\in [0,T]} \|u(t)\|_n$, 
 $\sup_{t\in [0,T]} \|u(t)\|_{k+1}$, such that for $T\leq T_M$   there exists a unique solution $(v,\rho)$ to \eqref{MFG2lkjenze}, and this solution satisfies 
\be\label{lkqehsrmd}
\begin{array}{rl}
& \ds   \sup_{t\in [t_0,T]}\|v(t)\|_{n-1} \; \leq \\
&\;  \ds (1+C_MT)  \|\frac{\delta G}{\delta m} (x_0,\cdot_x,m(T),\cdot_y) \|_{n-1;k}\left( \|\rho_0\|_{-k}+  T\sup_t \|R_2(t)\|_{-(k-1)}+ T\sup_t \|R_1(t)\|_{1}\right)  \\
& \qquad \ds +(1+C_MT)\|R_3\|_{n-1}  +C_MT\left(1+\sup_t \|R_1(t)\|_{n-1}+ \|R_2\|_{-(k-1)}\right) ,  
\end{array}
\ee 
as well as 
\be\label{lkqehsrmd2}
\begin{array}{rl}
\ds 
  \sup_{t\in [t_0,T]} \|\rho(t)\|_{-k}\; \leq & \ds  \left(1+C_MT \right)\|\rho_0\|_{-k}\\
& \ds \qquad +C_MT(\sup_t\|R_1(t)\|_{1}+\sup_t\|R_2(t)\|_{-(k-1)}+ \|R_3\|_{n-1} ).  
\end{array}
\ee
Moreover, we have, for any $r \le n-1$,
 \be\label{lkqehsrmdBISBIS}
 \begin{array}{ll}
\ds  \sup_{t\in [0,T]} \|D^r_x v(t)\|_\infty \leq 
 \ds (1+C_MT)\left(   \Big\|D^r_x \frac{\delta G}{\delta m} (x_0,\cdot,m(T))(\rho(T))\Big\|_\infty  + \|D^r_x R_3\|_\infty\right) \\ \qquad\qquad\qquad\qquad\qquad 
 + C_M T \big(\|\rho_0\|_{-k} + \sup_t\|R(t)\|_{n-1}+\sup_t\|R_2(t)\|_{-(k-1)}+ \|R_3\|_{n-1} \big). 
 \end{array}
\ee 
\end{Proposition}

\begin{proof}  After proving the a priori estimates, the existence of a solution can be obtained  using a continuation argument (see \cite{CDLL} for details). The uniqueness is an obvious consequence of the estimates. So it remains to prove the estimates. To simplify the expression, we omit the dependence of the constant $C$ with respect to $M$. Fix $t_1\in [t_0,T]$, $z_1\in C^k_b$   for $k\in\{ 1, \dots, n-1\}$. Let $z$ be the solution to 
\be\label{eq.zzzzz}
\left\{ \begin{array}{l}
-\partial_t z -{\rm Tr}(a(t,x)D^2z)+H_p(x_0,x,Du,m(t))\cdot Dz=0\qquad {\rm in}\; (t_0,t_1)\times \R^d, \\
z(t_1,\cdot)= z_1(x)\qquad  {\rm in}\; \R^d.
\end{array}\right. 
\ee
According to Proposition \ref{prop.highS} (with $k=1$), we have
$$
\sup_{t\in [t_0,t_1]}\| z(t)\|_{k}\leq (1+CT)\|z_1\|_{k},
$$
where $C$ depends on the regularity of $a$ and $H$ and on  $\sup_t \|u(t)\|_{k+1}$. 
Then, by duality, 
$$
\begin{array}{rl}
\ds \int_{\R^d} z_1\rho(t_1) \; = & \ds \int_{\R^d} z(t_0)\rho_0 
- \int_{t_0}^{t_1} \int_{\R^d} (H_{pp}Dv\cdot Dz +\frac{\delta H_p}{\delta m}(\rho)\cdot Dz)  m - \int_{t_0}^{t_1} \int_{\R^d} Dz\cdot R_2\\
\leq & \ds \|z({t_0})\|_{k}\|\rho_0\|_{-k} + C\|Dz\|_\infty \left(T\|Dv\|_\infty +\int_{t_0}^{t_1} \|\rho(t)\|_{-k}\right) +T \sup_t \|z(t)\|_k\|R_2\|_{-(k-1)}
\\
\leq & \ds (1+ C T)\|z_1\|_{k} \left( \|\rho_0\|_{-k} + C \left(T\|Dv\|_\infty +\int_{t_0}^{t_1} \|\rho(t)\|_{-k}\right) +T \|R_2\|_{-(k-1)}\right) ,
\end{array}
$$
where $\|R_2\|_{-(k-1)}:= \sup_t \|R_2(t)\|_{-(k-1)}$. 
Thus, taking the supremum over $\|z_1\|_{k}\leq 1$, we obtain: 
$$
\|\rho(t_1)\|_{-k}\leq (1+CT) \|\rho_0\|_{-k}+ CT\big(\|Dv\|_\infty +\|R_2\|_{-(k-1)} \big)+C \int_{t_0}^{t_1} \|\rho(t)\|_{-k}.
$$
Since this holds for all $t_1\in ({t_0},T]$,  by Gronwall's inequality we obtain
\be\label{EstiRhoByDv}
\sup_{t\in [{t_0},T]} \|\rho(t)\|_{-k}\leq (1+CT)\|\rho_0\|_{-k}+ CT\big(\|Dv\|_\infty +\|R_2\|_{-(k-1)} \big).
\ee
Next we apply Proposition \ref{prop.highS} (with $k=1$) to the HJ equation satisfied by $v$:  
we have, for any $r\leq n-1$,   
\be\label{ihkbrgd1}
\begin{array}{rl}
\ds   \sup_t \|v(t)\|_r \; \leq & \ds  (1+CT)\|v(T)\|_r  + CTC_1,
\end{array}
\ee
where $C$ depends on  $\sup_t \|a(t)\|_{n-1}$, on the regularity of $H$, on $\sup_t\|u(t)\|_n$, and where $C_1$ is estimated by 
\be\label{ihkbrgd2}
C_1= \sup_t \|\frac{\delta H}{\delta m}(x_0,\cdot,Du(t,\cdot),m(t))(\rho(t))\|_{n-1}+ \|R_1\|_{n-1} \leq C\sup_t\|\rho(t)\|_{-k}+ \|R_1\|_{n-1}\,,
\ee
where we used the inequality 
$$
\|\frac{\delta H}{\delta m}(x_0,\cdot,Du(t,\cdot),m(t))(\rho(t))\|_{n-1} \leq \|\frac{\delta H}{\delta m}(x_0,\cdot,Du(t,\cdot_x),m(t),\cdot)\|_{n-1,k} \|\rho(t)\|_{-k} \,.
$$
 Again we notice here that the right-hand side is estimated through the regularity of $H$ and   $\sup_t \|u(t)\|_n$.
Similarly we estimate, for $r\leq n-1$,  
\be\label{ihkbrgd3}
\begin{array}{rl}
\ds \|v(T)\|_{r} \le \ds   \|\frac{\delta G}{\delta m} (x_0,\cdot,m(T)) \|_{r; k}\sup_t\|\rho(t)\|_{-k} +\|R_3\|_{r}.
\end{array}
\ee
%{\color{green} $$
%\begin{array}{rl}
%\ds \|v(T)\|_{n-1} = \sum_{|\alpha| \le n-1} \|\partial_x^\alpha v(T)\|_\infty \leq  & \ds  \sum_{|\alpha| \le n-1} \|\partial_x^\alpha \frac{\delta G}{\delta m} (x_0,\cdot,m(T),\cdot) \|_{0; k}\|\rho(T)\|_{-k} +\| \partial_x^\alpha R_3\|_\infty\\
%\m
%= & \ds \|\frac{\delta G}{\delta m} (x_0,\cdot,m(T),\cdot) \|_{n-1; k}\sup_t\|\rho(t)\|_{-k} +\|R_3\|_{n-1}.
%\end{array}
%$$} 
Collecting the estimates in \eqref{EstiRhoByDv}, \eqref{ihkbrgd1}, \eqref{ihkbrgd2}, \eqref{ihkbrgd3},   we find, for $r\leq n-1$:
\begin{align}
\ds  \sup_t \|v(t)\|_{r}  \leq & \ds  (1+CT)\|\frac{\delta G}{\delta m} (x_0,\cdot,m(T),\cdot) \|_{r;k} \Bigl\{ (1+CT)\|\rho_0\|_{-k}+ CT(\|Dv\|_\infty+\|R_2\|_{-(k-1)})\Bigr\} \label{neu}\\
& \;  \ds +\|R_3\|_{r} (1+CT) +CT\big(\|\rho_0\|_{-k}+ T(\|Dv\|_\infty+\|R_2\|_{-(k-1)})+ \|R_1\|_{n-1}\big) . \notag
\end{align}
 We first consider this inequality for $r=1$. Recall that $\|\frac{\delta G}{\delta m}\|_{1;k}\leq M$. So, if we choose $T_M>0$ such that
$$
(1+CT_M)M CT_M+ CT^2_M<1,
$$
we obtain \eqref{lkqehsrmd} for $T\leq T_M$ and $n=2$.
Then  from \eqref{EstiRhoByDv} we infer \eqref{lkqehsrmd2} (with a constant only depending on $\sup_t \|u(t)\|_{k+1}$).  Having now estimated $\|Dv\|_\infty$, we deduce from \eqref{neu} that  \eqref{lkqehsrmd} holds.

To obtain \eqref{lkqehsrmdBISBIS}, we apply again Proposition \ref{prop.highS} to the HJ equation satisfied by $v$, together with estimates \eqref{ihkbrgd2} and \eqref{lkqehsrmd2}.
\end{proof}

%%%%%%%%%%%%%%%%%%%%
\subsection{The second order  linearized system}

Next we study the second order linearization of the MFG system. Given $(u,m)$ a solution to \eqref{eq.MFGcoupled} and $(v,\rho)$ and $(v',\rho')$ two solutions to \eqref{MFG2lkjenze} with arbitrary $R_1, R_2, R_3$ and $R_1', R_2', R_3'$, we consider a solution $(w,\mu)$ to 
\be\label{MFG3lkjenze}
\left\{ \begin{array}{rl}
(i) & \ds -\partial_t w -{\rm Tr}(aD^2w)+H_p\cdot Dw+ \frac{\delta H}{\delta m}(\mu(t))\\
& \ds \qquad   +  \frac{\delta^2 H}{\delta m^2}(\rho(t),\rho'(t)) +   H_{pp} Dv\cdot Dv'
+ \frac{\delta H_p}{\delta m}(\rho)\cdot Dv' + \frac{\delta H_p}{\delta m}(\rho')\cdot Dv =\tilde R_1(t,x) \\
& \ds \hspace{8cm} \qquad {\rm in}\; (t_0,T)\times \R^d\\
(ii) & \ds \partial_t \mu -\sum_{i,j} D_{ij}(a_{i,j}\mu) -\dive( \mu H_p)-\dive(mH_{pp}Dw) -\dive(m \frac{\delta H_p}{\delta m}(\mu))-\dive( \rho H_{pp}Dv')\\
& \ds  \qquad -\dive( \rho \frac{\delta H_p}{\delta m}(\rho')) -\dive(\rho'H_{pp}Dv) -\dive(mH_{ppp}DvDv')-\dive(m\frac{\delta H_{pp}}{\delta m}(\rho')Dv)  \\
& \ds  \qquad  -\dive(\rho' \frac{\delta H_p}{\delta m}(\rho))-\dive(m \frac{\delta H_{pp}}{\delta m}(\rho)Dv')-\dive(m \frac{\delta^2 H_p}{\delta m^2}(\rho,\rho'))=\dive(\tilde R_2(t,x))\\
& \ds \hspace{8cm} \qquad {\rm in}\; (t_0,T)\times \R^d\\
(iii) & \ds \mu(t_0)=0, \;  w(T,x)= \frac{\delta^2 G}{\delta m^2}(x_0,x,m(T))(\rho(T),\rho'(T))\\
& \ds \qquad \qquad \qquad \qquad \qquad +\frac{\delta G}{\delta m}(x_0,x,m(T))(\mu(T))+\tilde R_3(x) \qquad {\rm in}\;\R^d
\end{array}\right.
\ee 
where $H$ and its derivatives are  evaluated at $(x_0,x,Du(t,x), m(t))$.   Here again we work under the conditions 
assumed in the previous Sections which guarantee the existence, uniqueness  and enough regularity for $(u,m)$ as well as for the solutions of the linearized system. In particular, we always assume that $T\leq T_M$, where $T_M$ is now   given by  Proposition \ref{prop.LipschDeltaU}. The goal now is to establish estimates for $(w,\mu)$ in terms of $G$ as well as of $(u,m)$ and $(v,\rho)$, $(v',\rho')$.

The data of the problem are $\tilde R_1\in C^0([0,T], C^{n-2}_b)$, $\tilde R_2\in C^0([0,T], C^{-(k-1)})$ and $\tilde R_3\in C^{n-2}_b$.  By a solution to \eqref{MFG3lkjenze}, we mean a pair $(w, \mu)$ such that $w\in C^0([0,T],  C^{n-2}_b)$ satisfies  \eqref{MFG3lkjenze}-(i) (integrated in time) with the terminal condition in \eqref{MFG3lkjenze}-(iii) and $\mu\in C^0([0,T], C^{-k})$ solves \eqref{MFG3lkjenze}-(ii) in the sense of distributions with vanishing  initial condition. 
 Here we assume $n\geq 3$ and $k\geq 2$; the reason for this condition is just because we wish to keep the regularity threshold of $(w,\mu)$ consistent   with what stated previously for $(u,m)$ and for $(v,\rho)$. In general, the estimates below apply to any degree of $k,n$ but this  is obviously a cascade regularity: an estimate of $w$ in $C^{n-2}_b$ requires an estimate of $v$ in $C^{n-1}_b$ and of $u$ in $C^n_b$, while an estimate of $\mu$ in $C^{-k}$ requires an estimate of $\rho$ in $C^{-(k-1)}$.

\begin{Proposition}\label{Prop.estiw}  Let us fix $M>0$,   $n\geq 3$ and $k\geq 1$. Under the assumptions of Proposition  \ref{Prop.LipDep}, and if \eqref{cond.deltaG} holds,  there exist  $T_M>0$, depending on  $M$ and the regularity of $H$, such that for any $T\in (0,T_M]$, system \eqref{MFG3lkjenze}  has a unique solution which satisfies 
\be\label{ilauzehrd}
\begin{array}{rl}
\ds & \sup_t \|w(t)\|_{n-2} \\
& \qquad   \leq    \ds(1+C_MT)\Bigl( \|\frac{\delta^2 G}{\delta m^2} (x_0,\cdot,m(T), \cdot, \cdot) \|_{n-2;k-1,k-1}\|\rho(T)\|_{-(k-1)}\|\rho'(T)\|_{-(k-1)}
+\|\tilde R_3\|_{n-2}
\Bigr)\\
& \ds \qquad + C_MT(1+ \|\frac{\delta  G}{\delta m}\|_{n-2,k}) \left( \sup_t\|\tilde R_1(t)\|_{n-2}+\sup_t \|\tilde R_2(t)\|_{-(k-1)} \right.
\\
& \ds \qquad \qquad \qquad \qquad \qquad \qquad\qquad \qquad  +  {\mathcal R_{k-1,k}}{\mathcal R}'_{k-1,k}+ {\mathcal R_{k-1,n-1}}{\mathcal R}'_{k-1,n-1} \Bigr)
\end{array}
\ee
for some $C_M$ depending on $M$, on the regularity of $H$ as well as on $n,k, \sup_{t\in [0,T]} \|u\|_{n-1}, \sup_{t\in [0,T]} \|u\|_{k+1}$, and
\begin{align}
\ds & \sup_t \|\mu(t)\|_{-k}   \ds  \leq \tilde C_MT\Bigl(\Bigl(1+\|\frac{\delta^2 G}{\delta m^2} (x_0,\cdot,m(T), \cdot, \cdot) \|_{1;k-1,k-1}\Bigr)\|\rho(T)\|_{-(k-1)}\|\rho'(T)\|_{-(k-1)}\notag \\
& \qquad \ds +\sup_{t\in [0,T]}\|\tilde R_1(t)\|_{1} +\sup_{t\in [0,T]}\|\tilde R_2(t)\|_{-(k-1)}+\|\tilde R_3\|_{1} + {\mathcal R_{k-1,k}}{\mathcal R}'_{k-1,k}+  {\mathcal R_{k-1,2}}{\mathcal R}'_{k-1,2}\Bigr) \ ,\label{ilauzehrd2}
\end{align}
where $\tilde C_M$ depends on $M$, the regularity of $H$, $n,k,\sup_{t\in [0,T]} \|u\|_{k+1}$, and where we have set, for $k, j\geq 1$:
\[
\mathcal R_{k-1,j} := \sup_t ( \|\rho(t)\|_{-(k-1)} +  \|v(t)\|_{j})\,,\qquad {\mathcal R}_{k-1,j} ':=\sup_t ( \|\rho'(t)\|_{-(k-1)} +  \|v'(t)\|_{j})\,.
\]
In addition, if $$\|\frac{\delta  G}{\delta m}\|_{n-2;k} \le M, $$  then we have, for any $r \le n-2$, $(t, x_0) \in [0, T] \times \R^{d_0}$,
\begin{align}
& \|D^r w(t, \cdot)\|_\infty    \leq    %(1+C_MT)  
\Bigl(  \Big\| D_x^r \frac{\delta^2 G}{\delta m^2}(x_0,\cdot,m(T))(\rho(T),\rho'(T)) \Big\|_\infty + \|D_x^r \tilde R_3(\cdot)\|_\infty \Bigr)\notag\\ \label{ilauzehrdBISBIS}
& \qquad + C_MT \Bigl( \|\frac{\delta^2 G}{\delta m^2}(x_0,\cdot, m(T),\cdot, \cdot)\|_{n-2;k-1,k-1} \|\rho(T)\|_{ -(k-1)} \|\rho'(T)\|_{ -(k-1)}  \\ 
& \qquad\qquad +\sup_t\|\tilde R_1(t)\|_{n-2}+\sup_t \|\tilde R_2(t)\|_{-(k-1)} + \|\tilde R_3\|_{n-2} + {\mathcal R_{k-1,k}}{\mathcal R}'_{k-1,k}+ {\mathcal R_{k-1,n-1}}{\mathcal R}'_{k-1,n-1} \Bigr) .\notag
\end{align} 
\end{Proposition}

\begin{Remark}\label{rr'}
 We recall that the quantities $\|\rho(T)\|_{-(k-1)}$ and  $\mathcal R_{k-1,j}$ are estimated from \eqref{lkqehsrmd} and \eqref{lkqehsrmd2}. In particular, we have 
$$
\begin{array}{rl}
\ds  \mathcal R_{k-1,k} \; \leq & \ds (1+C_MT)C \left( \|\rho_0\|_{-(k-1)} +  \sup_t \|R_2(t)\|_{-(k-2)}
+\sup_t \|R_1(t)\|_{k}+ \|R_3\|_{k}\right)   \,,
\end{array}
$$
for some constant $C$ depending on $\|\frac{\delta G}{\delta m}\|_{k;k-1}$ and $\sup_t \|u(t)\|_{k+1}$, and similarly
$$
\begin{array}{rl}
\ds   \mathcal R_{k-1,n-1} &  \; \leq  (1+C_MT) C\left( \|\rho_0\|_{-(k-1)}  +  \sup_t \|R_2(t)\|_{-(k-2)}+\|R_3\|_{n-1}+ \sup_t \|R_1(t)\|_{n-1}\right)   
\end{array}
$$
for a constant $C$ depending on $\|\frac{\delta G}{\delta m}\|_{n;k-1}$ and $\sup_t \|u(t)\|_{n}$. Of course the same holds for $\rho',v'$ accordingly.
\end{Remark}

\begin{proof} We  omit the proof of the well-posedness of the system, which is a consequence of the estimates (as for Proposition \ref{prop.LipschDeltaU}). To simplify the expression, we also omit the dependence of the constant $C$ with respect to $M$.
We first estimate $\mu$ by duality. Fix $t_1\in [t_0,T]$, $z_1\in C^k_b$  for $k\in\{ 1,\dots, n-1\}$.  Let $z$ be the solution to \eqref{eq.zzzzz}. Recall that Proposition  \ref{prop.highS} (with $k=1$) implies that there is a constant $C>0$, depending on   $\sup_t \|u(t)\|_{k+1}$, such that 
$$
\sup_{t\in [t_0,t_1]}\| z(t)\|_{k}\leq (1+CT)\|z_1\|_{k}.
$$
Then 
$$
\begin{array}{l}
\ds \int_{\R^d} \mu(t_1)z_1= 
-\Bigl\{\int_{t_0}^{t_1}\int_{\R^d} Dz \cdot \Bigl( mH_{pp}Dw+ m \frac{\delta H_p}{\delta m}(\mu)+\rho H_{pp}Dv' +\rho'H_{pp}Dv\\
\ds \qquad \qquad
+\rho \frac{\delta H_p}{\delta m}(\rho')+\rho' \frac{\delta H_p}{\delta m}(\rho) + mH_{ppp}DvDv' \\
\ds  \qquad\qquad  +m\frac{\delta H_{pp}}{\delta m}(\rho')Dv +m \frac{\delta H_{pp}}{\delta m}(\rho)Dv'+m \frac{\delta^2 H_p}{\delta m^2}(\rho,\rho')+\tilde R_2(t,x)\Bigr) \Bigr\}.
\end{array}
$$
Hence
$$
\begin{array}{l}
\ds \int_{\R^d} \mu(t_1)z_1\leq CT  \|Dw\|_\infty\|Dz\|_\infty+
C\|Dz\|_\infty \int_{t_0}^{t_1} \|\mu(s)\|_{-k}ds \\
\ds \qquad   + CT\Bigl(  \sup_t \|\rho(t)\|_{-(k-1)}\sup_t\|v'(t)\|_{k}+\sup_t \|\rho'(t)\|_{-(k-1)}\sup_t\|v(t)\|_{k} \Bigr)\sup_t\|z(t)\|_{k} \\
\ds \qquad + CT\Bigl( \sup_t\|\rho(t)\|_{-(k-1)}\sup_t\|\rho'(t)\|_{-(k-1)} \Bigr) \sup_t \|z(t)\|_{k} 
+   CT\, \|Dv\|_\infty\|Dv'\|_\infty \|Dz\|_\infty \\
\qquad   \ds + CT\Bigl( \sup_t \|\rho(t)\|_{-(k-1)} \|Dv'\|_\infty+\sup_t \|\rho'(t)\|_{-(k-1)} \|Dv\|_\infty +\sup_t \|\rho(t)\|_{-(k-1)}\sup_t \|\rho'(t)\|_{-(k-1)}\Bigr)\|Dz\|_\infty \\ 
\qquad   \ds 
+CT\, \|\tilde R_2\|_{-(k-1)}\sup_t \|z(t)\|_k \,,
\end{array}
$$
where the constant $C$ depends on
%\footnote{to our convenience: it depends on $\|H_{pp}\|_\infty,  \|H_{ppp}\|_\infty$, but also on $\|H_{pp}(x_0,\cdot, Du(t,\cdot), m(t))\|_{k-1}$, $\|\frac{\delta H_p}{\delta m}(x_0, \cdot_x, Du(t,\cdot_x), m(t), \cdot_y)\|_{k-1, k-1}$, as well as on 
%$$ 
%\sup_{x_0,x,t} \,\,  \left\{ \|\frac{\delta H_{p}}{\delta m}(x_0, x, Du(t,x), m(t), \cdot_y)\|_{k} + \|\frac{\delta H_{pp}}{\delta m}(x_0, x, Du(t,x), m(t), \cdot_y)\|_{k-1}+  \|\frac{\delta^2 H_p}{\delta m^2}(x_0, x, Du(t,x), m(t), \cdot_y, \cdot_{y'})\|_{k-1, k-1}\right\}\,. 
%$$
%}
 the regularity of the function $H$ and on $\sup_t \|u(t)\|_{k}$. Taking the supremum over $\|z_1\|_k\leq 1$, we infer that:
\begin{align*}
\|\mu(t_1)\|_{-k} & \leq C \int_{t_0}^{t_1} \|\mu(s)\|_{-k}ds+  CT\Bigl\{\|Dw\|_\infty +\|\tilde R_2\|_{-(k-1)}  \\
&\qquad +  \Bigl(\sup_t\|\rho(t)\|_{-(k-1)}+\sup_t \|v(t)\|_{k}\Bigr) \Bigl(\sup_t\|\rho'(t)\|_{-(k-1)}+\sup_t \|v'(t)\|_{k}  \Bigr)\Bigr\}.
\end{align*}
By Gronwall's inequality, we obtain 
\be\label{lkajervlsf}
 \sup_t \|\mu(t)\|_{-k} \; \leq   CT\Bigl\{\|Dw\|_\infty +\sup_t \|\tilde R_2(t)\|_{-(k-1)} + {\mathcal R_{k-1,k}} {\mathcal R'_{k-1,k}} \Bigr\}\,,
\ee
where  $C$ depends on the regularity of the function $H$ and on $\sup_t \|u(t)\|_{k+1}$.
From Proposition \ref{prop.highS} (with $k=1$),
we have  
\be\label{estwgdas}
\sup_t \|w(t)\|_{n-2} \leq (1+CT)\Bigl(\|\frac{\delta^2 G}{\delta m}(\rho(T),\rho'(T))\|_{n-2}+ \|\frac{\delta G}{\delta m}(\mu(T))\|_{n-2} + \|\tilde R_3\|_{n -2}\Bigr)+ CT\sup_t \|f(t)\|_{n-2 },
\ee
where 
$$
\begin{array}{rl}
\ds f(t,x)\; = & \ds \frac{\delta H}{\delta m}(\mu(t))+  \frac{\delta^2 H}{\delta m^2}(\rho(t),\rho'(t)) + H_{pp} Dv\cdot Dv'
+\frac{\delta H_p}{\delta m}(\rho)\cdot Dv' + \frac{\delta H_p}{\delta m}(\rho')\cdot Dv -\tilde R_1(t,x).
\end{array}
$$
We estimate
$$
\begin{array}{rl}
\ds \sup_t \|f(t)\|_{n-2}\; \leq & \ds  \Bigl( \|\frac{\delta H}{\delta m}(x_0,\cdot_x,Du(t,\cdot_x), m(t),\cdot_y)\|_{n-2;k}  \sup_t \|\mu(t)\|_{-k} +\|\tilde R_1\|_{n-2}\\
& \qquad +  C \sup_t ( \|\rho(t)\|_{-(k-1)} +  \|v(t)\|_{n-1})
( \|\rho'(t)\|_{-(k-1)} +  \|v'(t)\|_{n-1})\Bigr)
\end{array}
$$
for a constant $C$ depending
%\footnote{to our reference: it depends on $\|\frac{\delta^2 H}{\delta m^2}(x_0,\cdot_x,Du(t,\cdot_x), m(t),\cdot_y, \cdot_{y'})\|_{n-2,k-1,k-1}$, $\|H_{pp}(x_0, \cdot, Du(t,\cdot), m(t))\|_{n-1}$, $\|\frac{\delta H_p}{\delta m}(x_0,\cdot_x,Du(t,\cdot_x), m(t),\cdot_y)\|_{n-2,k-1}$}
 on the regularity of $H$ and on $\sup_t \|u(t)\|_{n-1}$. 
So  we conclude, using also \eqref{lkajervlsf},
$$
\begin{array}{rl}
\ds \sup_t \|f(t)\|_{n-2}\; \leq & \ds C\, T \Bigl(\|Dw\|_\infty + {\mathcal R}_{k-1,k} {\mathcal R}'_{k-1,k}+\|\tilde R_2\|_{-(k-1)}\Bigr) 
\\
& \qquad +\sup_t \|\tilde R_1\|_{n-2} +  C \, {\mathcal R}_{k-1,n-1} {\mathcal R}'_{k-1,n-1}\,.
\end{array}
$$
Similarly, again from \eqref{lkajervlsf} we get
$$
  \|\frac{\delta G}{\delta m}(\mu(T))\|_{n-2} \le CT\, \|\frac{\delta G}{\delta m}(x_0,\cdot, m(T),\cdot)\|_{n-2;k}  \Bigl(\|Dw\|_\infty + {\mathcal R}_{k-1,k} {\mathcal R}'_{k-1,k}+\|\tilde R_2\|_{-(k-1)}\Bigr)
$$
and 
$$
\begin{array}{rl}
\ds \|\frac{\delta^2 G}{\delta m^2}(\rho(T),\rho'(T))\|_{n-2} & \ds \leq \|\frac{\delta^2 G}{\delta m^2}(x_0,\cdot, m(T),\cdot, \cdot)\|_{n-2;k-1,k-1} \|\rho(T)\|_{ -(k-1)} \|\rho'(T)\|_{ -(k-1)}\,.
\end{array}
$$
Then, we find
$$
\begin{array}{rl}
\ds \sup_t \|w(t)\|_{n-2}\; \leq & \ds (1+CT)\left(\|\frac{\delta^2 G}{\delta m^2}(x_0,\cdot,m(T), \cdot, \cdot) \|_{n-2;k-1,k-1}  \|\rho(T)\|_{ -(k-1)} \|\rho'(T)\|_{ -(k-1)}
+\|\tilde R_3\|_{n-2}\right) \\
& \ds + CT \, (1+\|\frac{\delta  G}{\delta m }\|_{n-2;k} )\left(\|Dw\|_\infty +\|\tilde R_2\|_{-(k-1)} + {\mathcal R}_{k-1,k} {\mathcal R}'_{k-1,k} \right) \\
& \ds  \qquad + CT\, \left(  \|\tilde R_1\|_{n-2}+  {\mathcal R}_{k-1,n-1} {\mathcal R}'_{k-1,n-1}\right) 
\end{array}
$$
where now the constant $C$ depends on both $ \sup_t \|u(t)\|_{k+1}$ and $\sup_t \|u(t)\|_{n-1}$.

For $n=3$, if we choose $T$  small enough (depending on $\|\frac{\delta  G}{\delta m }\|_{1,k} $ and  $\sup_t \|u(t)\|_{2}$) we estimate $\|Dw\|_\infty$. Then, plugging this estimate into \eqref{lkajervlsf} gives \eqref{ilauzehrd2} (with a  constant only depending on $\sup_t \|u(t)\|_{k+1}$). Finally, we deduce  \eqref{ilauzehrd} for $n> 3$. 

For any $r\leq n-2$, $x_0\in \R^{d_0}$ and $t\in [0,T]$, the estimate  \eqref{ilauzehrdBISBIS} on $D^r_x w$ follows again from Proposition \ref{prop.highS} (with $k=1$), that gives, arguing as before,
\begin{align*}
& \|D^r_x w(t,\cdot)\|_\infty  \leq (1+CT)\Bigl(\|D^r_x \frac{\delta^2 G}{\delta m^2}(\rho(T),\rho'(T))\|_\infty+ \|D^r_x \frac{\delta G}{\delta m}(\mu(T))\|_\infty + \|D^r_x \tilde R_3\|_\infty\Bigr) \\ 
& \qquad + CT \sup_t \|f(t)\|_{n-2 } \\
& \leq  \Bigl(\|D^r_x \frac{\delta^2 G}{\delta m^2}(\rho(T),\rho'(T))\|_\infty + \|D^r_x \tilde R_3\|_\infty\Bigr)+  (1+CT)  \|\frac{\delta G}{\delta m}(x_0,\cdot, m(t),\cdot)\|_{n-2;k}\sup_t \|\mu(t)\|_{-k}\\
&\qquad    + CT\Bigl(
  \|\tilde R_1\|_{n-2} + \|\tilde R_2\|_{-(k-1)} +\|\tilde R_3\|_{n-2}+ \|\frac{\delta^2 G}{\delta m^2}(x_0,\cdot, m(T),\cdot, \cdot)\|_{n-2;k-1,k-1} \|\rho(T)\|_{ -(k-1)} \|\rho'(T)\|_{  -(k-1)} \\ 
&\qquad\qquad \qquad   +  {\mathcal R}_{k-1,k} {\mathcal R}'_{k-1,k}+  {\mathcal R}_{k-1,n-1} {\mathcal R}'_{k-1,n-1}\Bigr),
\end{align*}
that yields the desired claim using \eqref{ilauzehrd2}. 
\end{proof}

By gathering together  Proposition \ref{prop.LipschDeltaU} and Proposition \ref{Prop.estiw}, we deduce the following three corollaries,  which  will be useful in the derivation of second order estimates for the solution of the master equation.

\begin{Corollary}\label{d2m}
Let $M>0$, $n\geq 3$ and $k\in \{2,\ldots,n-1\}$, and assume that 
$$
\|G\|_n +  \|\frac{\delta G}{\delta m}\|_{n-1;k} + \|\frac{\delta^2 G}{\delta m^2}\|_{n-2;k-1,k-1} \leq M\,.
$$
Let $(u,m)$ be the unique solution to   \eqref{eq.MFGcoupled} in some interval $[0,T_M]$ given by Proposition \ref{Prop.LipDep}, and let $(v,\rho)$ and $(v',\rho')$ be two solutions to \eqref{MFG2lkjenze} with  $R_1=R_2=R_3=0$ and initial conditions $\rho_0$, $\rho_0'$ respectively.  

Then there exists  a constant $C_M$ such that 
the solution  $(w,\mu)$   to  \eqref{MFG3lkjenze} corresponding to $(u,m)$, $(v,\rho)$ and $(v',\rho')$ and with $\tilde R_1=\tilde R_2=\tilde R_3=0$ satisfies, for any $T\in (0,T_M)$, $r \le n-2$: 
$$
 \sup_{t, x} |D_x^r w(t,x)|   \leq    \sup_x \Big|D_x^r \frac{\delta^2 G}{\delta m^2} (x_0,x,m(T))(\rho(T), \rho'(T)) \Big|  + C_M T  \|\rho_0\|_{-(k-1)}\|\rho'_0\|_{-(k-1)}
$$   
where $C_M$ depends on  $M$, as well as on $ \|a\|_n$ and the regularity of $H$. 
\end{Corollary}

\begin{proof} 
We first notice  that 
$$
{\rm Lip}_{0,1}(G)\leq \sup_{x_0,m} \|\frac{\delta G}{\delta m}(x_0,\cdot,m,\cdot)\|_{1,1} \leq M
$$
hence we are in the position to apply Proposition \ref{Prop.LipDep}, and there exists a time $T_M>0$ such that the unique solution $(u,m)$ to \eqref{eq.MFGcoupled}  satisfies $u \in C^n_b$ with an estimate depending on $M$ and $\sup_{x_0, m}\|G(x_0, \cdot, m)\|_n$.

From Proposition \ref{Prop.estiw}, we have
$$
\begin{array}{rl} 
\|D^r_x w(t, \cdot)\|_\infty   & \leq    
 \Big\| D^r_x \frac{\delta^2 G}{\delta m^2}(x_0,\cdot,m(T))(\rho(T),\rho'(T)) \Big\|_\infty  \\
& + C_MT \Bigl(\|\rho(T)\|_{ -(k-1)} \|\rho'(T)\|_{ -(k-1)} + {\mathcal R_{k-1,k}}{\mathcal R}'_{k-1,k}+ {\mathcal R_{k-1,n-1}}{\mathcal R}'_{k-1,n-1} \Bigr) .\notag
\end{array}
$$ 
On the other hand, we know from  Proposition \ref{prop.LipschDeltaU}   that 
$$
\begin{array}{rl}
\ds   & \sup\limits_{t }\|v(t)\|_{n-1} \leq  (1+C_MT) \|\frac{\delta G}{\delta m}\|_{n-1;k-1}\,  \|\rho_0\|_{-(k-1)} \\
\noalign{\medskip}
&   \ds  
  \sup\limits_{t } \|\rho(t)\|_{-(k-1)}  \leq   \ds  \left(1+C_MT \right) \|\rho_0\|_{-(k-1)}\,,  
\end{array}
$$
 which allows us to estimate ${\mathcal R_{k-1,k}}$ and ${\mathcal R_{k-1,n-1}}$.
Here the constant depends on $\sup_t \|u(t)\|_{n}$. A   similar estimate holds for $(v',\rho')$. 
Therefore, we conclude 
 the desired estimate.
\end{proof}

\begin{Corollary}\label{d2mx0}
Under the assumptions of Corollary \ref{d2m},  suppose in addition that
$$
\|D_{x_0}\frac{\delta G}{\delta m}\|_{n-2; k-1} \le M .
$$ 
Let $(u,m)$ be the unique solution to   \eqref{eq.MFGcoupled} in  $[0,T_M]$, let $(v,\rho)$ be a solution  to \eqref{MFG2lkjenze} with  $R_1=R_2=R_3=0$ and initial condition  $\rho_0$, and, for any $|l|=1$, $l\in \R^{d_0}$, $(v^l,\rho^l)$ be a solution  to \eqref{MFG2lkjenze} with zero initial condition and with 
\be\label{R}\begin{split} 
& R_1(t,x)=- \partial^l_{x_0}H(y_0,x,Du(t,x),m(t))\\
& R_2(t,x)=m(t,x) \partial^l_{x_0}H_{p}(y_0,x,Du(t,x),m(t))\\
& R_3(t,x) = \partial^l_{x_0} G(y_0,x,m(T)).
\end{split}
\ee
Then there exists  a constant $C_M$ such that 
the solution  $(w^l,\mu^l)$   to  \eqref{MFG3lkjenze} corresponding to $(u,m)$, $(v,\rho)$ and $(v^l,\rho^l)$ and with 
\be\label{tildeR}
\begin{split}
& \tilde R_1(t,x) = -\partial^l_{x_0} H_{p}(x_0,x,Du,m(t)) Dv - \partial^l_{x_0}\frac{\delta H}{\delta m}(x_0,x,Du,m(t))(\rho(t)), \\
& \tilde R_{2}(t,x) =  \rho \partial^l_{x_0} H_{p} (x_0,x,Du,m(t))+ m\partial^l_{x_0} H_{pp}(x_0,x,Du,m(t)) Dv + m \partial^l_{x_0} \frac{\delta H_{p}}{\delta m}(\rho) , \\
& \tilde R_3(x) = \partial^l_{x_0} \frac{\delta G}{\delta m}(x_0,x,m(T))(\rho(T)),
\end{split}
\ee 
satisfies, for any $T\in (0,T_M)$, $r \le n-2$,
$$
\sup_{t,x}\Bigl(\sum_{|l|=1} |D_x^r w^l(t, x)|^2\Bigr)^{1/2}   \leq    \ds \sup_x 
 \Big|D_x^rD_{x_0} \frac{\delta G}{\delta m}(x_0,x,m(T))(\rho(T)) \Big|  + C_MT \|\rho_0\|_{-(k-1)},
 $$
where $C_M$ depends on  $M$, as well as on $ \|a\|_n$ and the regularity of $H$. 
\end{Corollary}

\begin{proof} 
We first notice that 
$$
\sup_t\|\tilde R_1(t)\|_{n-2} + \sup_t \|\tilde R_2(t)\|_{-(k-1)} \leq C \sup_t \, \left(\|v(t) \|_{n-1}+ \|\rho(t)\|_{-(k-1)} \right) 
$$
for a constant depending on the regularity of $H$, on $\sup_t \|u(t)\|_{n-1}$ and on $\sup_t \|u(t)\|_{k}$. However, the latter term is bounded by $\sup_t \|u(t)\|_{n-1}$ since $k\leq n-1$. Next we estimate the terms $(v,\rho)$, $(v^l,\rho^l)$ and $\mu^l$: we have,  from  Proposition \ref{prop.LipschDeltaU} and Proposition \ref{Prop.estiw}, 
$$
\begin{array}{rl}
\ds   & \sup\limits_{t }\|v(t)\|_{n-1} \leq  (1+C_MT) \|\frac{\delta G}{\delta m}\|_{n-1;k-1}\,  \|\rho_0\|_{-(k-1)} \leq C_M  \|\rho_0\|_{-(k-1)}  \\
\noalign{\medskip}
&   \ds  
  \sup\limits_{t } \|\rho(t)\|_{-(k-1)}  \leq   \ds  \left(1+C_MT \right) \|\rho_0\|_{-(k-1)}\,,  
\end{array}
$$
%for some  constant depending on  $\sup_t \|u(t)\|_{n}$ and 
and
$$
\begin{array}{rl}
\ds  & \sup_{t}\|v^l(t)\|_{n-1} \; \leq  \ds \left\{(1+C_MT) \|\frac{\delta G}{\delta m}\|_{n-1;k-1} +  C_MT \right\} \leq C_M,\\
& \qquad \ds    
 \sup_{t} \|\rho^l(t)\|_{-(k-1)}\; \leq  \ds   C_MT, \qquad \|\mu^l\|_{-k}\leq C_MT\|\rho_0\|_{-(k-1)}. 
\end{array}
$$
%for some constant depending on $\sup_{x_0,m} \|G_{x_0}(x_0,\cdot ,m)\|_{n-1}$ and   $\sup_t \|u(t)\|_{n}$.
We note that the $w^l$ solve linear equations with the same diffusion and the same drift. So, combining  Proposition \ref{prop.highS} with the inequalities above and arguing as in the proof of Proposition \ref{Prop.estiw} gives, for any $r\leq n-2$,  
\begin{align*}
& \sup_x \Bigl(\sum_{|l|=1} |D_x^r w^l(t, x)|^2\Bigr)^{1/2} \leq \\
&  (1+CT) \sup_x \Bigl(\sum_{|l|=1} \Bigl(  | D^r_x \frac{\delta^2 G}{\delta m^2}(\rho(T),\rho^l(T)) | +| D^r_x \frac{\delta G}{\delta m}(\mu^l(T))|+ |D^r_x\partial^l_{x_0} \frac{\delta G}{\delta m}(\rho^l(T))|\Bigr)^2\Bigl)^{1/2} +C_MT\|\rho_0\|_{-(k-1)} \\
& \leq \sup_x \Bigl(\sum_{|l|=1} \Bigl(  |D^r_x\partial^l_{x_0} \frac{\delta G}{\delta m}(\rho^l(T))|+C_MT\|\rho_0\|_{-(k-1)}\Bigr)^2\Bigl)^{1/2} +C_MT\|\rho_0\|_{-(k-1)},
\end{align*} 
where we have omitted the dependence of $G$ with respect to $(x_0,x,m(T))$. This gives the result. 
\end{proof}

\begin{Corollary}\label{d2x0}
Under the assumptions of Corollary \ref{d2mx0}, suppose in addition that
$
\|D^2_{x_0} G(x_0,\cdot,m)\|_{n-2} \le M .
$ 
Fix $l,l'\in \N^{d_0}$ with $|l|=|l'|=1$. Let $(u,m)$ be the unique solution to   \eqref{eq.MFGcoupled} in  $[0,T_M]$ and  let $(v^l,\rho^l)$, $(v^{l'},\rho^{l'})$ be the solution  to \eqref{MFG2lkjenze} with  zero initial condition  and with $R_1,R_2,R_3$ and  $R_1',R_2',R_3'$ given by \eqref{R} for $l$ and $l'$ respectively.

Let   $(w^{l,l'},\mu^{l,l'})$   be the solution to  \eqref{MFG3lkjenze} corresponding to $(u,m)$, $(v^l,\rho^l)$ and $(v^{l'},\rho^{l'})$ and with 
\be\label{tilde2}
 \begin{split}
& \tilde R_1^{l,l'}(t,x)= - \Big( \partial^{l+l'}_{x_0}H + \partial^{l}_{x_0}H_{p} Dv^{l'} +  \partial^{l'}_{x_0}H_{p} Dv^l + 
\partial^{l}_{x_0} \frac{\delta H}{\delta m}(\rho^{l'}(t)) +  \partial^{l'}_{x_0} \frac{\delta H}{\delta m}(\rho^l(t))  \Big) \\
& \tilde R_2^{l,l'}(t,x)= \rho^{l'} \partial^{l}_{x_0}H_p + \rho^{l}  \partial^{l'}_{x_0}H_p + m (\partial^{l}_{x_0}H_{pp}Dv^{l'}+ \partial^{l'}_{x_0}H_{pp}Dv^{l})\\
& \qquad \qquad\qquad  +m (\partial^{l}_{x_0}\frac{\delta H_p}{\delta m}(\rho^{l'})+  \partial^{l'}_{x_0}\frac{\delta H_p}{\delta m}(\rho^{l}))
+ m \partial^{l+l'}_{x_0} H_p\\
& \tilde R_3^{l,l'}(t,x) =\partial^{l+l'}_{x_0}  G(x_0,x,m(T)) + D^{l}_{x_0}\frac{\delta G}{\delta m}(x_0,x,m(T))(\rho^{l'}(T))+
 \partial^{l'}_{x_0}\frac{\delta G}{\delta m}(x_0,x,m(T))(\rho^l(T))\ ,
\end{split}
\ee
where $H$ and its derivatives are computed at $(x_0,x,Du(t,x),m(t))$. Then there exists  a constant $C_M$ such that, for any $T\in (0,T_M)$, $r \le n-2$: 
$$
\sup_{t,x} \Bigl(\sum_{l,l'} |D_x^r w^{l,l'}(t, x)|^2\Bigr)^{1/2} \;  \leq   \sup_x  | D_x^r  D^2_{x_0}G(x_0, \cdot,m(T)) |  + C_MT, 
$$   
where $C_M$ depends on  $M$, as well as on $ \|a\|_n$ and the regularity of $H$. 
\end{Corollary}

\begin{proof} We can estimate $(v^l,\rho^l)$ and $(v^{l'},\rho^{l'})$ and $\mu^{l,l'}$---and therefore  $\tilde R_1^{l,l'}$ and  $\tilde R_2^{l,l'}$---exactly as in the previous Corollary. Moreover, as the $w^{l,l'}$ solve a HJ with the same diffusion and the same drift term, we can use Proposition \ref{prop.highS} to bound the sum $(\sum_{l,l'} |D_x^r w^{l,l'}(t, x)|^2)^{1/2}$: 
$$
\sup_{t,x} \Bigl(\sum_{l,l'} |D_x^r w^{l,l'}(t, x)|^2\Bigr)^{1/2} \;  \leq    \sup_x \Bigl(\sum_{l,l'} ( | D_x^r  \partial^{l+l'}_{x_0}G(x_0, \cdot,m(T)) |+C_MT)^2\Bigr)^{1/2} + C_MT,
$$
which gives the required estimate after rearranging. 
\end{proof}

%%%%%%%%%%%%%%%%%%%%%%%%%%%%%%%%%
%%%%%%%%%%%%%%%%%%%%%%%%%%%%%%%%%
%%%%%%%%%%%%%%%%%%%%%%%%%
\section{Estimates on the first order master equation}\label{Sec.FirstOrdreMaster}

In this section, we complete our program by proving regularity results for the solutions of the various first order master equations encountered in the previous sections. We mainly consider the first order master equation:
\be\label{eq.Master1}
\left\{\begin{array}{l} 
\ds - \partial_t U(t,x_0,x,m)  - {\rm Tr}(a(t,x)D^2_{xx}U(t,x_0,x,m)) +H(x_0,x,D_xU(t,x_0,x,m),m)\\
\ds    \qquad  -\int_{\R^d}{\rm Tr}(a(t,y)D^2_{ym}U(t,x_0,x,m,y))\ m(dy) \\
\ds    \qquad   + \int_{\R^d} D_m U(t,x_0,x,m,y)\cdot H_p(x_0,y,D_xU(t,x_0,y,m),m) \ m(dy) =0 
%\\
% \ds \qquad \qquad  
 \; {\rm in }\; (0,T)\times \R^d\times \Pw\\
 U(T,x_0,x,m)= G(x_0,x,m) \qquad  {\rm in }\;  \R^d\times \Pw\,.\\
\end{array}\right.
\ee
In the above equation, $x_0\in \R^{d_0}$ is considered as a parameter. Our aim is to build a solution to this equation and study its regularity. The method for finding a solution to \eqref{eq.Master1} is well-known: if we set 
\be\label{rep.U1}
U(t_0,x_0,x,m_0):= u(t_0,x)
\ee
where $(u,m)$  is the solution to \eqref{eq.MFGcoupled}, then $U$ is a solution to \eqref{eq.Master1}. 

In order to study the Major-Minor agents' problem, we also have to consider a linear master equation 
\be\label{eq.MasterL}
\left\{\begin{array}{l} 
 \ds -\partial_t U^{0}   -\int_{\R^d}{\rm Tr}(a(t,y)D^2_{ym}U^0(t,x_0,m,y))\ m(dy) \\
\ds \qquad + \int_{\R^d} D_mU^0(t,x_0,m,y)\cdot H_p(x_0,y, D_{x}U(t,x_0,y,m), m)dm(y)=0 \\ 
\ds U^0(T,x_0,m) = G^0(x_0,m) \qquad \text{in } \R^d \times \Pw,
\end{array}\right.
\ee
where $U$ is the solution to \eqref{eq.Master1}. In this case, we build the solution $U^0$ by the simple formula:
\be\label{rep.UL}
U^0(t_0,x_0,m_0) = G^0(x_0,m(T)),
\ee
where $(u,m)$ is also the solution to \eqref{eq.MFGcoupled}.

Our aim is to show that, if $G$ and $G^0$ are regular enough, then \eqref{eq.Master1} and \eqref{eq.MasterL} have classical solutions, given by the above representation formulas. Moreover, we show that the regularity of these solutions only deteriorate linearly in time. This last point is the key result in order to build later solutions to the second order master equation and to the master equation for the Major-Minor agents' problem. 

Throughout the section, the assumptions of Subsection \ref{subsec.Hyp} on $a$, $H$, $G$ and $G^0$ are in force. 

%%%%%%%%%%%%%%%%%%%%%%%%%%%%%%%%%
\subsection{First order differentiability of $U$ and $U^0$}

\begin{Proposition}\label{Prop.DerivU} For any $M>0$, there exists $T_M>0$ and $K_M>0$, depending on $C_0$ and $\gamma$ and $\|Da\|_\infty$, and there exists $C_M>0$, depending also on $n$,  $k\in \{2,\dots,n-1\}$, $
\sup_t \|a(t)\|_n$ and the regularity of $H$ such that, if 
\be\label{condkjahzlr}
 \|G\|_n
+ \left\| \frac{\delta G}{\delta m}\right\|_{n-1;k}  \leq M, 
\ee
and if $T\in (0,T_M]$, then  the map $U$ defined by \eqref{rep.U1} is a classical solution to \eqref{eq.Master1}, and  satisfies
$$
\sup_{ t\in [0,T]} \|U(t)\|_n \leq \|G\|_n+ C_MT
$$
Moreover, for any $|\alpha| \le n-1$, $\partial_x^\alpha \frac{\delta U}{\delta m}$ is of class $C^1$ in $m$, and for $k\in \{2, \dots, n-1\}$, 
$$
\sup_{ t\in [0,T]}  \left\|\frac{\delta U}{\delta m}(t)\right\|_{n-1;k} \leq 
 \left\|\frac{\delta G}{\delta m}\right\|_{n-1;k}+ C_MT.
$$
\end{Proposition}

\begin{Remark} We show in the proof the following representation:
\be\label{reformBIS}
\int_{\R^d} \frac{\delta U}{\delta m}(t_0,x_0,x,m_0,y)\rho_0(dy) = v(t_0,x)
\ee
where $(u,m)$ is the solution of the MFG system \eqref{eq.MFGcoupled} and $(v,\rho)$ is the solution of the linearized system \eqref{MFG2lkjenze} with  right-hand side $R_1=R_2=R_3=0$ and with initial condition $(t_0, \rho_0)$. Note that the normalization condition \eqref{eq.convention} is satisfied because, if one chooses $\rho_0=m_0$, then $(v,\rho)=(0,m)$. 
\end{Remark}

The proof relies  on the following lemma, in which we also provide estimates to obtain later one the differentiability of $U$ with respect to $x_0$.

\begin{Lemma}\label{Lem.DerivU} Under the assumptions of Proposition \ref{Prop.DerivU}, we fix $(t_0,m_0), (t_0,m_1)\in [0,T)\times \Pw$, $y_0, \xi \in \R^d$ with $|\xi| \le 1$.  Let $(u,m)$ be the solution to \eqref{eq.MFGcoupled} with $x_0=y_0$ and with initial condition $(t_0,m_0)$,  and, for $h\in (0,1)$, let $(u_h,m_h)$ be the solution to \eqref{eq.MFGcoupled} with $x_0=y_0 + \xi h$ and with initial condition $(t_0, (1-h)m_0+hm_1)$. Let also $(v,\rho)$ be the solution to \eqref{MFG2lkjenze} associated with $(u,m)$, $x_0=y_0$ and with 
\begin{equation}\label{Rchoi} \begin{split}
& R_1(t,x)=- H_{x_0}(y_0,x,Du(t,x),m(t))\cdot \xi \\
& R_2(t,x)=m(t,x) H_{x_0p}(y_0,x,Du(t,x),m(t))\cdot \xi\\
& R_3(t,x) = G_{x_0}(y_0,x,m(T))\cdot\xi,
\end{split} \end{equation} 
and initial condition $(t_0, m_1-m_0)$.  Then there exists  a constant $C$ (independent of $h$) such that
\be\label{lkaejnzred:1}
\sup_{t\in [t_0,T]}\|u_h(t)-u(t)-hv(t)\|_{n-1} \leq  Ch^2 
\ee
and
\be\label{lkaejnzred:2}
\ds \sup_{t\in [t_0,T]} \|m_h(t)-m(t)-h\rho(t)\|_{-k}\; \leq Ch^2\,.
\ee
\end{Lemma}

\begin{Remark} The goal of this Lemma is to identify the first order derivatives $\frac{\delta U}{\delta m}$ and $D_{x_0}U$. The constant $C$ above will depend on the regularity of $H$ and $G$, as well as on $ \sup_{t\in [t_0,T]} \|u(t)\|_n$; however this is not detailed later since it will not be relevant; indeed, \eqref{lkaejnzred:1} and \eqref{lkaejnzred:2} are only used for letting $h\to 0$.
\end{Remark}

\begin{proof} We set 
$$
v_h(t,x)= u_h(t,x)-u(t,x)-hv(t,x), \quad\rho_h(t,x)=m_h(t,x)-m(t,x)-h\rho(t,x).
$$
Then the pair $(v_h,\rho_h)$ solves 
$$
\left\{ \begin{array}{l}
\ds -\partial_t v_h -{\rm Tr}(a(t,x)D^2v_h)+H_p(y_0,x,Du,m(t))\cdot Dv_h+ \frac{\delta H}{\delta m}(y_0,x,Du,m(t))(\rho_h(t))=R_{h,1}(t,x) \\
\ds \hspace{11cm} \qquad {\rm in}\; (t_0,T)\times \R^d\\
\ds \partial_t \rho_h -\sum_{i,j} D_{ij}(a_{i,j}\rho_h) -\dive( \rho_h H_p(x,Du,m(t))-\dive(mH_{pp}(x,Du,m(t))Dv_h)\\
\ds \hspace{4cm} \qquad -\dive(m \frac{\delta H_p}{\delta m}(x,Du,m(t))(\rho_h))=\dive(R_{h,2}(t,x)) \qquad {\rm in}\; (t_0,T)\times \R^d\\
\ds \rho_h(t_0)=0, \qquad v_h(T,x)= \frac{\delta G}{\delta m}(x,m(T))(\rho_h(T)) +R_{h,3}(x) \qquad {\rm in}\;\R^d
\end{array}\right.
$$
where 
\begin{align*}
\ds R_{h,1}(t,x) &= -\Bigl(H( y_0+ \xi h,x,Du_h(t,x),m_h(t))-H(y_0,x,Du(t,x),m(t))\\
& -H_p(y_0,x,Du(t,x),m(t))\cdot D(u_h(t,x)-u(t,x))- \frac{\delta H}{\delta m}(y_0,x,Du(t,x),m(t))(m_h(t)-m(t)) \\
& - h H_{x_0}(y_0,x,Du(t,x),m(t))\cdot \xi\Bigr), \\
\ds R_{h,2}(t,x)&= m_h(t,x)H_p( y_0+ \xi h,x,Du_h(t,x), m_h(t))-m(t,x)H_p(y_0,x,Du(t,x),m(t))\\
& \ds \qquad  -(m_h(t,x)-m(t,x))H_p(y_0,x,Du(t,x),m(t))\\
& \ds \qquad -m(t,x) H_{pp}(y_0,x,Du(t,x),m(t))D(u_h-u)(t,x)\\
& \ds \qquad -h m(t,x) H_{x_0p}(y_0,x,Du(t,x),m(t))\cdot \xi \\
& \ds \qquad - m(t,x)\frac{\delta H_p}{\delta m}(y_0,x,Du(t,x),m(t))(m_h(t)-m(t)), \\
 \ds R_{h,3}(x) &=G( y_0+ \xi h,x,m_h(T))-G(y_0,x,m(T))-\frac{\delta G}{\delta m}(y_0,x,m(T))(m_h(T)-m(T)) \\
& \ds \qquad -h G_{x_0}(y_0,x,m(T))\cdot\xi. 
\end{align*}
Next we estimate $R_{h,1}$, $R_{h,2}$ and $R_{h,3}$. As 
\begin{align*}
R_{h,1} = - \int_0^1 & \Bigl\{ (H_p(x_\tau, x,p_\tau(t,x),m_\tau(t))-H_p(y_0,x, Du(t,x),m(t))) \cdot D(u_h(t,x)-u(t,x))  \\
& +(H_{x_0}(x_\tau,x,p_\tau(t,x),m_\tau(t))-H_{x_0}(y_0, x,  Du(t,x),m(t))) \cdot h\xi  \\
& + \int_{\R^d} (\frac{\delta H}{\delta m}(x_\tau,x,p_\tau(t,x),m_\tau(t),y)- \frac{\delta H}{\delta m}(y_0,x,Du(t,x),m(t),y)) (m_h(t)-m(t))(dy)\Bigr\} \ d\tau
\end{align*}
where $x_\tau:=(1-\tau)y_0 + \tau  (y_0+ \xi h)$, $p_\tau := (1-\tau )Du(t,x)+\tau Du_h(t,x)$ and $m_\tau(t,x):= (1-\tau) m(t,x)+\tau m_h(t,x)$, we have 
$$
\|R_{h,1}(t)\|_{n-1}\leq C\big( \|u_h(t)-u(t)\|^2_n+ h^2 + \dw^2(m_h(t),m(t))\big),
$$
In the same way, 
\begin{align*}
\|R_{h,3}\|_{n-1} & \leq C(\dw^2(m_h(T),m(T)) + h^2)
\\
\m
& \leq C( \|u_h(T)-u(T)\|^2_n+ \dw^2(m_h(T),m(T)) + h^2). 
\end{align*}
Finally, for $k\geq 2$, we have
\begin{align*}
&\|R_{h,2}(t)\|_{-(k-1)}\\
&\qquad = \sup_{\|\phi\|_{k-1}\leq 1} \int_{\R^d} \phi(t,x)\Bigl( H_p(x_0,x,Du_h(t,x), m_h(t))-H_p(y_0,x,Du(t,x),m(t))\Bigr)(m_h(t,dx)-m(t,dx))\\
& \ds \qquad\ + \int_{\R^d} \phi(t,x)\Bigl(H_p(x_0,x,Du_h(t,x),m_h(t))-H_p(y_0,x,Du(t,x),m(t))\\
& \ds \qquad\qquad - H_{x_0p}(y_0,x,Du(t,x),m(t))\cdot h \xi \\
& \ds \qquad\qquad  - H_{pp}(y_0,x,Du(t,x),m(t))D(u_h-u)(t,x) - \frac{\delta H_p}{\delta m}(y_0,x,Du(t,x),m(t))(m_h(t)-m(t)) \Bigr)m(t,dx)\\
& \qquad \leq C(\|u_h-u\|^2_2+\dw^2(m_h(t),m(t)) + h^2).
\end{align*}
By Proposition \ref{prop.LipschDeltaU}, there exist constants $T_M,C_M>0$, depending on  $M$, $n$, $k$, $\sup_{t\in [0,T]} \|u\|_n$, such that, if $T\leq T_M$ and if \eqref{condkjahzlr} holds, 
then 
$$
\begin{array}{rl}
\ds  \sup_{t\in [0,T]}\|v_h(t)\|_{n-1} \; \leq & \ds (1+C_MT)\|R_{h,3}\|_{n-1} 
+C_MT\left(  \sup_t \|R_{h,1}(t)\|_{n-1}+ \sup_t \|R_{h,2}(t)\|_{-(k-1)}\right) \\
\leq & \ds C\left(\sup_t \|u_h(t)-u(t)\|^2_{n}+ \sup_t \dw^2(m_h(t),m(t)) + h^2\right).
\end{array}
$$
We then infer by Proposition \ref{Prop.LipDep} and the definition of $v_h$ that 
$$
\sup_{t\in [t_0,T]}\|u_h(t)-u(t)-hv(t)\|_{n-1} \leq C( \dw^2((1-h)m_0+hm_1,m_0) + h^2) \leq Ch^2. 
$$
The estimate of $\rho_h$ comes from Proposition \ref{prop.LipschDeltaU} in the same way. 
\end{proof}

\begin{proof}[Proof of Proposition \ref{Prop.DerivU}] Proposition \ref{Prop.LipDep} and the representation formula \eqref{rep.U1} imply the estimate on $\|U(t,\cdot,m)\|_n$. Let us now show that  the map $U$ given by \eqref{rep.U1} is  differentiable with respect to $m$. Fix $x_0\in\R^{d_0}$, $(t_0,m_0), (t_0,m_1)\in [0,T)\times \Pw$, let $(u,m)$, $(u_h,m_h)$ and  $(v,\rho)$ be as in Lemma \ref{Lem.DerivU} with $\xi = 0$, so $R_1 = R_2 = R_3 = 0$. Then 
$$
\sup_{t\in [t_0,T]}\|u_h(t)-u(t)-hv(t)\|_{n-1} \leq   o(h). 
$$
Taking $t=t_0$, this implies that 
$$
\left\| U(t_0,x_0, \cdot, (1-h)m_0+hm_1)-U(t_0,x_0, \cdot,m_0)-hv(t_0,\cdot)\right\|_{n-1} \leq o(h). 
$$

So, if we choose $m_1=\delta_y$ for a fixed $y\in \R^d$, we have just proved that the map $\hat U(h; m_0,y)= U(t_0,x_0, x, (1-h)m_0+h\delta_y)$ has a derivative at $h=0$ and that this derivative is given by $v(t_0,x)$. Note that the map $(m_0,y) \mapsto v(t_0,x; m_0,y)$ is continuous and bounded thanks to the estimates in Proposition \ref{prop.LipschDeltaU} and the uniqueness of the solution. So we can apply Lemma \ref{lem.condC1} which says that $U$ is $C^1$ in $m$ with 
$$
v(t_0,x)= \frac{\delta U}{\delta m}(t_0,x_0, x,m_0,y).
$$
Then by linearity and continuity one easily checks that \eqref{reformBIS} and   the  normalization condition \eqref{eq.convention} hold.  A similar argument applies to derivatives of $ \frac{\delta U}{\delta m}$ with respect to $x$.

Next we check that $U$ solves \eqref{eq.Master1}. Let us start with $m(t_0)=m_0$ with a smooth density. Then $(u,m)$ is a classical solution and, as 
$$
U(t,x_0, x,m(t))= u(t,x)\qquad \forall (t,x)\in [t_0,T]\times \R^d, 
$$
we have, for any $h>0$ and in view of the equation for $m$: 
\begin{align*} 
& u(t_0+h,x)-u(t_0,x) = U(t_0+h, x_0,x, m(t_0+h))-U(t_0,x_0,x, m(t_0))\\
 & = \int_{t_0}^{t_0+h} \int_{\R^d} \frac{\delta U}{\delta m}(t_0+h, x_0,x,m(t),y) \partial_t m(t,y)dydt +
U(t_0+h, x_0,x, m(t_0)) -U(t_0,x_0,x, m(t_0))\\
& = - \int_{t_0}^{t_0+h} \int_{\R^d} D_mU(t_0+h, x_0,x,m(t),y)\cdot H_p(x_0,y, D_xu(t,y),m(t)) m(t,y)dydt\\
& \qquad + \int_{t_0}^{t_0+h} \int_{\R^d}{\rm Tr}(a(t,y)D^2_{ym}U(t,x_0,x,m,y))\ m(dy)dt\\
%\int_{\R^d} \dive_y[D_mU](t_0+h, x_0,x,m(t),y)m(t,y)dydt\\ 
& \qquad +
U(t_0+h, x_0,x, m(t_0)) -U(t_0,x_0,x, m(t_0)). 
 \end{align*}
 On the other hand, by the equation for $u$, 
\begin{align*}
& u(t_0+h,x)-u(t_0,x)=  \int_{t_0}^{t_0+h} \Bigl( -{\rm Tr}(a(t,x)D^2u(t,x))+H(x_0,x,Du(t,x),m(t))\Bigr)dt\\
% \int_{t_0}^{t_0+h} \Bigl( -\Delta u(t,x)+H(x_0,x,D_xu(t,x),m(t))\Bigr)dt  \\
 &\qquad  = \int_{t_0}^{t_0+h} \Bigl( -{\rm Tr}(a(t,x)D^2_{xx}U(t,x_0,x,m(t)))+H(x_0,x,D_xU(t,x_0,x,m(t)),m(t))\Bigr)dt.
% & = \int_{t_0}^{t_0+h} \Bigl( -\Delta_x U(t,x_0,x,m(t))+H(x_0,x,D_xU(t,x_0,x,m(t)),m(t))\Bigr)dt.
 \end{align*}
 So 
 \begin{align*} 
& U(t_0+h, x_0,x,  m_0) -U(t_0,x_0,x, m_0)\\
& =  \int_{t_0}^{t_0+h} \int_{\R^d} D_mU(t_0+h, x_0,x,m(t),y)\cdot H_p(x_0,y, D_xU(t,x_0,y,m(t)),m(t)) m(t,y)dydt\\
& \qquad - \int_{t_0}^{t_0+h} \int_{\R^d}{\rm Tr}(a(t,y)D^2_{ym}U(t,x_0,x,m,y))\ m(dy)dt\\
%& \qquad - \int_{t_0}^{t_0+h} \int_{\R^d} \dive_y[D_mU](t_0+h, x_0,x,m(t),y)m(t,y)dydt\\
%& \qquad + \int_{t_0}^{t_0+h} \Bigl( -\Delta_x U(t,x_0,x,m(t))+H(x_0,x,D_xU(t,x_0,x,m(t)),m(t))\Bigr)dt.
 &\qquad  +\int_{t_0}^{t_0+h} \Bigl( -{\rm Tr}(a(t,x)D^2_{xx}U(t,x_0,x,m(t)))+H(x_0,x,D_xU(t,x_0,x,m(t)),m(t))\Bigr)dt.
 \end{align*}
Therefore $U$ has a time-derivative  at $(t_0,x_0,x, m_0)$ and 
 \begin{align*} 
 \partial_t U(t_0,x_0,x, m_0) &= \int_{\R^d} D_mU(t_0, x_0,x,m_0,y)\cdot H_p(x_0,y, D_xU(t_0,x_0,y,m_0),m_0)) m_0(y)dy\\
& \qquad -\int_{\R^d}{\rm Tr}(a(t_0,y)D^2_{ym}U(t_0,x_0,x,m,y))\ m(dy)\\
& \qquad  -{\rm Tr}(a(t_0,x)D^2_{xx}U(t_0,x_0,x,m_0))+H(x_0,x,D_xU(t,x_0,x,m_0),m_0).
 \end{align*}
This shows that $U$ satisfies \eqref{eq.Master1} at any point $(t_0,x_0,x,m_0)$ where $m_0$ has a smooth density. The general case can be treated by a density argument, since the right-hand side of the above equation is continuous in $(t_0,x_0,x,m_0)$. 

Let us now explain the estimates on $\frac{\delta U}{\delta m}$. 
In view of \eqref{condkjahzlr}, \eqref{reformBIS} and  Proposition \ref{prop.LipschDeltaU}, we have, for any $r\leq n-1$, 
\begin{align*}
&  \Bigl\| D_x^r  \frac{\delta U}{\delta m}(t_0,x_0,\cdot,m_0)(\rho_0)\Bigr\|_\infty =  \|D_x^r v(t_0,x_0,\cdot)\|_\infty \\
& \qquad \leq (1+C_MT) \|D_x^r \frac{\delta G}{\delta m}(x_0, \cdot, m(T),\cdot) \|_{0;k}\|\rho_0\|_{-k} +C_MT\|\rho_0\|_{-k}.
\end{align*}
Taking the sup over $\rho_0$ with $\|\rho_0\|_{-k}\leq 1$, $x_0\in \R^{d_0}$, summing over $r\leq n-1$  and then taking the sup over $t, m$ gives the estimate on  
$\frac{\delta U}{\delta m}$.  Notice that the estimate given by Proposition \ref{prop.LipschDeltaU} depends on $\sup_t\|u(t)\|_{n}$ (we use here that $k\leq n-1$); but  this latter term is estimated in terms  of  $M$ only, because of  Proposition \ref{Prop.LipDep} and since 
$\|G(x_0,\cdot,m)\|_n\leq M$.
\end{proof}

\begin{Proposition}\label{Prop.DerivUL}  Under the assumptions of  Proposition \ref{Prop.DerivU}, let   $M, T_M, C_M>0$ be given accordingly. Assume, in addition, that $T\in (0,T_M]$ and 
\be\label{hyphypGG0}
  \sup_{x_0, m}  \left|G^0(x_0 ,m)\right| + \left|D_{x_0} G^0(x_0 ,m)\right|+ \left\| \frac{\delta G^0}{\delta m}(x_0, m,\cdot)\right\|_{n-1; k}+  \|D_{x_0} G (x_0, \cdot, m)\|_{n-1}\leq M. 
\ee
Then, the map $U^0$ defined by \eqref{rep.UL} is a classical solution to \eqref{eq.MasterL}. In addition, $U^0$ and $U$ are differentiable with respect to $x_0$ and  satisfy
\begin{align}
\sup_{t} \Bigl\|(U^0,U)(t)\Bigr\|_n \le & \ \Bigl\|(G^0,G)\Bigr\|_n + C_MT.  \label{kjernzd} \\
\sup_{t} \Bigl\|D_{x_0} (U^0, U)(t)\Bigr\|_{n-1} \le & \  \Bigl\|(D_{x_0}G^0,D_{x_0}G)\Bigr\|_{n-1} + C_MT.  \label{kjernzd2}
\end{align} 
and
\begin{align}  \label{kjernzd3}
\sup_{t}  \Bigl\|\frac{\delta (U^0,U)}{\delta m}(t)\Bigr\|_{n-1;k} \le \Bigl\|\frac{\delta (G^0,G)}{\delta m}\Bigr\|_{n-1;k} + C_M T .
\end{align}
\end{Proposition}

 As we will see in the proof, it is possible to estimate $U^0$ and $U$ separately. However we will need the specific form of the estimate in the analysis of the MFG problem with a major player. 

\begin{proof}  
Differentiability of $U$ with respect to $x_0$ can be checked as for its differentiability with respect to $m$: let $\xi$ be any unit vector of $\R^{d_0}$,  $(u,m)$, $(u_h,m_h)$ and  $(v,\rho)$ be as in Lemma \ref{Lem.DerivU} with $m_1 = m_0$. Then,  by Proposition \ref{prop.LipschDeltaU}  and the fact that
\begin{equation}\label{r1r2r3}
\sup_t\|R_1(t)\|_{n-1}+\sup_t\|R_2(t)\|_{-(k-1)} \le C, \qquad \|R_3\|_{n-1} \le \sup_{x_0,m}\|G_{x_0}(x_0,\cdot,m(T))\|_{n-1}
\end{equation}
one has 
$$
\left\| U(t_0,x_0 + h\xi,\cdot,  m_0)-U(t_0,x_0,\cdot, m_0)-hv(t_0,\cdot)\right\|_{n-1} \leq o(h),
$$
and so 
\be\label{Dx0U}
U_{x_0}(t_0,x_0,x,m_0) \cdot \xi= v(t_0,x).
\ee

To show the differentiability of $U^0$ with respect to $m$ we proceed as in the proof of Proposition \ref{Prop.DerivU}. Fix $x_0\in\R^d$, $(t_0,m_0), (t_0,m_1)\in [0,T)\times \Pw$, let $(u,m)$, $(u_h,m_h)$ and  $(v,\rho)$ be as in Lemma \ref{Lem.DerivU} with $\xi = 0$, so $R_1 = R_2 = R_3 = 0$. Then 
$$
\sup_{t\in [t_0,T]}\|\rho_h(t)\|_{-k} \leq   o(h),
$$
where $\rho_h(t,x)=m_h(t,x)-m(t,x)-h\rho(t,x)$. This inequality and Proposition \ref{Prop.LipDep} imply
\begin{multline}\label{eq1g0}
\Big| G^0(x_0,m_h(T))- G^0(x_0,m(T)) - h \frac{\delta G^0}{\delta m}(x_0,m(T))(\rho(T)) \Big| \le \Big|\frac{\delta G^0}{\delta m}(x_0,m(T))(\rho_h(T))\Big| \\
+ \Big| \int_0^1 \int_{\R^d} \left(\frac{\delta G^0}{\delta m}(x_0,(1-\tau)m(T)+\tau m_h(T),y) - \frac{\delta G^0}{\delta m}(x_0,m(T),y)\right) (m_h(t)-m(t))(dy) \ d\tau\Big| \\ \le
 o( \dw(m_h(T),m(T)) + h) \le o(h).
\end{multline}
For $y \in \R^d$ choose now $m_1 = \delta_y$, then
$$
\Big| U^0(t_0,x_0, (1-h)m_0+h\delta_y)-U^0(t_0,x_0,m_0)-h \frac{\delta G^0}{\delta m}(x_0,m(T))(\rho(T)) \Big| \leq o(h).
$$
Note that $\rho_0 \mapsto \rho(T)$ is linear and continuous as a map from $C^{-k}$ onto itself. Apply then Lemma \ref{lem.condC1} to get that $U^0$ is $C^1$ in $m$ with 
\be\label{deltaU0deltam}
\frac{\delta U^0}{\delta m}(t_0,x_0,m_0,y)= \frac{\delta G^0}{\delta m}(x_0,m(T))(\rho(T)).
\ee
%{\color{green} Then we can prove as we did for $U$, by using Proposition \ref{prop.LipschDeltaU}, that
%\be\label{ezjslkfdngf} 
%\|\frac{\delta U^0}{\delta m}(t, x_0, m,\cdot)\|_k \leq (1+C_MT) \|\frac{\delta G^0}{\delta m}(x_0,m(T),\cdot)\|_k,  
%\ee
% with a constant $C_M$ which depends on $\|u(t)\|_{k+1}$; hence it depends on $\|G(x_0,\cdot,m)\|_{k+1}$, which is bounded by $M$ by assumption. }
 Moreover, one can check as in the proof of Proposition \ref{Prop.DerivU} that $U^0$ solves \eqref{eq.MasterL} (here it is even simpler, and based on the fact that by definition of $U^0$, $U^0(t_0+h, x_0, m(t_0+h))-U^0(t_0,x_0, m(t_0))= 0$). 

Concerning the differentiability of $U^0$ with respect to $x_0$, let $\xi$ be any unit vector of $\R^d$,  $(u,m)$, $(u_h,m_h)$ and  $(v,\rho)$ be as in Lemma \ref{Lem.DerivU} with $m_1 = m_0$. Then,
\begin{align*}
\Big| G^0 (y_0+ \xi h, m_h(T))  & - G^0(y_0, m(T)) - h G_{x_0}^0(y_0, m(T))\cdot\xi - h \frac{\delta G^0}{\delta m}(y_0,m(T))(\rho(T)) \Big| \le \\
& \left| G^0(y_0+ \xi h, m_h(T)) - G^0(y_0, m_h(T)) - h G_{x_0}^0(y_0, m_h(T))\cdot\xi \right| + \\
& h \left|  G_{x_0}^0(y_0, m_h(T)) - G_{x_0}^0(y_0, m(T)) \right| + \\
& \left| G^0(y_0, m_h(T)) - G^0(y_0, m(T)) - h \frac{\delta G^0}{\delta m}(y_0,m(T))(\rho(T)) \right|.
\end{align*}
The third term of this inequality can be treated as in \eqref{eq1g0}. Therefore,
\[
\Big| U^0(t_0,y_0+ \xi h,m_0) - U^0(t_0,y_0,m_0) - h G_{x_0}^0(y_0, m(T))\cdot\xi - h \frac{\delta G^0}{\delta m}(y_0,m(T))(\rho(T)) \Big| \le o(h),
\]
hence it follows that
\be\label{Dx0U0}
D_{x_0} U^0(t_0,x_0,m_0)\cdot \xi= G_{x_0}^0(x_0, m(T))\cdot \xi +\frac{\delta G^0}{\delta m}(x_0,m(T))(\rho(T))\,.
\ee
%Then Proposition \ref{prop.LipschDeltaU} and \eqref{r1r2r3} (for $n=2$) imply
%\begin{align*}
%\sup_{\ t\in [0,T],\ m\in \Pw} \|U^0(t,\cdot,m)\|_{1} & \leq  \sup_{m}\|G^0(\cdot,m)\|_{1}
%\\
%& \quad +  \left\| \frac{\delta G^0}{\delta m}(x_0, m,\cdot)\right\|_{k}\, C_M T (1+ \sup_{x_0,m}\|G_{x_0}(x_0,\cdot,m)\|_{1})
%\,.
%\end{align*}
%This yields  the estimate on $D_{x_0} U^0$.\\
%

\medskip

We now prove the estimates. By Proposition \ref{prop.esti1MFG-0} and the representation formulas \eqref{rep.U1} and \eqref{rep.UL}, we have, for any $x_0\in \R^{d_0}$, $m\in \Pw$ and $r \le n$,  
\begin{align*}
& |U^0(t,x_0,m)|^2 + |D^r_x U(t,x_0, x,m)|^2  =  |G^0(x_0, m(T))|^2 + |D^r_x u(t,x)|^2 \\
& \qquad  \leq |G^0(x_0, m(T))|^2+  (\sup_{x} |D^r_x G(x_0, x, m(T))| +  C_MT)^2 \\
& \qquad \le \Big(\big(|G^0(x_0, m(T))|^2+  \sup_{x} |D^r_x G(x_0, x, m(T)|^2\big)^{1/2} +  C_MT\Big)^2 , 
\end{align*} 
(where we used that $x^2+(y+z)^2\leq ((x^2+y^2)^{1/2}+z)^2$ for nonnegative reals $x,y,z$) which gives \eqref{kjernzd}. Next we prove \eqref{kjernzd2}. For $|l|=1$, $l\in \N^{d_0}$, we  represent $\partial^l_{x_0}U^0$ and $\partial^l_{x_0}U$ by \eqref{Dx0U0} and \eqref{Dx0U} respectively, where $(v^l,\rho^l)$ is as in Lemma \ref{Lem.DerivU} with $\xi= e_l$, $m_1 = m_0$ (so that $\rho^l_0=0$). Then we have, for $r \le n-1$, 
 \begin{align*}
& \sum_{|l|=1} |\partial^l_{x_0}U^0(t,x_0,m)|^2 + | D^r_x \partial^l_{x_0} U(t,x_0, x, m) |^2 \\
& \qquad  = \sum_{|l|=1} \Big| \partial^l_{x_0} G^0(x_0, m(T)) +\frac{\delta G^0}{\delta m}(x_0,m(T))(\rho^l(T))\Big|^2 + |D^r_x v^l(t,x)|^2 
\end{align*}  
Note that $\sup_t\|\rho(t)\|_{-k}\leq C_MT$ by Proposition \ref{Prop.estiw}. As the $v^l$ solve HJ equations with the same diffusion and the same drift, Proposition \ref{prop.highS}, \eqref{hyphypGG0} and \eqref{r1r2r3} imply that 
\begin{align*}
& \sup_x \Bigl(\sum_{|l|=1}|D^r v^l|^2\Bigr)^{1/2}  \leq (1+CT) \sup_x \Bigl(\sum_{|l|=1}|D^r v^l(T)|^2\Bigr)^{1/2}+ C_MT\\
&  \leq (1+CT) \sup_x \Bigl(\sum_{|l|=1} (  \|D^r_x \frac{\delta G}{\delta m} (x_0,\cdot,m(T),\cdot) \|_{0; k}\|\rho(T)\|_{-k}  + |D^r_x \partial^l_{x_0} G(x_0,x,m(T)) |)^2\Bigr)^{1/2}+C_MT\\ 
&  \leq \sup_x \Bigl(\sum_{|l|=1} (  |D^r_x \partial^l_{x_0} G(x_0,x,m(T)) |+C_MT )^2 \Bigr)^{1/2}+C_MT
%\\
%&  
\leq \sup_x (\sum_{|l|=1} |D^r_x \partial^l_{x_0} G(x_0,x,m(T)) |^2 )^{1/2}+C_MT,
\end{align*}
while 
 \begin{align*}
&  \sum_{|l|=1} |\partial^l_{x_0}U^0(t,x_0,m)|^2  \leq \sum_{|l|=1} \Bigl( | \partial^l_{x_0} G^0(x_0, m(T))|+|\frac{\delta G^0}{\delta m}(x_0,m(T))(\rho^l(T))|\Bigr)^2\\
&\qquad  \leq  \sum_{|l|=1} ( | \partial^l_{x_0} G^0(x_0, m(T))+ C_M T )^2
\leq  \Bigl( \big( \sum_{|l|=1}  | \partial^l_{x_0} G^0(x_0, m(T))|^2\big)^{1/2}+ C_M T \Bigr)^2 .  
\end{align*}  
Using that $((x+z)^2+(y+z)^2)^{1/2}\leq (x^2+y^2)^{1/2}+\sqrt{2}z$, we obtain
 \begin{align*}
& \sup_x \left(\sum_{|l|=1} |\partial^l_{x_0}U^0(t,x_0,m)|^2 + | D^r_x \partial^l_{x_0} U(t,x_0, x, m) |^2\right)^{1/2} \\
& \qquad \leq \sup_x\Bigl(  \sum_{|l|=1}  \left| \partial^l_{x_0} G^0(x_0, m(T))\right|^2  + |D^r_x \partial^l_{x_0} G(x_0,x,m(T)) |^2  \Bigr)^{1/2}+C_MT,
\end{align*}  
from which we derive \eqref{kjernzd2},   by taking the $\sup$ over $x_0$,  summing over $r$  and finally taking the sup over $m$.

For \eqref{kjernzd3}, let $(v,\rho)$ be as in Lemma \ref{Lem.DerivU} with $m_1 - m_0 = \rho_0 \in C^{-k}$ and $\xi = 0$, as in \eqref{deltaU0deltam} and \eqref{reformBIS}. We have, for any $r\leq n-1$, 
$$
\Big|\frac{\delta U^0}{\delta m}(t, x_0, m)(\rho_0) \Big|^2+ \Big|D^r_x \frac{\delta U}{\delta m}(t, x_0, x, m)(\rho_0) \Big|^2 = \Big|\frac{\delta G^0}{\delta m}(x_0,m(T))(\rho(T))\Big|^2 + |D^r_x v(t, x)|^2.
$$
So again by Proposition \ref{prop.LipschDeltaU},
\begin{align*}
& \Big|\frac{\delta U^0}{\delta m}(t, x_0, m)(\rho_0) \Big|^2+ \Big|D^r_x \frac{\delta U}{\delta m}(t, x_0, x, m)(\rho_0) \Big|^2 \\
&  \le  \Big|\frac{\delta G^0}{\delta m}(x_0,m(T))(\rho(T))\Big|^2 + \Bigl(\sup_x \Big|D^r_x \frac{\delta G}{\delta m} (x_0,\cdot,m(T))(\rho(T))\Big| + C_MT \|\rho(T)\|_{-k}\Bigr)^2 \\
&  \le \left[ \sup_x \frac1{\|\rho(T)\|_{-k}} \left(\Big|\frac{\delta G^0}{\delta m}(x_0,m(T))(\rho(T))\Big|^2 +  \Big|D^r_x \frac{\delta G}{\delta m}(x_0,x,m(T))(\rho(T)) \Big|^2 \right)^{1/2}   +C_MT \right]^2 \|\rho(T)\|^2_{-k} \\
&  \le  (1+C_MT)^2    \left[ \sup_{x, \|\rho\|_{-k} = 1} \left(\Big|\frac{\delta G^0}{\delta m}(x_0,m(T))(\rho)\Big|^2 +  \Big|D^r_x \frac{\delta G}{\delta m}(x_0,x,m(T))(\rho) \Big|^2 \right)^{1/2}   +C_MT \right]^2 \|\rho_0\|^2_{-k},
\end{align*}     
%\footnote{ even if it was equivalent without the square, it seems easier for the reader to see the square since it comes from $\|\rho(T)\|^2_{-k}$...}
This gives \eqref{kjernzd3}.
 \end{proof}

%%%%%%%%%%%%%%%%%%%%%%%%
\subsection{Second order differentiability of $U$ and $U^0$} 

\begin{Proposition}\label{Prop.DerivU2} Let $U$ be the solution of \eqref{eq.Master1} given by \eqref{rep.U1}.  Let $n\geq 3$ and $k\in \{2,\dots,n-1\}$. Suppose, in addition to  the assumptions of Proposition \ref{Prop.DerivU},  that $G$ is of class $C^2$  and that $ \|\frac{\delta^2 G}{\delta m^2}(x_0, \cdot, m,\cdot,\cdot)\|_{n-2;k-1,k-1} \leq M$. 
%\begin{align*}
%& \sup_{x_0, m}     \|\frac{\delta^2 G}{\delta m^2}(x_0, \cdot, m,\cdot,\cdot)\|_{n-2;k-1,k-1} \leq M\,.
%\end{align*}
 Then  there exists $T_M>0$  (depending on $M$ and on the data but not on $G$) such that, if $T\in (0,T_M]$, the map $U$ is $C^2$ with respect to the measure variable and the parameter $x_0$, and satisfies
$$
\ds  \sup_{t\in [0,T]} \|\frac{\delta^2 U}{\delta m^2}(t)\|_{n-2;k-1,k-1} \leq \ds \|\frac{\delta^2 G}{\delta m^2}\|_{n-2;k-1,k-1}+ C_MT .
$$
\end{Proposition}

\begin{proof}   Our first goal is to show that $\delta U/\delta m$ is differentiable with respect to $m$. Let $(t_0,m_0)\in [0,T)\times \Pw$, $y,y' \in \R^d$ and
\begin{itemize}
\item $(u,m)$ (respectively $(u_h,m_h)$) be the solution of the MFG system \eqref{eq.MFGcoupled} with initial condition $(t_0,m_0)$ (respectively $(t_0, (1-h)m_0+h\delta_{y'})$),
\item $(v,\rho)$ (respectively $(v', \rho')$) be the solution of the first order linearized system \eqref{MFG2lkjenze} with zero right-hand side, initial condition $(t_0, \delta_y)$ (respectively $(t_0, \delta_{y'})$) and where  the Hamiltonian and its derivatives are evaluated at $(x_0,x,Du(t,x),m(t))$,
\item $(\tilde v_h, \tilde \rho_h)$  be the solution to the first order  linearized system \eqref{MFG2lkjenze} with zero right-hand side, with initial condition $(t_0, \delta_y)$ and where the Hamiltonian and its derivatives are evaluated at $(x_0,x,Du_h(t,x),m_h(t))$,
\item $(w,\mu)$ be the solution to the second order linearized system \eqref{MFG3lkjenze} associated with $(u,m)$,  $(v, \rho), (v',\rho')$ and with right-hand side $0$.
\end{itemize}
Recall  (see \eqref{reformBIS}) that
\be\label{reprfor123}
\begin{aligned}
& \tilde v_h(t_0,x) = \frac{\delta U}{\delta m} (t_0,x_0,x, (1-h)m_0+h\delta_{y'}, y), \\
& v(t_0,x) = \frac{\delta U}{\delta m} (t_0,x_0,x, m_0 , y), \qquad \text{and} \quad 
v'(t_0,x) = \frac{\delta U}{\delta m} (t_0,x_0,x, m_0 , y')
\end{aligned}
\ee
so we expect $w(t_0,\cdot)$ to represent the derivative in $m$ of $\delta U/\delta m$, namely $\frac{\delta^2 U}{\delta m^2} (t_0,x_0,x, m_0 , y, y')$.

We consider 
$$
(\hat v_h, \hat \rho_h):= (\tilde v_h,\tilde \rho_h)-(v,\rho)-h (w,\mu). 
$$
Let us first note that, by Proposition \ref{Prop.LipDep}, we have 
\be\label{lehsfnjn0}
\sup_{t\in [t_0,T]} \Bigl( \|\tilde u_h(t,x)-u(t,x)\|_{n-1} + \dw(m_h(t),m(t)) \Bigr)\leq C\dw((1-h)m_0+h\delta_{y'},m_0)\leq C h. 
\ee
Next we claim that 
\be\label{lehsfnjn}
\sup_{t\in [t_0,T]} \|\tilde v_h(t,x)-v(t,x)\|_{n-2} + \|\tilde \rho_h(t)-\rho(t)\|_{-(k-1)} \leq Ch.
\ee
Indeed, the pair $(\tilde v_h,\tilde \rho_h)-(v,\rho)$ solves the  first order linearized system \eqref{MFG2lkjenze}, associated with $(u,m)$, initial condition $(t_0, 0)$ and with a right-hand side given by 
\begin{align*}
R_{h,1}(t,x)= & - \Bigl(\big(H_p(x_0,x,Du_h,m_h(t))-H_p(x_0,x,Du,m(t))\big)\cdot D\tilde v_h \\
& + \big(\frac{\delta H}{\delta m}(x_0,x,Du_h,m_h(t))-\frac{\delta H}{\delta m}(x_0,x,Du,m(t))\big)(\tilde \rho_h(t) ) \Bigr)\\
R_{h,2}(t,x) = & \tilde \rho_h(H_p(x_0,x,Du_h,m_h(t))-H_p(x_0,x,Du,m(t)))\\ 
&  +(m_hH_{pp}(x_0,x, Du_h,m_h)-mH_{pp}(x_0,x, Du,m))\cdot D\tilde v_h\\
& + \Big(m_h\frac{\delta H_{p}}{\delta m} (x_0,x, Du_h,m_h)-m\frac{\delta H_{p}}{\delta m} (x_0,x, Du,m)\Big)(\tilde \rho_h)\\
R_{h,3}(t,x) = & \left(\frac{\delta G}{\delta m}(x_0,x,m_h(T)) - \frac{\delta G}{\delta m}(x_0,x,m(T)) \right)(\tilde \rho_h(T)).
\end{align*}
 Applying Proposition  \ref{prop.LipschDeltaU} and using \eqref{lehsfnjn0} we infer that \eqref{lehsfnjn} holds. 
 
In view of the equations satisfied by $(\tilde v_h, \tilde \rho_h)$, $(v,\rho)$ and $(w,\mu)$, the pair $(\hat v_h, \hat \rho_h)$ solves the first order linearized system \eqref{MFG2lkjenze}, associated with $(u,m)$,  initial condition $(t_0, 0)$ and with 
\begin{align*}
R_{h,1}(t,x)= & - \Bigl[ \big(H_p(x_0,x,Du_h,m_h(t))-H_p(x_0,x,Du,m(t))\big)\cdot D \tilde v_h \\
& \qquad - hH_{pp}(x_0,x,Du,m(t)) D v\cdot Dv' - h\frac{\delta H_p}{\delta m}(x_0,x,Du,m(t))(\rho'(t))\cdot D v  \\
& \qquad + \Big(\frac{\delta H}{\delta m}(x_0,x,Du_h,m_h(t))-\frac{\delta H}{\delta m}(x_0,x,Du,m(t))\Big)(\tilde \rho_h(t)) \\
&  \qquad - h \frac{\delta^2 H}{\delta m^2}(x_0,x,Du,m(t)) (\rho(t),\rho'(t)) - h\frac{\delta H_p}{\delta m}(x_0,x,Du,m(t))( \rho(t))\cdot D v' \Bigr], 
\end{align*}
\begin{align*}
& R_{h,2}(t,x) =  \tilde \rho_h \Big(H_p(x_0,x,Du_h,m_h(t)) - H_p \Big) - h \rho \Big( H_{pp}D v' + \frac{\delta H_p}{\delta m}(\rho') \Big) + \\
& \qquad D \tilde v_h \cdot \Big( m_h H_{pp}(x_0,x,Du_h,m_h(t)) - m H_{pp} \Big) - h D  v \cdot \Big( \rho' H_{pp} + m \frac{\delta H_{pp}}{\delta m}(\rho') +  mH_{ppp}Dv' \Big) + \\
& \qquad  \Big(m_h \frac{\delta H_p}{\delta m}(x_0,x,Du_h,m_h(t)) - m \frac{\delta H_p}{\delta m} \Big)  (\tilde \rho_h) - h \Big(\rho' \frac{\delta H_p}{\delta m} + m Dv' \cdot \frac{\delta H_{pp}}{\delta m}\Big)  (\rho)  - hm  \frac{\delta^2 H_p}{\delta m^2}(\rho, \rho'),
\end{align*}
\begin{align*}
R_{h,3}(x) = & \frac{\delta G}{\delta m}(x_0,x,m_h(T))(\tilde \rho_h(T)) - \frac{\delta G}{\delta m}(x_0,x,m(T))(\tilde \rho_h(T)) - h \frac{\delta^2 G}{\delta m^2}(x_0,x,m(T))(\rho(T), \rho'(T))
\end{align*}
(for brevity, $H_p$ and its derivatives in $R_{h,2}$ are evaluated at $(x_0,x,Du,m(t))$, unless otherwise specified). Using
\be\label{primebounds}
\sup_{t\in [t_0,T]}\|u_h(t)-u(t)-hv'(t)\|_{n-2} \leq  Ch^2 , \quad 
\sup_{t\in [t_0,T]} \|m_h(t)-m(t)-h\rho'(t)\|_{-(k-1)}\; \leq Ch^2
\ee
(see \eqref{lkaejnzred:1} and \eqref{lkaejnzred:2} in Lemma \ref{Lem.DerivU}) as well as the above estimate \eqref{lehsfnjn0}, we have 
$$
\sup_{t} \left( \|R_{h,1}(t,\cdot)\|_{n-2} + \|R_{h,2}(t,\cdot)\|_{-(k-1)} + \|R_{h,3}(t,\cdot)\|_{n-2} \right) \leq Ch^2.
$$
Then Proposition  \ref{prop.LipschDeltaU} and the representation formula \eqref{reprfor123} implies that 
\begin{align*}
&  \|  \frac{\delta U}{\delta m} (t_0,x_0,\cdot, (1-h)m_0+h\delta_{y'}, y) -\frac{\delta U}{\delta m} (t_0,x_0,\cdot, m_0 , y) -h w(t_0,\cdot)\|_{n-2}
\\
& \qquad 
=  \|\tilde v_h(t_0,\cdot)-v(t_0,\cdot)-h w(t_0,\cdot)\|_{n-2}
\leq  \sup_{t} \|\hat v_h(t) \|_{n-2} \leq Ch^2. 
\end{align*}
Note that we also have
\be\label{tildehjhj}
\sup_{t\in [t_0,T]} \|\tilde \rho_h(t) - \rho(t) - h\mu(t)\|_{-k}\; \leq Ch^2.
\ee
Hence, we can apply Lemma \ref{lem.condC1} as in the proof of Proposition \ref{Prop.DerivU} and infer that $\delta U/\delta m$ has a derivative in $m$ given by $w$: 
$$
\frac{\delta^2 U}{\delta m^2} (t_0,x_0,x,  m_0,y,y') = w(t_0, x). 
$$
If, in general, $w$ is the solution to the second order linearized system \eqref{MFG3lkjenze} associated with $(v, \rho), (v',\rho')$ (having initial data $(t_0, \rho_0)$ and $(t_0, \rho'_0)$ respectively)  and with $R_i=0$, $\tilde R_i=0$, $i=1,\dots 3$, then by a linearity argument one may also conclude that
\be\label{reprD2Um}
\int_{\R^d} \frac{\delta^2 U}{\delta m^2}(t_0,x_0,x,m_0,y,y')\rho_0(dy)\rho'_0(dy') = w(t_0,x).
\ee
Thus, the estimate on $\frac{\delta^2 U}{\delta m^2}$ follows from Corollary \ref{d2m}, which gives
\begin{align}\label{elzjnfd,mlk,k}
&  \|\frac{\delta^2 U}{\delta m^2} (t_0, x_0, \cdot,m_0,\cdot, \cdot)\|_{n-2;k-1,k-1}
%\\
%&\qquad \qquad  
\leq \ds \|\frac{\delta^2 G}{\delta m^2} (x_0, \cdot,m(T), \cdot,\cdot)\|_{n-2;k-1,k-1}+ C_MT \, ,
%\notag
\end{align}
using the fact that $\sup\limits_{t } \|\rho(t)\|_{-(k-1)}  \leq   \ds  \left(1+C_MT \right) \|\rho_0\|_{-(k-1)}$ and that the same holds for $\rho'$.
\end{proof}

Next we discuss the second order regularity of $U$ and $U^0$ with respect to $m$ and $x_0$. 

\begin{Proposition}\label{Prop.DerivU02} Let $U^0$ and $U$ be the solutions of  \eqref{eq.MasterL}  and \eqref{eq.Master1} respectively. Suppose, in addition to  the assumptions of Propositions  \ref{Prop.DerivUL} and \ref{Prop.DerivU2}, that we have  
\begin{align*}
 \Bigl\|D^2_{x_0}(G^0,G)\Bigr\|_{n-2}+  \Bigl\|D_{x_0} \frac{\delta (G^0,G)}{\delta m}\Bigr\|_{n-2;k-1} + \left\|\frac{\delta^2 (G^0,G)}{\delta m^2 }\right\|_{n-2;k-1,k-1} \leq M.
\end{align*}
Then there exists $T_M>0$  (depending on $M$ and on the data but not on $G$) such that, if $T\in (0,T_M]$, the maps $U^0$ and $U$ are $C^2$ with respect to the measure variable and $x_0$, and 
\begin{align*}
\sup_t \ \Bigl\|D^2_{x_0}(U^0,U)(t)\Bigr\|_{n-2} & \le  \Bigl\|(D^2_{x_0}G^0,D^2_{x_0}G)\Bigr\|_{n-2} + C_M T, \\
\sup_t \ \Bigl\|D_{x_0}\frac{\delta (U^0,U)}{\delta m}(t)\Bigr\|_{n-2;k-1} & \le \Bigl\|D_{x_0}\frac{\delta (G^0,G)}{\delta m}\Bigr\|_{n-2;k-1} + C_MT .
\end{align*}   
Moreover, 
\begin{align*}
\sup_t \ \left\|\frac{\delta^2 (U^0,U)}{\delta m^2 }(t)\right\|_{n-2;k-1,k-1} \le \left\|\frac{\delta^2 (G^0,G)}{\delta m^2 }\right\|_{n-2;k-1,k-1} + C_MT.
\end{align*} 

\end{Proposition}

\begin{proof} \underline{Step 1.} The \underline{differentiability of $\delta U/\delta m$ with respect to $x_0$} can be achieved exactly as for its differentiability with respect to $m$ in Proposition \ref{Prop.DerivU2}. For any direction $\xi\in \R^{d_0}$, let
\begin{itemize}\label{itemi1}
\item $(u,m)$ (respectively $(u_h,m_h)$) be the solution of the MFG system \eqref{eq.MFGcoupled} with initial condition $(t_0,m_0)$ and parameters $x_0$ and $x_0 + h \xi$ respectively,
\item $(v,\rho)$ (respectively $(v', \rho')$) be the solution of the first order linearized system \eqref{MFG2lkjenze} with zero right-hand side (respectively right-hand side as in \eqref{Rchoi}), initial condition $(t_0, \delta_y)$ (respectively $(t_0, 0)$) and where  the Hamiltonian and its derivatives are evaluated at $(x_0,x,Du(t,x),m(t))$,
\item $(\tilde v_h, \tilde \rho_h)$  be the solution to the first order  linearized system \eqref{MFG2lkjenze} with zero right-hand side, with initial condition $(t_0, \delta_y)$ and where the Hamiltonian and its derivatives are evaluated at $(x_0+h\xi,x,Du_h(t,x),m_h(t))$,
\item $(w,\mu)$ be the solution to the second order linearized system \eqref{MFG3lkjenze} associated with $(v, \rho), (v',\rho')$ (and $(u,m)$), and with right-hand side
\begin{align*}
& \tilde R_1(t,x) = -H_{x_0p}(x_0,x,Du,m(t)) \xi \cdot Dv - \frac{\delta H_{x_0}}{\delta m}(x_0,x,Du,m(t))(\rho(t))\cdot \xi, \\
& \tilde R_{2}(t,x) =  \rho H_{x_0p} (x_0,x,Du,m(t)) \xi + mH_{x_0pp}(x_0,x,Du,m(t))\xi\ Dv + m \frac{\delta H_{x_0p}}{\delta m}(\rho)\xi , \\
& \tilde R_3(x) = \frac{\delta G_{x_0}}{\delta m}(x_0,x,m(T))(\rho(T))\cdot \xi,
\end{align*}
\end{itemize}
so that
\begin{align*}
& \tilde v_h(t_0,x) = \frac{\delta U}{\delta m} (t_0,x_0+h\xi, x, m_0, y), \\
& v(t_0,x) = \frac{\delta U}{\delta m} (t_0,x_0,x, m_0 , y), \qquad \text{and} \quad 
v'(t_0,x) = U_{x_0} (t_0,x_0,x, m_0)\cdot \xi.
\end{align*}
Then we find $\frac{\delta U_{x_0}}{\delta m} (t_0,x_0,x,m_0,y)\cdot \xi= w(t_0,x)$, and if one replaces $\delta_y$ by an arbitrary $\rho_0 \in C^{-(k-1)}$ as the initial datum for $\rho$, the following representation holds:
\be\label{repdeltaUx_0deltam}
\frac{\delta U_{x_0}}{\delta m} (t_0,x_0,x,m_0)(\rho_0) \cdot \xi= w(t_0,x) .
\ee

 \underline{Step 2.} The \underline{second order differentiability of $U$ with respect to $x_0$} can be checked in a similar way: let $(u,m)$ and $(u_h,m_h)$ be as before,
\begin{itemize}\label{itemi2}
\item $(v,\rho)$, $(\tilde v_h, \tilde \rho_h)$ be the solutions of the first order linearized system \eqref{MFG2lkjenze} with right-hand side as in \eqref{Rchoi}, initial condition  $(t_0, 0)$, and Hamiltonian and its derivatives evaluated at $(x_0,x,Du(t,x),m(t))$ and $(x_0+h\xi,x,Du_h(t,x),m_h(t))$ respectively,
\item  $(w,\mu)$ be the solution to the second order linearized system \eqref{MFG3lkjenze} associated with $(v, \rho), (v',\rho') = (v, \rho)$ (and $(u,m)$), and  with right-hand side $\tilde R_1, \tilde R_2, \tilde R_3$ given by \eqref{tilde2}.
\end{itemize}
Then we find 
\be\label{repUx_0x_0}
D^2_{x_0} U(t_0,x_0,x,m_0)\xi\cdot \xi = w(t_0,x) .
\ee

\underline{Step 3.} We now prove the regularity of $U^0$. To show that $\delta U^0/\delta m$ is differentiable with respect to $m$, let $(t_0,m_0)\in [0,T)\times \Pw$, $y,y' \in \R^d$ and
\begin{itemize}
\item $(u,m)$ (respectively $(u_h,m_h)$) be the solution of the MFG system \eqref{eq.MFGcoupled} with initial condition $(t_0,m_0)$ (respectively $(t_0, (1-h)m_0+h\delta_{y'})$),
\item $(v,\rho)$ (respectively $(v', \rho')$) be the solution of the first order linearized system \eqref{MFG2lkjenze} with zero right-hand side, initial condition $(t_0, \delta_y)$ (respectively $(t_0, \delta_{y'})$) and where  the Hamiltonian and its derivatives are evaluated at $(x_0,x,Du(t,x),m(t))$,
\item $(\tilde v_h, \tilde \rho_h)$  be the solution to the first order  linearized system \eqref{MFG2lkjenze} with zero right-hand side, with initial condition $(t_0, \delta_y)$ and where the Hamiltonian and its derivatives are evaluated at $(x_0,x,Du_h(t,x),m_h(t))$,
\item $(w,\mu)$ be the solution to the second order linearized system \eqref{MFG3lkjenze} associated with $(v, \rho), (v',\rho')$ (and $(u,m)$), and with right-hand side $0$,
\end{itemize}
as in the proof of differentiability of $\delta U/\delta m$ with respect to $m$ in Proposition \ref{Prop.DerivU2}. Note that
\begin{align*}
 & \frac{\delta U^0}{\delta m}(t_0,x_0,(1-h)m_0+ h \delta_{y'},y) = \frac{\delta G^0}{\delta m}(x_0,m_h(T))(\tilde \rho_h(T)), \\
 & \frac{\delta U^0}{\delta m}(t_0,x_0,m_0,y)= \frac{\delta G^0}{\delta m}(x_0,m(T))(\rho(T)).
\end{align*}
Therefore, using \eqref{primebounds} and \eqref{tildehjhj}
\begin{align*}
& \Big| \frac{\delta G^0}{\delta m}(x_0,m_h(T))(\tilde \rho_h(T)) - \frac{\delta G^0}{\delta m}(x_0,m(T))(\rho(T))  \\
& \qquad - h\Big( \frac{\delta^2 G^0}{\delta m^2}(x_0,m(T))(\rho(T), \rho'(T)) + \frac{\delta G^0}{\delta m}(x_0,m(T))(\mu(T)) \Big) \Big| \le Ch^2\,.
\end{align*}
Lemma \ref{lem.condC1} then implies that $\frac{\delta U^0}{\delta m}(t_0,x_0,\cdot,y)$ has a derivative, and by linearity, if $\mu$ is the solution to the second order linearized system \eqref{MFG3lkjenze} associated with $(v, \rho), (v',\rho')$ (that in turn have initial data $(t_0, \rho_0)$ and $(t_0, \rho'_0)$ respectively and with zero right-hand side), then
\be\label{diff2U0}
\int_{\R^d} \frac{\delta^2 U^0}{\delta m^2}(t_0,x_0,m_0,y,y')\rho_0(dy)\rho'_0(dy') = \frac{\delta^2 G^0}{\delta m^2}(x_0,m(T))(\rho(T), \rho'(T)) + \frac{\delta G^0}{\delta m}(x_0,m(T))(\mu(T)).
\ee
Hence, by the representation formula \eqref{reprD2Um} for $\delta^2 U /\delta^2 m$, Proposition \ref{prop.LipschDeltaU}, \ref{Prop.estiw} and Corollary \ref{d2m}, we have, for $r\leq n-2$, 
\begin{align*}
& \Big| \frac{\delta^2 U^0}{\delta m^2}(t,x_0,m_0)(\rho_0, \rho'_0) \Big|^2 + \Big| D^r_x \frac{\delta^2 U}{\delta m^2}(t,x_0,x,m_0)(\rho_0, \rho'_0) \Big|^2 \\
 & = \Big(\Big| \frac{\delta^2 G^0}{\delta m^2}(x_0,m(T))(\rho(T), \rho'(T))\Big| + \Big| \frac{\delta G^0}{\delta m}(x_0,m(T))(\mu(T)) \Big|\Bigr)^2 + | D^r_x w(t,x)|^2  \\
 &  \le  \Bigl(\Big| \frac{\delta^2 G^0}{\delta m^2}(x_0,m(T))(\rho(T), \rho'(T))\Big|+ C_M T  \|\rho_0\|_{-(k-1)}\|\rho'_0\|_{-(k-1)}\Bigr)^2 \\
& \qquad \qquad \qquad \qquad \qquad \qquad + 
\Bigl( \sup_x \Big|D_x \frac{\delta^2 G}{\delta m^2} (x_0,x,m(T))(\rho(T), \rho'(T)) \Big| + C_M T  \|\rho_0\|_{-(k-1)}\|\rho'_0\|_{-(k-1)}\Bigr)^2  \\
 & \le \Bigl\{ \sup_x \frac1{\|\rho(T)\|_{-(k-1)}\|\rho'(T)\|_{-(k-1)}} \Bigl(\Big| \frac{\delta^2 G^0}{\delta m^2}(x_0,m(T))(\rho(T), \rho'(T))\Big|^2 + \Big|D^r_x \frac{\delta^2 G}{\delta m^2} (x_0,x,m(T))(\rho(T), \rho'(T)) \Big|^2 \Bigr)^{1/2}  \\
& \qquad\qquad\qquad\qquad \times (1+C_MT) + C_M T \Bigr\}^2  \|\rho_0\|_{-(k-1)}^2\|\rho'_0\|_{-(k-1)}^2  ,
\end{align*} 
(where we use that $(x+z)^2+(y+z)^2\leq ((x^2+y^2)^{1/2}+2z)^2$, for $x,y,z\geq 0$). Taking the square root, then sup over $x_0$, $\rho_0$ and $\rho_0'$ and summing over $r\leq n-2$ gives the estimate on $\left\|\frac{\delta^2 (U^0,U)}{\delta m^2 }\right\|_{n-2;k-1,k-1}$.
\smallskip

 Differentiability of $\delta U^0/\delta m$ with respect to $x_0$ 
follows analogous lines: $(v,\rho)$, $(v', \rho')$, $(\tilde v_h, \tilde \rho_h)$ and $(w,\mu)$ have to be changed according to Step 1.
By \eqref{repdeltaUx_0deltam}, we have, using the notations of Corollary \ref{d2mx0} and for any $r\leq n-2$:  
\begin{multline*}
\sum_{|l|=1} \Big| \partial_{x_0}^l\frac{\delta U^0}{\delta m}(t_0,x_0,m_0,y)(\rho_0)  \Big|^2 + \Big|D^r_x\partial_{x_0}^l \frac{\delta U}{\delta m} (t,x_0,x,m_0)(\rho_0) \Big |^2 = \\
\sum_{|l|=1}  \Big| \partial_{x_0}^l \frac{\delta G^0}{\delta m}(\rho(T)) +  \frac{\delta^2 G^0}{\delta m^2}(\rho(T), \rho^l(T)) + \frac{\delta G^0}{\delta m}(\mu^l(T)) \Big|^2  + |D^r_x w^l(t,x)|^2,
\end{multline*}
where $G^0$ and its derivatives are all evaluated at $(x_0,m(T))$. We obtain the bounds on $(\frac{\delta U^0_{x_0}}{\delta m}$, $\frac{\delta U_{x_0}}{\delta m})$  by using Propositions \ref{prop.LipschDeltaU}, \ref{Prop.estiw} and Corollary \ref{d2mx0}.

Finally, second order differentiability of $U^0$ with respect to $x_0$, and the corresponding bound,  can be obtained similarly: let $l,l'\in \R^{d_0}$ with $|l|=|l'|=1$, $(v^l,\rho^l)$, $(v^{l'}, \rho^{l'})$ and $(w^{l,l'},\mu^{l,l'})$ be as in Corollary \ref{d2x0}. Note that
$$
\partial^{l+l'}_{x_0} U^0(t_0,x_0,m_0) =  \partial^{l+l'}_{x_0} G^0 + \partial^{l}_{x_0}\frac{\delta G^0_{x_0}}{\delta m}(\rho^{l'}(T))  +
\partial^{l'}_{x_0}\frac{\delta G^0_{x_0}}{\delta m}(\rho^l(T))+  \frac{\delta^2 G^0}{\delta m^2}(\rho^l(T), \rho^{l'}(T)) + \frac{\delta G^0}{\delta m}(\mu^{l,l'}(T)),
$$
while $\partial^{l+l'}_{x_0} U^0(t_0,x_0,m_0)$ is given by polarizing the representation formula \eqref{repUx_0x_0}. We can then conclude by Propositions \ref{prop.LipschDeltaU}, \ref{Prop.estiw} and Corollary \ref{d2x0}.
\end{proof}

%%%%%%%%%%%%%%%%%%%%%%%%%%
\subsection{Uniform continuity estimates on  second order derivatives}

\begin{Proposition} \label{prop:rlipdelta2U} Let $U$ be the solution of \eqref{eq.Master1} given by \eqref{rep.U1} and  $n \ge 4$, $k\in \{3, \dots, n-1\}$. Suppose, in addition to  the assumptions of  Proposition \ref{Prop.DerivU2},  that
\be\label{iukzbdu}
{\rm Lip}_{n-3;k-2,k-2} \left(\frac{\delta^2 G}{\delta m^2}\right)\leq M. 
\ee
Then there exists $T_M>0$ (depending on $M$ and on the data but not on $G$), such that
\begin{align*}
&  \sup_{t} {\rm Lip}_{n-3;k-2,k-2} \left(\frac{\delta^2 U}{\delta m^2}(t)\right) 
%\\
%& \qquad  
%\qquad \qquad \qquad  
\leq 
\sup_{ x_0} {\rm Lip}_{n-3;k-2,k-2} \left(\frac{\delta^2 G}{\delta m^2}\right)
+ C_MT.
\end{align*}
\end{Proposition}

\begin{proof} 
%We will detail only the proof of Lipschitz regularity of $\delta^2 U/\delta m^2$ with respect to $x_0$ and $m$. Lipschitz regularity of $\delta U_{x_0}/\delta m$ and $D^2_{x_0} U$ can be proven by following identical lines.
%
We establish for later use a slightly stronger estimate involving the dependence with respect to $x_0$. This is used in Proposition \ref{prop:rlipdelta2U0} below. 
Let $(t_0, m_1, m_2)\in [0,T] \times \Pw^2$ and  $x_0^1, x_0^2 \in \R^{d_0}$ be fixed. We use the representation formula \eqref{reprD2Um} for $\delta^2 U/\delta m^2(t_0, x^1_0, m_1)$ and $\delta^2 U/\delta m^2(t_0, x_0^2, m_2)$. In particular we let, for $i = 1, 2$,
\begin{itemize}
\item $(u^i,m^i)$ be the solution to the MFG system \eqref{eq.MFGcoupled} starting from $m_i$ at time $t_0$ with $H$ (and $G$) evaluated at $(x_0^i,x,Du^i(t,x),m^i(t))$ (and $(x_0^i,x,m^i(T))$)\,,
\item $(v_i,\rho_i)$ (respectively $(v'_i, \rho'_i)$) be the solution of the first order linearized system \eqref{MFG2lkjenze} with zero right-hand side, initial condition $(t_0, \rho_0)$ (respectively $(t_0,\rho'_0)$) and where  the Hamiltonian and its derivatives are evaluated at $(x_0^i,x,Du^i(t,x),m^i(t))$,
\item $(w^i,\mu^i)$ be the solution to the second order linearized system \eqref{MFG3lkjenze} associated with $(v_i, \rho_i), (v'_i,\rho'_i)$ (and $x_0^i, u^i, m^i$), and with zero right-hand side.
\end{itemize}

We aim at estimating $(\bar w, \bar \mu):= (w^1-w^2, \mu^1-\mu^2)$, since
\be\label{rbarw}
\bar w(t_0,x)= \frac{\delta^2 U}{\delta m^2}(t_0,x^1_0,x,m_1)(\rho_0,\rho'_0) - \frac{\delta^2 U}{\delta m^2}(t_0,x_0^2,x,m_2)(\rho_0,\rho'_0).
\ee

We first set $ (\bar v, \bar \rho):= (v_1-v_2,\rho_1-\rho_2)$ and 
$ (\bar v', \bar \rho'):= (v'_1-v'_2,\rho'_1-\rho'_2)$. The pair $(\bar v, \bar \rho)$ solves the first order linearized system \eqref{MFG2lkjenze} with zero initial datum, $H$ and its derivatives evaluated at $(x_0^1,x,Du^1(t,x),m^1(t))$, and right-hand side
\begin{align*}
R_1(t,x)& = -(H_p^1-H_p^2)\cdot Dv_2 - (\frac{\delta H^1}{\delta m}-\frac{\delta H^2}{\delta m})(\rho_2(t)), \\
R_2(t,x)& =   \rho_2 (H_p^1-H_p^2)+ (m^1H_{pp}^1-m^2H_{pp}^2)Dv_2 +(m^1 \frac{\delta H_p^1}{\delta m}-m^2 \frac{\delta H_p^2}{\delta m})(\rho_2), \\
R_3(x)& =  (\frac{\delta G^1}{\delta m}-  \frac{\delta G^2}{\delta m}) (\rho_2(T)),
\end{align*}
where $H^i$ and its derivatives correspond to $H$ and its derivatives evaluated at $(x_0^i,x,Du^i(t,x),m^i(t))$.

By Proposition \ref{prop.LipschDeltaU} we have 
\be\label{estivivi}
\sup_{t\in [t_0,T]}\|v_i(t)\|_{n-1} \leq  C\|\rho_0\|_{-(k-2)}, \qquad \sup_{t\in [t_0,T]} \|\rho_i(t)\|_{-(k-2)} \leq  \left(1+CT \right)\|\rho_0\|_{-(k-2)} ,
\ee
where $C$ depends on the regularity of $\delta G /\delta m$, $H_{x_0}$, $H_{x_0p}$, $m^i$ and  $\sup_t \|u^i\|_{n}$. Note that, by the above estimates and Proposition \ref{Prop.LipDep}, 
$$
 \sup_t \|R_1(t)\|_{n-2} + \sup_{t}\|R_2(t)\|_{-(k-2)}+  \|R_3\|_{n-2}  \leq C\big(\dw(m_1,m_2) + |x_0^1 - x_0^2|\big) \|\rho_0\|_{-(k-2)}  , 
$$
and therefore by Proposition \ref{prop.LipschDeltaU} (applied to $n-1\geq 2$ and $k-2\geq 1$) we obtain
\begin{align}
& \sup_{t}\|\bar v(t)\|_{n-2} \leq C\,  T\big(\dw(m_1,m_2) + {|x_0^1 - x_0^2|}\big) \|\rho_0\|_{-(k-2)},   \label{estivv} \\
& \sup_{t} \|\bar \rho(t)\|_{-(k-1)} \leq   CT\big(\dw(m_1,m_2) + {|x_0^1 - x_0^2|}\big) \|\rho_0\|_{-(k-2)}. \label{estirhorho}
\end{align}
Completely analogous estimates hold for $v'_i, \rho'_i$ and their differences $\bar v', \bar \rho'$.

We now proceed by estimating $(\bar w, \bar \mu)$, which solves the first order linearized system with zero initial datum, $H$ and its derivatives evaluated at $(x_0^1,x,Du^1(t,x),m^1(t))$,  and right-hand side
\begin{align*}
\overline R_1(t,x) & := -\Bigl(  (H_p^1-H_p^2)\cdot Dw^2 + (\frac{\delta H^1}{\delta m}-\frac{\delta H^2}{\delta m})(\mu^2(t)) \\
 &
 +\frac{\delta^2 H^1}{\delta m^2}(\rho_1(t),\rho'_1(t))-  \frac{\delta^2 H^2}{\delta m^2}(\rho_2(t),\rho'_2(t)) 
+H_{pp}^1 Dv_1\cdot Dv'_1 - H_{pp}^2 Dv_2\cdot Dv'_2 \\
& \qquad  + \frac{\delta H_p^1}{\delta m}(\rho_1)\cdot Dv'_1-\frac{\delta H_p^2}{\delta m}(\rho_2)\cdot Dv'_2 
+ \frac{\delta H_p^1}{\delta m}(\rho'_1)\cdot Dv_1-\frac{\delta H_p^2}{\delta m}(\rho'_2)\cdot Dv_2 
%& \qquad + \big(D^2_{x_0p}H^1 Dv^1 - D^2_{x_0p}H^2 Dv^2 + \frac{\delta H^1_{x_0}}{\delta m}(\rho^1(t)) - \frac{\delta H^2_{x_0}}{\delta m}(\rho^2(t))\big) \xi \\
%& \qquad +\frac12(D^2_{x_0}H^1 - D^2_{x_0}H^2)\xi \, \xi 
\Bigr),
\end{align*}
\begin{align*}
\overline R_2(t,x) & := \mu^2 (H_p^1-H_p^2)+ (m^1H_{pp}^1-m^2H_{pp}^2)Dw^2
+(m^1 \frac{\delta H_p^1}{\delta m}-m^2 \frac{\delta H_p^2}{\delta m})(\mu^2) +\rho_1 H_{pp}^1Dv'_1
\\
&
- \rho_2 H_{pp}^2Dv'_2 +\rho'_1 H_{pp}^1Dv_1- \rho'_2 H_{pp}^2Dv_2
+ m^1 H_{ppp}^1Dv_1Dv'_1- m^2 H_{ppp}^2Dv_2Dv'_2\\
&\qquad \qquad + m^1 \frac{\delta^2 H_p^1}{\delta m^2}(\rho_1,\rho'_1)- m^2 \frac{\delta^2 H_p^2}{\delta m^2}(\rho_2,\rho'_2) + \rho_1 \frac{\delta H_p^1}{\delta m}(\rho'_1)-\rho_2 \frac{\delta H_p^2}{\delta m}(\rho'_2)\\
&\qquad \qquad  +\rho'_1 \frac{\delta H_p^1}{\delta m}(\rho_1)-\rho'_2 \frac{\delta H_p^2}{\delta m}(\rho_2)
 + m^1  \frac{\delta H_{pp}^1}{\delta m} (\rho'_1)Dv_1 \\
 &\qquad \qquad\qquad \qquad- m^2  \frac{\delta H_{pp}^2}{\delta m} (\rho'_2)Dv_2 + m^1  \frac{\delta H_{pp}^1}{\delta m} (\rho_1)Dv'_1-m^2  \frac{\delta H_{pp}^2}{\delta m} (\rho_2)Dv'_2
%& \qquad \qquad + \big(\rho^1 D^2_{x_0p} H^1 - \rho^2 D^2_{x_0p} H^2  + m_1 D^3_{x_0pp} H^1 \, Dv^1 - m_2 D^3_{x_0pp} H^2 \, Dv^2\big) \xi \\
%& \qquad \qquad + \big(m^1 \frac{\delta H^1_{x_0 p}}{\delta m}(\rho^1) - m^2 \frac{\delta H^2_{x_0 p}}{\delta m}(\rho^2)  +  \frac12 m^1D^3_{x_0x_0p}H^1\xi -  \frac12 m^2D^3_{x_0x_0p}H^2\xi \big) \xi
\end{align*}
and 
\begin{align*}
\overline R_3(x)  :=& \frac{\delta^2 G^1}{\delta m^2}(\rho_1(T),\rho_1'(T))- \frac{\delta^2 G^2}{\delta m^2}(\rho_2(T),\rho_2'(T))
+ (\frac{\delta G^1}{\delta m}- \frac{\delta G^2}{\delta m})(\mu^2(T)).
\end{align*}
Recall also that Proposition \ref{Prop.estiw}  and Remark \ref{rr'} (applied to $n-1$ and $k-1$) yield 
\be\label{fasdkjd2}
\sup_t \|w^i(t)\|_{n-3}   \leq  (1+CT) \|\rho_0\|_{-(k-2)}\|\rho_0'\|_{-(k-2)}, \quad
\sup_t \|\mu^i(t)\|_{-(k-1)}   \leq CT \|\rho_0\|_{-(k-2)}\|\rho_0'\|_{-(k-2)}.
\ee
By the previous inequalities, \eqref{estivivi}, \eqref{estivv} and \eqref{estirhorho} we get
\begin{align*}
\sup_t\|\overline R_1(t)\|_{n-3}+\sup_t \|\overline R_2(t)\|_{-(k-1)} \leq C T(\dw(m_1,m_2) + |x_0^1 - x_0^2|)\|\rho_0\|_{-(k-2)}\|\rho'_0\|_{-(k-2)}.
\end{align*}  
Similarly, using also the Lipschitz regularity of $\delta G/\delta m$, 
\begin{align*}
\|\overline R_3\|_{n-3}  \leq & (1+CT)  \Bigl\| \frac{\delta^2 G}{\delta m^2}(x_0^2, m^2(T)) - \frac{\delta^2 G}{\delta m^2}(x_0^1,m^1(T))\Bigr\|_{n-3;k-2,k-2}\|\rho_0\|_{-(k-2)}\|\rho'_0\|_{-(k-2)} \\
& + CT(\dw(m_1,m_2) + |x_0^1 - x_0^2|)\|\rho_0\|_{-(k-2)}\|\rho'_0\|_{-(k-2)},
\end{align*} 
Then, recalling that $\bar w = w^1-w^2$ satisfies \eqref{rbarw}, we obtain in view of \eqref{lkqehsrmdBISBIS} in Proposition \ref{prop.LipschDeltaU} and for any $r\leq n-3$, 
\begin{align}\label{ieukbfdn}
& \Bigl\| D^r_x \frac{\delta^2 U}{\delta m^2}(t_0,x_0^2,  m_2) - D^r_x\frac{\delta^2 U}{\delta m^2}(t_0,x_0^1, m_1)\Bigr\|_{0;k-2,k-2} \\
& \; \leq (1+C_MT)  \Bigl\|D^r_x \frac{\delta^2 G}{\delta m^2}(x_0^2, m_2(T)) -D^r_x \frac{\delta^2 G}{\delta m^2}(x_0^1,m_1(T)))\Bigr\|_{0;k-2,k-2} 
+ C_MT(\dw(m_1,m_2) + |x_0^1 - x_0^2|). \notag
\end{align}
Choosing $x^1_0=x^2_0$, summing over $r\leq n-3$ and recalling Proposition \ref{Prop.LipDep} and \eqref{iukzbdu} then gives the claim. 

Note that we have also the following inequality for $\bar \mu = \mu^1 - \mu^2$, that will be useful in the next proposition:
\begin{equation}\label{muest}
\ds \sup_{t\in [t_0,T]} \|\mu^1(t) - \mu^2(t)\|_{-k}\; \leq CT (\dw(m_1,m_2) + |x_0^1 - x_0^2|)\|\rho_0\|_{-(k-2)}\|\rho'_0\|_{-(k-2)}.
\end{equation}
\end{proof}

 Finally we establish the Lipschitz regularity of the second order derivatives of $G^0$ and $G$ with respect to $x_0$ and $m$. 

\begin{Proposition} \label{prop:rlipdelta2U0} Let $U$ be the solution of \eqref{eq.Master1} given by \eqref{rep.U1} and $U^0$ be the solution to \eqref{eq.MasterL} given by \eqref{rep.UL}.   Suppose that the assumptions of Proposition \ref{prop:rlipdelta2U} hold and that in addition:
\begin{align*}
& {\rm Lip}_{n-3;k-2,k-2} (\frac{\delta^2 G^0}{\delta m^2}, \frac{\delta^2 G}{\delta m^2})  
+ {\rm Lip}_{n-3;k-2} (\frac{\delta G^0_{x_0}}{\delta m}, \frac{\delta G_{x_0}}{\delta m}) + {\rm Lip}_{n-3} (D^2_{x_0} G^0, D^2_{x_0} G)  
\le M
\end{align*}
and 
\begin{align*}
& {\rm Lip}_{n-3;k-2,k-2}^{x_0} (\frac{\delta^2 G^0}{\delta m^2}, \frac{\delta^2 G}{\delta m^2})  
+ {\rm Lip}_{n-3;k-2}^{x_0} (\frac{\delta G^0_{x_0}}{\delta m}, \frac{\delta G_{x_0}}{\delta m}) + {\rm Lip}_{n-3}^{x_0} (D^2_{x_0} G^0, D^2_{x_0} G)  
\le M\,,
\end{align*}
for some $n\geq 4$ and $k \in \{3, \ldots, n-1\}$. 
Then
\begin{align*}
& \sup_t {\rm Lip}_{n-3;k-2,k-2} (\frac{\delta^2 U^0(t)}{\delta m^2}, \frac{\delta^2 U(t)}{\delta m^2})  \leq 
{\rm Lip}_{n-3;k-2,k-2} (\frac{\delta^2 G^0}{\delta m^2}, \frac{\delta^2 G}{\delta m^2})   +C_MT, 
\end{align*}
\begin{align*}
& \sup_t {\rm Lip}_{n-3;k-2,k-2}^{x_0} (\frac{\delta^2 U^0(t)}{\delta m^2}, \frac{\delta^2 U(t)}{\delta m^2})  \leq 
{\rm Lip}_{n-3;k-2,k-2}^{x_0} (\frac{\delta^2 G^0}{\delta m^2}, \frac{\delta^2 G}{\delta m^2})   +C_MT, 
\end{align*}
\begin{align*}
& \sup_t {\rm Lip}_{n-3;k-2} (\frac{\delta U^0_{x_0}(t)}{\delta m}, \frac{\delta U_{x_0}(t)}{\delta m}) \leq 
{\rm Lip}_{n-3;k-2} (\frac{\delta G^0_{x_0}}{\delta m}, \frac{\delta G_{x_0}}{\delta m}) +C_MT,
\end{align*}
\begin{align*}
& \sup_t {\rm Lip}_{n-3;k-2}^{x_0} (\frac{\delta U^0_{x_0}(t)}{\delta m}, \frac{\delta U_{x_0}(t)}{\delta m}) \leq 
{\rm Lip}_{n-3;k-2}^{x_0} (\frac{\delta G^0_{x_0}}{\delta m}, \frac{\delta G_{x_0}}{\delta m}) +C_MT, 
\end{align*}
and
\begin{align*}
& {\rm Lip}_{n-3} (D^2_{x_0} U^0(t), D^2_{x_0} U(t))  \leq  {\rm Lip}_{n-3} (D^2_{x_0} G^0, D^2_{x_0} G)  +C_MT,
\end{align*}
\begin{align*}
& {\rm Lip}_{n-3}^{x_0} (D^2_{x_0} U^0(t), D^2_{x_0} U(t))  \leq  {\rm Lip}_{n-3}^{x_0} (D^2_{x_0} G^0, D^2_{x_0} G)  +C_MT.
\end{align*}
\end{Proposition}

\begin{proof}  We will detail only the proof of Lipschitz estimates of $(\delta^2 U^0/\delta m^2,\delta^2 U/\delta m^2)$. Lipschitz regularity of $\delta U^0_{x_0}/\delta m$ and $D^2_{x_0} U^0$, $\delta U_{x_0}/\delta m$ and $D^2_{x_0} U$ can be proven by following identical lines using the representation formulas that appear in the proof of Proposition \ref{Prop.DerivU02}.

Let us start with $\delta^2 U^0/\delta m^2$. Let $(t_0, m_1, m_2)\in [0,T] \times \Pw^2$ and  $x_0^1, x_0^2 \in \R^d$ be fixed. Let also, as in the proof of the previous Proposition \ref{Prop.DerivU02}, for $i = 1, 2$
\begin{itemize}
\item $(u^i,m^i)$ be the solution to the MFG system \eqref{eq.MFGcoupled} starting from $m_i$ at time $t_0$ with $H$ (and $G$) evaluated at $(x_0^i,x,Du^i(t,x),m^i(t))$ (and $(x_0^i,x,m^i(T))$\,),
\item $(v_i,\rho_i)$ (respectively $(v'_i, \rho'_i)$) be the solution of the first order linearized system \eqref{MFG2lkjenze} with zero right-hand side, initial condition $(t_0, \rho_0)$ (respectively $(t_0,\rho'_0)$) and where  the Hamiltonian and its derivatives are evaluated at $(x_0^i,x,Du^i(t,x),m^i(t))$,
\item $(w^i,\mu^i)$ be the solution to the second order linearized system \eqref{MFG3lkjenze} associated with $(v_i, \rho_i), (v'_i,\rho'_i)$ (and $(u^i,m^i)$), and with zero right-hand side.
\end{itemize}
Recall that \eqref{diff2U0} provides a representation formula for $\delta^2 U^0/\delta m^2$, that is
$$
\frac{\delta^2 U^0}{\delta m^2}(t_0,x^i_0,m_i)(\rho_0, \rho'_0) = \frac{\delta^2 G^0}{\delta m^2}(x^i_0,m^i(T))(\rho_i(T), \rho_i'(T)) + \frac{\delta G^0}{\delta m}(x_0^i,m^i(T))(\mu^i(T)),
$$
and $\frac{\delta^2 U}{\delta m^2}(t_0,x^i_0,x,m_i)(\rho_0, \rho'_0) = w^i(t_0, x)$.
Let us recall the following inequalities
\begin{align*}
& \sup_{t\in [t_0,T]} \dw(m^1(t),m^2(t)) \leq (1+CT) \dw(m_0^1, m_0^2) + CT |x_0^1 - x_0^2|, \\
& \sup_{t\in [t_0,T]} \|\rho_i(t)\|_{-(k-2)} \leq  \left(1+CT \right)\|\rho_0\|_{-(k-2)}, \\
& \sup_{t\in [t_0,T]} \|\rho'_i(t)\|_{-(k-2)} \leq  \left(1+CT \right)\|\rho'_0\|_{-(k-2)}, \\
& \sup_{t\in [t_0,T]} \|\rho_1(t) - \rho_2(t)\|_{-(k-1)} \leq  CT\big(\dw(m_1, m_2) + {|x_0^1 - x_0^2|}\big) \|\rho_0\|_{-(k-2)},\\
& \sup_{t\in [t_0,T]} \|\rho'_1(t) - \rho'_2(t)\|_{-(k-1)} \leq  CT\big(\dw(m_1, m_2) + {|x_0^1 - x_0^2|}\big) \|\rho'_0\|_{-(k-2)},\\
& \sup_{t\in [t_0,T]} \|\mu^i(t)\|_{-(k-1)} \leq CT \|\rho_0\|_{-(k-2)}\|\rho_0'\|_{-(k-2)}, \\
& \sup_{t\in [t_0,T]} \|\mu^1(t) - \mu^2(t)\|_{-k} \leq CT (\dw(m_1,m_2) + |x_0^1 - x_0^2|)\|\rho_0\|_{-(k-2)}\|\rho'_0\|_{-(k-2)},
\end{align*}
that are consequences of Proposition \ref{Prop.LipDep}, \eqref{estivivi}, \eqref{estirhorho}, \eqref{fasdkjd2} and \eqref{muest} respectively. Setting
$$
\theta_T:= CT(\dw(m_1,m_2) + |x_0^1 - x_0^2|)\|\rho_0\|_{-(k-2)}\|\rho'_0\|_{-(k-2)}, 
$$
we obtain, using \eqref{ieukbfdn} also, for any $r\leq n-3$, 
\begin{align*}
& \ds \Big|\Big(\frac{\delta^2 U^0}{\delta m^2}(t,x^1_0,m_1) - \frac{\delta^2 U^0}{\delta m^2}(t,x^2_0,m_2)\Big)(\rho_0, \rho'_0) \Big|^2 + \sup_x\Big|D^r_x\Big(\frac{\delta^2 U}{\delta m^2}(t,x^1_0,x,m_1) - \frac{\delta^2 U}{\delta m^2}(t,x^2_0,x,m_2)\Big)(\rho_0, \rho'_0) \Big|^2 \\
&  \leq (1+CT) \Bigl\{ \Big| \frac{\delta^2 G^0}{\delta m^2}(x^1_0,m^1(T))(\rho_1(T), \rho_1'(T)) - \frac{\delta^2 G^0}{\delta m^2}(x^2_0,m^2(T))(\rho_1(T), \rho_1'(T)) \Big| +\theta_T \Bigr\}^2 \\
& + (1+CT)\Bigl\{ \sup_x  \Big|D^r_x \frac{\delta^2 G}{\delta m^2}(x^1_0,x,m^1(T))(\rho_1(T),\rho_1'(T))- D^r_x \frac{\delta^2 G}{\delta m^2}(x^2_0,x,m^2(T))(\rho_1(T),\rho_1'(T)) \Big|+\theta_T  \Bigr\}^2.
\end{align*} 
Choosing $m_1=m_2=m$ and rearranging gives the Lipschitz  estimates in $x_0$: 
\begin{align*}
& \ds \Big|\Big(\frac{\delta^2 U^0}{\delta m^2}(t,x^1_0,m) - \frac{\delta^2 U^0}{\delta m^2}(t,x^2_0,m)\Big)(\rho_0, \rho'_0) \Big|^2 + \sup_x\Big|D^r_x\Big(\frac{\delta^2 U}{\delta m^2}(t,x^1_0,x,m) - \frac{\delta^2 U}{\delta m^2}(t,x^2_0,x,m)\Big)(\rho_0, \rho'_0) \Big|^2 \\
&  \leq (1+CT) \Bigl\{ \Bigl( \Big| \frac{\delta^2 G^0}{\delta m^2}(x^1_0,m^1(T))(\rho_1(T), \rho_1'(T)) - \frac{\delta^2 G^0}{\delta m^2}(x^2_0,m^1(T))(\rho_1(T), \rho_1'(T)) \Big|^2 + \\
& \qquad\qquad\qquad  \sup_x  \Big|D^r_x \frac{\delta^2 G}{\delta m^2}(x^1_0,x,m^1(T))(\rho_1(T),\rho_1'(T))- D^r_x \frac{\delta^2 G}{\delta m^2}(x^2_0,x,m^1(T))(\rho_1(T),\rho_1'(T)) \Big|^2 \Bigr)^{1/2}
+ \\
& \qquad\qquad+ CT |x_0^1 - x_0|\|\rho_0\|_{-(k-2)}\|\rho'_0\|_{-(k-2)}\Bigr\}^2,
\end{align*} 	
while the choice $x_0^1=x_0^2$ gives similarly the Lipschitz estimates in $m$. 
\end{proof}

%%%%%%%%%%%%%%%%%%%%%%%%%%%%%%%%%
%%%%%%%%%%%%%%%%%%%%%%%%%%%%%%%%%
%%%%%%%%%%%%%%%%%%%%%%%%%%%%%%%%%
\appendix
\section{Estimates for solutions to HJ equations}\label{sec.HJ}

\subsection{Main estimates}\label{Append.main}

In this section, we assume that the data $a$, $h$ and $g$ are smooth and we are looking for a priori estimates on the smooth and globally bounded solution  $u$ to the HJ equation 
\be\label{eq.HJ}
\left\{ \begin{array}{l}
-\partial_t u(t,x) -{\rm Tr}(a(t,x)D^2u(t,x))+h(t,x,Du(t,x))=0\qquad {\rm in}\; (0,T)\times \R^d\\
u(T,x)= g(x) \qquad {\rm in}\; \R^d
\end{array}\right.
\ee
We always assume below that there exists $C_0>0$ and $\gamma\geq 1$ such that 
$$
a(t,x)\geq C_0^{-1}I_d, \qquad \|Da\|_\infty\leq C_0
$$
and  
$$
|D_xh(t,x,p)|\leq C_0(1+|p|^\gamma)\,
$$
 for every $(t,x,p)\in (0,T)\times \R^{d}\times \R^d$.

\begin{Proposition}\label{prop.LipEstiBernstein} (Lipschitz estimates.) For any $M>0$ there exists $T_M, C_M>0$, depending on $M$, $C_0$ and $\gamma$, such that, if $T\in (0,T_M)$ and $\|Dg\|_\infty\leq M$, then 
$$
 \sup_{t\in[0,T]}\|Du(t)\|_\infty\leq \|Dg\|_\infty+C_MT 
$$
\end{Proposition}

\begin{proof} We use a standard Bernstein method. 
Let $v(t,x)=\sum_i u_i^2(t,x)$. Then 
$$
\partial_t v(t,x)= 2 \sum_i u_i(t,x) u_{i,t}(t,x), \qquad 
v_j(t,x)= 2 \sum_i u_i(t,x) u_{ij}(t,x), 
$$
$$ 
v_{jk}(t,x)= 2 \sum_i (u_{ik}(t,x)u_{ij}(t,x)+ u_i(t,x)u_{ijk}(t,x)). 
$$
Thus 
$$
\begin{array}{l}
\ds -\partial_t v-{\rm Tr}(a(t,x)D^2v(t,x))  \\
\ds = 
- 2 \sum_i u_i(t,x) u_{i,t}(t,x) - 2\sum_{i,j,k}a_{jk}(t,x) (u_{ik}(t,x)u_{ij}(t,x)+ u_i(t,x)u_{ijk}(t,x) \\
\ds = - 2\sum_{i,j,k}a_{jk}(t,x) u_{ik}(t,x)u_{ij}(t,x) -2 \sum_i u_i(t,x) D_i\left(\partial_t u+{\rm Tr}(a(t,x)D^2u(t,x))\right)\\
\ds \qquad \qquad + \sum_{i,j,k} u_i(t,x)(a_{jk})_i(t,x)u_{jk}(t,x)
\end{array}
$$
where $(a_{jk})_i$ denotes  the $x_i$-derivative of the element $a_{jk}$ of the matrix $a(t,x)$.

Using the equation for $u$ we find 
\be\label{base}
\begin{array}{l}
\ds -\partial_t v-{\rm Tr}(a(t,x)D^2v(t,x))   \\
%\ds = - 2\sum_{i,j,k}a_{jk}(t,x) u_{ik}(t,x)u_{ij}(t,x) -2 \sum_i u_i(t,x) D_i\left(\partial_t u+{\rm Tr}(a(t,x)D^2u(t,x))\right) \\
%\ds \qquad \qquad + \sum_{i,j,k} u_i(t,x)(a_{jk})_i(t,x)u_{jk}(t,x)\\
\ds = - 2\sum_{i,j,k}a_{jk}(t,x) u_{ik}(t,x)u_{ij}(t,x) -2 \sum_i u_i(t,x) \left( h_i(t,x,Du(t,x))+h_p(t,x,Du(t,x))\cdot Du_i(t,x)\right) \\
\ds \qquad \qquad + \sum_{i,j,k} u_i(t,x)(a_{jk})_i(t,x)u_{jk}(t,x).
\end{array}
\ee
Using our assumptions on $a$ and $h$, we infer that 
$$
\begin{array}{l}
\ds -\partial_t v-{\rm Tr}(a(t,x)D^2v(t,x))+ h_p(t,x,Du(t,x))\cdot Dv(t,x) \\
\m
\qquad \ds \leq -2C_0^{-1}|D^2u|^2 + 2C_0|Du|(1+|Du|^\gamma) +  \|Da\|_\infty |Du| \, |D^2u| \\
\m
\qquad \ds \leq 2C_0|Du|(1+|Du|^\gamma) +c_d \|Da\|_\infty^2C_0  |Du|^2 
\end{array}
$$
for some constant $c_d$ only depending on the dimension $d$. In particular, by maximum principle we estimate 
\be\label{maxT}
\| v\|_{L^\infty(Q_T)} \leq \|Dg\|_{L^\infty(Q_T)}^2+ T \left[2C_0 \|Du\|_{L^\infty(Q_T)} (1+\|Du\|_{L^\infty(Q_T)}^\gamma) +c_d \|Da\|_\infty^2C_0 \|Du\|_{L^\infty(Q_T)}^2\right]
\ee
which implies
\be\label{TM}
\begin{array}{l}
\ds
\| v\|_{L^\infty(Q_T)} \leq \|Dg\|_{L^\infty(Q_T)}^2+ 4T^2C_0^2 + \frac14 \|Du\|_{L^\infty(Q_T)}^2 
\\
\qquad \qquad \ds + \hat C \, T\, \|Du\|_{L^\infty(Q_T)}^2 \left[ \|Du\|_{L^\infty(Q_T)}^{\gamma-1}  +1 \right]
\end{array}
\ee
for some $\hat C$ only depending on $d$ and $C_0$. 
Recall that $\| v\|_{L^\infty(Q_T)}= \|Du\|_{L^\infty(Q_T)}^2$, and define $T_M$ as
$$
T_M= \min \left\{ \frac1{2C_0}M, \frac1{4\hat C(1+(2M)^{\gamma-1})}\right\}\,.
$$
Then it is easy to see that  
\be\label{TM2}
\|Du\|_{L^\infty(Q_T)} \leq 2M \qquad \forall T\leq T_M\,.
\ee
Indeed,  for $T < T_M$ and $\|Du\|_{L^\infty(Q_T)}\leq 2M$, \eqref{TM} implies
$$
\begin{array}{l}
\ds
\|Du\|_{L^\infty(Q_T)}^2 \leq 
\|Dg\|_{L^\infty(Q_T)}^2+ 4T_M^2C_0^2 + \frac14 \|Du\|_{L^\infty(Q_T)}^2 
\\
\qquad \qquad \ds + \hat C \, T_M\, \|Du\|_{L^\infty(Q_T)}^2 \left[ (2M)^{\gamma-1}  +1 \right]
\\
\ds < 
\|Dg\|_{L^\infty(Q_T)}^2+ M^2 + \frac12 \|Du\|_{L^\infty(Q_T)}^2  
\end{array}
$$
hence
$$
\|Du\|_{L^\infty(Q_T)} < 2M
$$
whenever $T < T_M$ and $\|Du\|_{L^\infty(Q_T)}\leq 2M$.   A continuity argument implies that 
$$
\sup\,  \{ T\,:\, \|Du\|_{L^\infty(Q_T)}\leq 2M\} = T_M
$$
so  \eqref{TM2} holds true. Using this information, we deduce from \eqref{maxT} that
$$
\|Du\|_{L^\infty(Q_T)}^2 \leq \|Dg\|_{L^\infty(Q_T)}^2+  C_M \, T\, \|Du\|_{L^\infty(Q_T)} 
$$
where $C_M= 2C_0  (1+(2M)^\gamma) +c_d \|Da\|_\infty^2C_0\, 2M$. Hence
$$
\left(\|Du\|_{L^\infty(Q_T)}- \frac12 C_M\, T\right)^2 \leq \|Dg\|_{L^\infty(Q_T)}^2 + \frac14 C_M^2 \, T^2
$$
which implies
$$
\|Du\|_{L^\infty(Q_T)} \leq   C_M\, T + \|Dg\|_{L^\infty(Q_T)} \,.
$$
\end{proof}

\begin{Proposition}\label{prop.LipEstiBernsteinL} (Lipschitz estimates, linear case.) We now assume that $T\leq 1$ and that
$$
|D_xh(t,x,p)|\leq C_1+C_2|p|\qquad \forall (t,x,p)\in (0,T)\times \R^{d}\times \R^d\,,
$$
for some constants $C_1,C_2>0$. Then there exists a constant $C$, depending on $C_0$, $C_2$ and $\|Da\|_\infty$ only, such that 
$$
\begin{array}{rl}
\ds    \sup_{t\in[0,T]}\|Du(t)\|_\infty \; \leq & \ds  \|Dg\|_\infty(1+CT) + CC_1T.
\end{array}
$$
\end{Proposition}

\begin{proof} Our starting point is inequality \eqref{base} in the previous proof. Using our assumptions on $a$ and $h$ we get: 
$$
\begin{array}{l}
\ds -\partial_t v-{\rm Tr}(a(t,x)D^2v(t,x))+ h_p(t,x,Du(t,x))\cdot Dv(t,x) \\
\m
\qquad \ds \leq -2C_0^{-1}|D^2u|^2 + 2 |Du|(C_1+C_2\, |Du|) +  \|Da\|_\infty |Du| \, |D^2u| \\
\m
\qquad \ds \leq 2 |Du|(C_1+C_2\, |Du|) +c_d \|Da\|_\infty^2C_0  |Du|^2 
\end{array}
$$
which implies
$$
 -\partial_t v-{\rm Tr}(a(t,x)D^2v(t,x))+ h_p(t,x,Du(t,x))\cdot Dv(t,x)   \leq \lambda \, v + 2\, C_1\,  v^{1/2}
$$
where $\lambda = 2 C_2+ c_d \|Da\|_\infty^2C_0$.  By the maximum principle we get
$$
\begin{array}{l}
\ds \|v\|_{L^\infty(Q_T)} \leq  e^{\lambda T} \left( 2 C_1 \, T \|v\|_{L^\infty(Q_T)}^{1/2} + \| Dg\|_\infty^2\right), 
\end{array}
$$
from which we derive that 
$$
\|v\|_{L^\infty(Q_T)}^{1/2} \leq  2\, C_1 \, T\, e^{\lambda T}  + e^{\lambda T/2}\, \| Dg\|_\infty.
$$
Since $T\leq 1$ (and so $e^{\lambda T/2} \leq 1+ c_\lambda T$), the conclusion follows.
\end{proof}

\vskip1em

\begin{Proposition}\label{prop.LipEstiBernstein2} (Second order estimate.)  Assume that $h$ and $a$ are of class $C^2_b$. 
Then, for any $M>0$, there are constants $T_M,C_M>0$, depending on $M$ and on
\be\label{larebzrnesdkl}
\sup_{t\in[0,T]}\|a(t)\|_2+ \sup_{|p|\leq \|Du\|_\infty}\|D^2_{xp}h(\cdot,\cdot,p)\|_\infty
\ee
such that, if $\|D^2 g\|_\infty\leq M$ and $T\in (0,T_M)$, then 
$$
 \sup_{t\in[0,T]}\|D^2u(t)\|_\infty  \leq \|D^2g\|_\infty+ C_MT.
$$
If, in addition, $h$ is affine in $p$, then there is a constant $C$, depending only on $C_0$, $\sup_{t\in[0,T]}\|a(t)\|_2$ and on $\|D^2_{xp}h\|_\infty$, such that, for any $T\in (0,1]$, 
$$
 \sup_{t\in[0,T]} \|D^2u(t) \|_\infty \leq (1+CT)\|D^2 g\|_\infty+ CT\sup_{|p|\leq \|Du\|_\infty}\| D^2_{xx}h(\cdot, \cdot,p)\|_\infty.
$$
\end{Proposition} 

\begin{proof}  We use the Bernstein method again. Let $w(t,x)=\sum_{i,j}u_{ij}^2$. Then 
$$
\begin{array}{l}
\ds -\partial_t w-{\rm Tr}(a(t,x)D^2w(t,x))  \\
\ds = - 2\sum_{i,j,k,l}a_{kl}(t,x) u_{ijk}(t,x)u_{ijl}(t,x) -2 \sum_{i,j} u_{ij}(t,x) D_{i,j}\left(\partial_t u+\sum_{k,l}a_{kl}u_{kl}\right)\\
\ds \qquad \qquad + 2\sum_{i,j,k,l} u_{ij}(t,x)\left( (a_{kl})_i(t,x)u_{jkl}(t,x)+(a_{kl})_j(t,x)u_{ikl}(t,x)+(a_{kl})_{ij}u_{kl} \right).
\end{array}
$$
So
\be\label{oajlrenztrektj:d}
\begin{array}{l}
\ds -\partial_t w-{\rm Tr}(a(t,x)D^2w(t,x))  \\
\ds = - 2\sum_{i,j,k,l}a_{kl} u_{ijk}u_{ijl} -2 \sum_{i,j} u_{ij} \left(h_{ij}+h_{i,p}\cdot Du_j+h_{j,p}\cdot Du_i+ h_{pp}Du_i\cdot Du_j +h_pDu_{ij}\right)\\
\ds \qquad \qquad + 2\sum_{i,j,k,l} u_{ij}(t,x)
\left( (a_{kl})_i(t,x)u_{jkl}(t,x)+(a_{kl})_j(t,x)u_{ikl}(t,x)+(a_{kl})_{ij}u_{kl} \right)
\end{array}
\ee
which yields, using the ellipticity of $a(t,x)$,
$$
\begin{array}{l}
\ds -\partial_t w-{\rm Tr}(a(t,x)D^2w(t,x))+ h_p(t,x,Du(t,x))\cdot Dw(t,x)  \\
\ds \leq  - 2C_0^{-1}|D^3u|^2  + C_h |D^2u| \left(1+|D^2u|+ |D^2u|^2\right)
+C |D^2u|\left(\|a\|_1\, |D^3u|+ \|a\|_2\, |D^2u| \right)\,
\end{array}
$$
for some constant $C_h$ depending on $\sup_{|p|\leq \|Du\|_\infty}\|D^2_{x,p}h(\cdot,\cdot,p)\|_\infty$. Young's inequality leads to
$$
\begin{array}{l}
\ds -\partial_t w-{\rm Tr}(a(t,x)D^2w(t,x))+ h_p(t,x,Du(t,x))\cdot Dw(t,x)  \\
\ds \leq   C |D^2u| \left(1+|D^2u|+ |D^2u|^2\right)\,
\end{array}
$$
where now $C$ depends on $\|a\|_2$ as well.
We conclude using maximum principle as in the proof of Proposition \ref{prop.LipEstiBernstein}. 
\vskip1em

If $h$ is affine in $p$, then with the same estimates we deduce from \eqref{oajlrenztrektj:d}:  
$$
\begin{array}{l}
\ds  -\partial_t w-{\rm Tr}(a(t,x)D^2w(t,x))+ h_p(t,x,Du(t,x))\cdot Dw(t,x) \\
\m
\ds \leq |D^2u| \left( 2 \|D_{xx} h\|_\infty + C |D^2u|\right)
\\
\m
\ds \leq C w + 2\, \|D_{xx} h\|_\infty\, |D^2 u|\,,
\end{array}
$$
where $C$ depends  on $\|a\|_{2}$, $C_0$ and $\sup_{|p|\leq \|Du\|_\infty}\|D^2_{x,p}h(\cdot,\cdot,p)\|_\infty$. The conclusion follows as in Lemma \ref{prop.LipEstiBernsteinL}.
\end{proof}

\vskip1em

\begin{Proposition} (Third order estimate) Assume that $h$ and $a$ (and the solution $u$) are of class $C^3_b$. 
Then there is a constant $C$, depending on $\|D^2u\|_\infty$, on $\|Da\|_\infty+\|D^2a\|_\infty+\|D^3a\|_\infty$ and on
$$
 \sup_{|p|\leq \|Du\|_\infty} \left\{\|D^3_{(x,p)}h(\cdot,\cdot,p)\|_\infty+\|h_{pp}(\cdot,\cdot,p)\|_\infty\right\}, 
$$
 such that,  for any $T\in (0,1]$, 
$$
 \sup_{t\in[0,T]}\|D^3u(t)\|_\infty \leq (1+CT)\|D^3g\|_\infty+ CT,
$$
\end{Proposition} 

\begin{proof} Let $w=\sum_{ijk} u_{ijk}^2$. Then 
$$
\begin{array}{l}
\ds -\partial_t w-{\rm Tr}(a(t,x)D^2w(t,x))  \\
\ds = - 2\sum_{i,j,k,l,m}a_{lm}(t,x) u_{ijkl}(t,x)u_{ijkm}(t,x) -2 \sum_{i,j} u_{ijk}(t,x) D_{i,j,k}\left(\partial_t u+\sum_{l,m}a_{lm}u_{lm}\right)\\
\ds  \qquad + 2\sum_{i,j,k,l,m} u_{ijk}(t,x)\Bigl( (a_{lm})_{ijk} u_{lm}+(a_{lm})_{ij}u_{klm}+(a_{lm})_{ik}u_{jlm}+(a_{lm})_{jk}u_{ilm} \\
 \qquad \qquad\qquad\qquad +(a_{lm})_{i}u_{jklm}+ (a_{lm})_{j}u_{iklm}+ (a_{lm})_{k}u_{ijlm}\Bigr)
\end{array}
$$
So
\be\label{fourth}
\begin{array}{l}
\ds -\partial_t w-{\rm Tr}(a(t,x)D^2w(t,x))  \\
\ds = - 2\sum_{i,j,k,l,m}a_{lm}(t,x) u_{ijkl}(t,x)u_{ijkm}(t,x)  -2 \sum_{i,j} u_{ijk}(t,x) D_{i,j,k} \left\{h\right\}\\
\ds  \qquad + 2\sum_{i,j,k,l,m} u_{ijk}(t,x)\Bigl( (a_{lm})_{ijk} u_{lm}+(a_{lm})_{ij}u_{klm}+(a_{lm})_{ik}u_{jlm}+(a_{lm})_{jk}u_{ilm} \\
 \qquad \qquad\qquad\qquad +(a_{lm})_{i}u_{jklm}+ (a_{lm})_{j}u_{iklm}+ (a_{lm})_{k}u_{ijlm}\Bigr)\,.
\end{array}
\ee
As before, the coercivity of $a$ implies
$$
- 2\sum_{i,j,k,l,m}a_{lm}(t,x) u_{ijkl}(t,x)u_{ijkm}(t,x) \leq - 2C_0^{-1} |D^4u|^2\,,
$$
whereas last term in \eqref{fourth} is estimated as
$$
\begin{array}{l}
 2\sum_{i,j,k,l,m} u_{ijk}(t,x)\Bigl( (a_{lm})_{ijk} u_{lm}+(a_{lm})_{ij}u_{klm}+(a_{lm})_{ik}u_{jlm}+(a_{lm})_{jk}u_{ilm} \\
 \qquad \qquad\qquad\qquad +(a_{lm})_{i}u_{jklm}+ (a_{lm})_{j}u_{iklm}+ (a_{lm})_{k}u_{ijlm}\Bigr)
 \\
 \m
 \ds \qquad \leq  C_0^{-1} |D^4u|^2 +   |D^3 u| \left( 2 \| D^3 a\|_\infty\, |D^2 u |+ C \, |D^3 u|\right)
 \,,
\end{array}
$$
for some $C$ depending on $C_0$ and $\|D^2 a\|_\infty$. Finally,  a direct computation of $D_{i,j,k} \left\{h\right\}$ and a straightforward estimate of all terms involved imply
$$
\begin{array}{l}
-2 \sum_{i,j} u_{ijk}(t,x) D_{i,j,k} \left\{h\right\} \leq - h_p(t,x,Du(t,x))\cdot Dw(t,x) 
\\
\m
\ds \qquad +
C\, |D^3 u| \left[    \|D^2 h\|_\infty\,  |D^3 u| (1+ |D^2 u|)+  \|D^3 h\|_\infty
(1+ |D^2 u|^3 ) 
 \right]\,.
 \end{array}
$$
Hence, putting all together we deduce from \eqref{fourth}:
$$
\begin{array}{l}
\ds -\partial_t w-{\rm Tr}(a(t,x)D^2w(t,x)) +h_p(t,x,Du(t,x))\cdot Dw(t,x) \\
\m
\ds \leq C\, |D^3 u| \left[    \|D^2 h\|_\infty\,  |D^3 u| (1+ |D^2 u|)+  \|D^3 h\|_\infty
(1+ |D^2 u|^3 ) \right]
\\
\m
\ds
\qquad +   |D^3 u| \left( 2 \| D^3 a\|_\infty\, |D^2 u |+ C \, |D^3 u|\right)
\\
\m
\ds
\qquad
\leq  C |D^3 u|^2 + C |D^3 u|
 \,,
\end{array}
$$
where $C$ depends also on $\|D^ku\|_\infty$ for $k\leq 2$. We conclude as in   Proposition \ref{prop.LipEstiBernsteinL}.
\end{proof}

\begin{Lemma} \label{lem.High} (Higher order estimate) Let  $n\in \N$ with $n\geq 3$ and assume that $h$ and $a$  (and the solution $u$) are of class $C^n_b$. 
There is a constant $C$, depending on $n$, $d$, $\sup_t \|u(t)\|_{n-1}$,  $\sup_t \|a(t)\|_n$ and on 
\be\label{kjaetbzrlef}
 \sup_{|p|\leq \|Du\|_\infty} \sum_{k=0}^n \|D^k_{(x,p)}h(\cdot,\cdot,p)\|_\infty, 
\ee
 such that,  for any $T\in (0,1]$, 
 $$
 \sup_t \|D^nu(t)\|_\infty\leq (1+CT)\|D^ng\|_\infty+ CT.
 $$
\end{Lemma}

\begin{proof}  Let $\ds w:= \sum_{|k|=n} u^2_k$ where the multi-index $k=(k_1,\dots, k_d)$ belongs to $\N^d$ and $|k|=\sum_i k_i$. Then 
$$
\begin{array}{l}
\ds -\partial_t w-{\rm Tr}(a(t,x)D^2w(t,x))  \\
\ds = - 2\sum_{|k|=n}\, \sum_{i,j}a_{ij}(t,x) u_{k,i}(t,x)u_{k,j}(t,x) -2 \sum_{|k|=n} u_{k}(t,x) D_{k}\left\{\partial_t u+{\rm Tr}(aD^2u)\right\}\\
\ds  \qquad + 2\sum_{|k|=n}u_k\left(D_k({\rm Tr}(aD^2u))-{\rm Tr}(aD^2u_k)\right).
\end{array}
$$
As  $n\geq 3$, a simple induction argument shows that $D_k\{h\}$ is of the form 
$$
D_{k}\left\{h\right\}= f_k+g_k\cdot D^nu +h_p\cdot Du_k
$$
where  the map
$$
f_k= f_k(t,x,Du(t,x), \dots, D^{n-1}u(t,x)) 
$$
is a polynomial function of the derivatives of $u$ up to order $n-1$ with coefficients involving derivatives of $h$ with respect to $(x,p)$ up to order $n$  computed at $(t,x,Du(t,x))$, while
$$
g_k\cdot D^nu= \sum_{|\xi|=n-1} \sum_{z+\xi=k}\, D_{z,p}h(t,x,Du(t,x))Du_\xi+ h_{pp}(t,x,Du(t,x)) Du_{z}Du_\xi\,,
%\sum_{k_j\geq 1} h_{pp}(t,x,Du(t,x)) Du_j\cdot Du_{k-e_j}. 
$$
where $\xi$ is any multi-index of length $n-1$, $z$ is a multi-index of length $1$ ($z=e_j$ for some $j\in \{1,\dots,d\}$) and $\xi+ z= k$.

Therefore 
$$
\begin{array}{l}
\ds -\partial_t w-{\rm Tr}(a(t,x)D^2w(t,x)) + h_p\cdot Dw \\
\m
\ds = - 2\sum_{i,j}\sum_{|k|=n}a_{ij}(t,x) u_{k,i}(t,x)u_{k,j}(t,x) -2 \sum_{|k|=n} u_{k}(t,x) \left(f_k+g_k\cdot D^nu \right)\\
\m
\ds  \qquad + 2\sum_{|k|=n}u_k\left(D_k({\rm Tr}(aD^2u))-{\rm Tr}(aD^2u_k)\right)
\\
\m \ds 
\leq - 2C_0^{-1}\sum_{|k|=n} |Du_k|^2 + C |u_k| (1+ |u_k|)
\\
\m
\ds  \qquad + 2\sum_{|k|=n}u_k\left(D_k({\rm Tr}(aD^2u))-{\rm Tr}(aD^2u_k)\right)
\end{array}
$$
where $C$ depends on $\sup_t \|u(t)\|_{n-1}$ and the quantity in \eqref{kjaetbzrlef}. 
Last term can be estimated as before: the higher order quantity involves $Du_k$, so we have  through  Young's inequality 
$$
2\sum_{|k|=n}u_k\left(D_k({\rm Tr}(aD^2u))-{\rm Tr}(aD^2u_k)\right)
\leq 2C_0^{-1}\sum_{|k|=n} |Du_k|^2 + C\, |u_k|(1+ |u_k|),
$$
for some  $C$ depending on $\sup_t \|a(t)\|_n$ and $\sup_t \|u(t)\|_{n-1}$. Finally, we conclude with maximum principle, as in Lemma \ref{prop.LipEstiBernsteinL}.
\end{proof}

\begin{Proposition}\label{prop.High} (Higher order estimate, further informations) Let  $n\in \N$ with $n\geq 3$ and assume that $h$ and $a$  (and the solution $u$) are of class $C^n_b$. For any $M>0$, there are constants $K_M,T_M>0$, depending on $M$, $C_0$ and $\gamma$,  and a constant $C_M>0$ depending on 
$$
\sup_{t\in [0,T_M]} \|a(t)\|_n + \sup_{|p|\leq K_M} \sum_{k=0}^n \|D^k_{(x,p)}h(\cdot,\cdot,p)\|_\infty,
$$
 such that, if $\|g\|_n \le M$, then, for any $T\in (0,T_M)$ and any $r\leq n$, 
we have 
$$
 \sup_{t\in[0,T]} \|D^r_xu(t)\|_\infty \leq \|D^r_xg\|_\infty+  C_M T
 $$
 and therefore
\be\label{unest}
 \sup_{t\in[0,T]} \|u(t)\|_n \leq \|g\|_n+  C_M T.
\ee
\end{Proposition} 

\begin{proof}
The proof is a straightforward combination of Propositions \ref{prop.LipEstiBernstein}, \ref{prop.LipEstiBernstein2} and Lemma \ref{lem.High}.
\end{proof}

We finally address the same issue for (uncoupled)  systems of linear parabolic equations: let $(u^l)_{l=1, \dots, k}$ solve the system 
$$
\left\{\begin{array}{l}
-\partial_t u^l -{\rm Tr}(a(t,x)D^2u^l)+V(t,x)\cdot Du^l +f^l(t,x)=0 \qquad {\rm in}\; (0,T)\times \R^d\\
u^l(T,x)= g^l(x) \qquad {\rm in}\;  \R^d
\end{array}\right.
$$
 where 
$a$, $V$ and the $f^l$ are bounded in $C^n_b$ independently of $t\in[0,1]$, for some $n\in \N^*$. Note that the diffusion and the drift terms are independent of $l$. 

\begin{Proposition}[Higher order estimate, systems of affine equations]  \label{prop.highS} There is a constant $C$, depending on $k$, $d$, $\sup_t \|a(t)\|_n$ and on $\sup_t \|V(t)\|_n$, such that, for any $T\in (0,1]$ and for any $r \le n$, 
$$
 \sup_{t, x} \left(\sum_{l=1}^k |D^r_x u^l(t, x)|^2\right)^{1/2} \leq (1+CT) \sup_{x} \left(\sum_{l=1}^k |D^r_x g^l(x)|^2\right)^{1/2}+  CT \sup_l(\|g^l\|_{r} + \sup_t \|f^l(t)\|_r).
 $$
 In particular, if $k=1$,   for any $r \le n$
 $$
 \sup_{t\in [0,T]} \|D^r_x u(t)\|_\infty \leq (1+CT) \|D^r_x g\|_\infty+ CT \sup_t \|D^r_x f(t)\|_\infty.
 $$ 
\end{Proposition}

 The only small point here is that the supremum over $x$ is outside the sum (and not inside  as it would be given by simply applying to each $u^l$ the previous Propositions). 

\begin{proof} The proof runs exactly along the same lines as before and so we just explain briefly the idea for $r=0$. Let us consider $v(t,x)=\sum_{l=1}^k (u^l(t,x))^2$. Then $v$ solves 
$$
-\partial_t v -{\rm Tr}(aD^2v)+V\cdot Dv = - 2 \sum_{l=1}^k u^l f^l -\sum_{i,j,l} a_{ij} u^l_iu^l_j
$$
We infer the result by using the positivity of $a$ and the maximum principle. 

%%%%%%%%%%%%%%%%%%%%%%%%%%%%
%%%%%%%%%%%%%%%%%%%%%%%%%%%%
\subsection{Systems with parameters}

In this section we revisit the above estimates for specific systems of Hamilton-Jacobi equations involving a parameter $y$. The motivation for the specific form of the system is the analysis of the MFG problems with a major player. Note that here the variables-parameter couple $(x; y)$ plays the role of $(x_0; x)$ in the HJ system \eqref{syst.2bis} analyzed throughout Section \ref{s:sshj}. As usual, we discuss linear and nonlinear systems separately.

%%%%%%%%%%%%%%%%%%%%%%%%%%%%
\subsubsection{Nonlinear systems}

Here we consider the system consisting in a coupling of a non-linear HJ equation with a linear one:  
\be\label{eq.HJsyst}
\left\{ \begin{array}{l}
-\partial_t u^0(t,x) -\Delta u^0(t,x)+h^0(t,x,Du^0(t,x))=0\qquad {\rm in}\; (0,T)\times \R^d,\\
-\partial_t u(t,x;y)-\Delta u(t,x;y) + h^0_p(t,x,Du^0(t,x))\cdot Du(t,x;y) +f(t,x;y)= 0 \; {\rm in}\; (0,T)\times \R^d,\\
u^0(T,x)= g^0(x), \; u(T,x)= g(x;y) \qquad {\rm in}\; \R^d,
\end{array}\right.
\ee
where $h^0:[0,T]\times \R^d \times \R^d\to \R$ and $f:[0,T]\times \R^d\times \R^{d_1}\to \R$ ($d_1$ being the space parameter of the variable $y$) are smooth maps satisfying in addition the bounds: 
\be\label{eq.concsyst}
 |D_{x,p}h^0(t,x,p)| + |D^2_{x,p}h^0(t,x,p)|  \leq C_0(|p|^\gamma+1), 
\ee
for some $\gamma>0$ and $C_0>0$.

\begin{Proposition}\label{prop.HighSyst}
%\footnote{Probably we can remove $f$ here. This proposition is used in Prop. 6.3 only, where no $f$ appears} 
Let  $r, n\in \N$ and assume (in addition to \eqref{eq.concsyst}) that $h^0, h^0_p$ are of class $C^{r}_b$ and that $f$ is bounded in $C^{r,n}_b$ independently of $t\in[0,1]$ for some $n \in \N$. For any $M>0$, there are constants $K_M,T_M>0$, depending on $M$, $C_0$ and $\gamma$ in \eqref{eq.concsyst},  and a constant $C_M>0$ depending on 
$$
 \sup_{|p|\leq K_M} \sum_{k=0}^r \|D^k_{(x,p)}h^0(\cdot,\cdot,p)\|_\infty + \sum_{k=0}^r \|D^k_{(x,p)}h_p^0(\cdot,\cdot,p)\|_\infty, \quad \sup_t\|f(t)\|_{r,n}
$$
 such that, if $\|g^0\|_r+ \|g\|_{r,n}\leq M$ and $T\in (0,T_M)$, and if $(u^0,u)$ is the solution to \eqref{eq.HJsyst}, then we have, for $l \le n$,
$$
\sup_{t,x,y }\Bigl(|D^r u^0(t,x)|^2+|D_x^r D^l_y u(t,x;y)|^2\Bigr)^{1/2} \leq  \sup_{x,y} \Bigl(|D^r g^0(x)|^2+|D_x^r D^l_y g(x;y)|^2\Bigr)^{1/2}+  C_M T.
$$
\end{Proposition} 

Let us recall that $D^r_x D^l_yu = (\partial^\beta_x\partial^\alpha_y u)_{|\beta| = r, |\alpha|=l}$, hence $|D^r_x D^l_y u|^2 = \sum_{|\beta| = r, |\alpha|=l} (\partial^\beta_x \partial^\alpha_y u)^2$. Let us also point out that the main difference compared to Proposition \ref{prop.High} is that we need to estimate $u^0$ and $u$ at the same time. 

\begin{proof} The proof uses  the same technique as for a single Hamilton-Jacobi equation without parameter. We explain only the main changes.
We first prove the result for $l=0$.

By the maximum principle we can first bound $|u^0|^2+|u|^2$ by $\|(g^0)^2+g^2\|_\infty+CT$. Next  we address the Lipschitz estimate.  We claim that, for any $M>0$ and any $n\in\N$, if $\|Dg^0\|_\infty+\|D_x g\|_\infty \leq M$, then there exists $T_M$ and $C_M$ (depending on $M$, $C_0$, $n$ and $\gamma$ in \eqref{eq.concsyst} only) such that 
$$
\sup_{t,x,y }\Bigl(|D u^0(t,x)|^2+|D_x  u(t,x;y)|^2\Bigr)^{1/2} \leq  \sup_{x,y} \Bigl(|D g^0(x)|^2+|D_x g(x;y)|^2\Bigr)^{1/2}+ C_MT(1+\sup_t\|D_x f(t)\|_\infty).
$$
To this aim, let us set:     
$\ds
v(t,x)= \sum_{i=1}^d ((u^0_i)^2 + (u_i)^2).
$
Then following the computation in the proof of Proposition \ref{prop.LipEstiBernstein}, we find: 
\begin{align*}
& -\partial_t v-\Delta v(t,x) \\
& \ds = 
- 2 \sum_i \left( u^0_i D_{x_i} (\partial_t u^0+\Delta u^0) + u_iD_{x_i} (\partial_t u+\Delta u)\right) - 2(|D^2u^0|^2+|D^2u|^2) \\
& \ds = 
- 2 \sum_i \left( u^0_i ( h^0_{x_i}+h^0_p\cdot Du^0_i) + u_i(h^0_{p,x_i}\cdot Du+ h^0_{pp}Du^0_i\cdot Du+ h^0_p \cdot Du_i+f_i)\right) - 2(|D^2u^0|^2+|D^2u|^2), 
\end{align*}
so that 
\begin{align*}
& -\partial_t v-\Delta v(t,x) + h^0_p\cdot Dv \\
& \ds = 
- 2 \sum_i \left( u^0_i  h^0_{x_i} + u_i(h^0_{p,x_i}\cdot Du+ h^0_{pp}Du^0_i\cdot Du+f_i)\right) - 2(|D^2u^0|^2+|D^2u|^2).
\end{align*}
Using our assumption on $h^0$ we get
\begin{align*}
& -\partial_t v-\Delta v(t,x) + h^0_p\cdot Dv \leq  C v^{1/2} (|v|^{\theta}+1+\|D_xf\|_\infty),
\end{align*}
for some $C>0$ and $\theta>0$ which depend on $C_0$ and $\gamma$ only. We derive from this the Lipschitz estimate thanks to the maximum principle exactly as in the proof of Proposition \ref{prop.LipEstiBernstein}. 

The higher order estimates can be checked exactly as in Propositions \ref{prop.LipEstiBernstein2} and Lemma \ref{lem.High}, so we omit the proof. Note that higher order estimates on $D^r u^0$ and $D^r u$ depend on $D^{r-1} u^0$ and $D^{r-1} u$, but this dependance affects the constant $C_M$ only.

Let us finally explain how to handle the derivative with respect to $y$: we note that $\partial_y^\alpha u$ satisfies the  same linear equation as $u$ with $f$ replaced by $\partial_y^\alpha f$, and the final datum $g$ is replaced by $\partial_y^\alpha g$. So, in order to estimate $D_xD^l_y(u^0,u)$ for instance,  we just set $c_l = (\sum_{|\alpha|=l}1)^{-1}$, $v^\alpha = \sum_{i=1}^d (c_l (u^0_i)^2+ (\partial^\alpha_y u_i)^2)$ and $w=\sum_{|\alpha|=l} v^\alpha$. As above,
\begin{align*}
& -\partial_t v^\alpha-\Delta v^\alpha(t,x) + h^0_p\cdot Dv^\alpha \leq  C (v^\alpha)^{1/2} (|v^\alpha|^{\theta}+1+\|D_x \partial_y^\alpha f \|_\infty) \leq  C w^{1/2} (|w|^{\theta}+1+\|D_x D^l_y f \|_\infty),
\end{align*}
and summing up one concludes the desired inequality, noting that $w=\sum_{|\beta|=1}  (\partial^\beta_x u^0)^2+ \sum_{|\beta|=1, |\alpha|=l} (\partial^\beta_x \partial^\alpha_y u)^2$. 

\end{proof}

\subsubsection{Linear systems}

We also need to quantify the regularity of linear systems of the form
\be\label{eq.HJsystL}
\left\{ \begin{array}{l}
\ds -\partial_t u^0(t,x) -\Delta u^0(t,x)+V^0(t,x)\cdot Du^0(t,x)+f^0(t,x)=0\qquad {\rm in}\; (0,T)\times \R^d,\\
\ds -\partial_t u(t,x;y)-\Delta u(t,x;y) + V^0(t,x)\cdot Du(t,x;y)+V(t,x;y)\cdot Du^0(t,x) \\
\ds \qquad \qquad \qquad\qquad\qquad\qquad\qquad\qquad\qquad\qquad +f(t,x;y)= 0 \qquad {\rm in}\; (0,T)\times \R^d,\\
\ds u^0(T,x)= g^0(x), \; u(T,x)= g(x;y) \qquad {\rm in}\; \R^d
\end{array}\right.
\ee

\begin{Proposition}\label{prop.prop.HighSystL} Assume that, independently on $t \in (0,1]$,  $V^0, f^0$ are bounded in $C^r$, and $V, f$ are bounded $C^{r,n}_b$ for some $r,n\geq 0$. Then, if $(u^0,u)$ is a solution of \eqref{eq.HJsystL} which is bounded in $C^{r}_b \times C^{r,n}_b$ and if $\|g^0\|_{r}+ \|g\|_{r,n}\leq M$, we have, for any $T \in (0, 1]$,  $l \le n$,
\begin{align*}
& \sup_{t,x,y} \Bigl(|D_x^r u^0(t,x)|^2 + |D^r_x D^l_y u(t,x;y)|^2\Bigr)^{1/2} \\
&\qquad \leq  \sup_{x,y}\Bigl(|D^r_x  g^0(x)|^2+ |D^r_x D^l_y g(x;y)|^2\Bigr)^{1/2}  +C_M T,
\end{align*}
where $C_M$ depends on $M$, the bounds on $V^0, f^0$ and $V, f$ in $C^r$ and $C^{r,n}_b$  respectively. 

In addition, for $r = 0$ and $l \le n$, we have
\begin{multline*}
\sup_{t,x,y} \Bigl(|u^0(t,x)|^2 + | D^l_y u(t,x;y)|^2\Bigr)^{1/2}\leq (1+CT)\sup_{x,y} \Bigl(|g^0(x)|^2+ |D^l_y g(x;y)|^2\Bigr)^{1/2}  +CT( \| f^0\|_\infty+\|D^l_y f\|_\infty),
\end{multline*}
where $C$ depends just on the bound of $V^0$ and $V$.
\end{Proposition}

\begin{proof} We first note that the derivatives of $u$ with respect to the parameter $y$ solve a system which has the same structure as the one for $u$: so we just need to check the result for $n=0$, and proceed as in the proof of Proposition \ref{prop.HighSyst} for $n > 0$.

Let us start with the $L^\infty$ bounds: We consider $\tilde v:= (u^0)^2+u^2$. Then $v$ satisfies 
\begin{align*}
-\partial_t \tilde v-\Delta \tilde v & = -2u^0 \left( \partial_t u^0+\Delta u^0\right)-2u\left(\partial_t u+\Delta u\right) -2(|Du^0|^2+|Du|^2)\\ 
& = -2u^0 \Bigl( V^0(t,x)\cdot Du^0(t,x)+f^0(t,x)\Bigr)-2u\Bigl(V^0(t,x)\cdot Du(t,x) \\
& \qquad \qquad +V(t,x;y)\cdot Du^0(t,x) +f(t,x;y)\Bigr)-2(|Du^0|^2+|Du|^2)\\
& \leq C\tilde v +\tilde v^{1/2}(\|f^0\|_\infty+\|f\|_\infty), 
\end{align*}
where $C$ depends on $\|V^0\|_\infty$ and $\|V\|_\infty$ only. This implies the result for $r=n=0$. 

We now check the $C^1$ estimate. Let us set as usual $v(t,x)= \sum_{i=1}^d ((u^0_i)^2 + (u_i)^2)$. 
Then  
\begin{align*}
& -\partial_t v-\Delta v(t,x) 
%\\
%& \ds 
= 
- 2 \sum_i \left( u^0_i D_{x_i} (\partial_t u^0+\Delta u^0) + u_iD_{x_i} (\partial_t u+\Delta u)\right) - 2(|D^2u^0|^2+|D^2u|^2) \\
& \ds = 
- 2 \sum_i \Bigl( u^0_i ( V^0_{x_i}\cdot Du^0+ V^0\cdot Du^0_i+f^0_i) + u_i(V^0_{x_i}\cdot Du+V^0\cdot Du_i +V_{x_i}\cdot Du^0+ V\cdot Du^0_i+f_i)\Bigr)\\
& \qquad \qquad \qquad  - 2(|D^2u^0|^2+|D^2u|^2)\\
&  \leq Cv +v^{1/2}(\|Df^0\|_\infty+ \|D_xf\|_\infty),  
\end{align*}
where $C$ depends on the $C^1$ bound on $V^0$ and on $V$ and on $d$ only. This implies the estimate for $r=1$ and $n=0$. 

As for the $C^2$ estimate, let us set as usual $w(t,x)= \sum_{i,j=1}^d ((u^0_{ij})^2 + (u_{ij})^2)$. 
Then  
$$
-\partial_t w-\Delta w(t,x) \leq Cw +C w^{1/2}(1+\|D^2 f^0\|_\infty+ \|D^2_x f\|_\infty + \|D u^0\|_\infty+ \|D_x u\|_\infty),  
$$
where $C$ depends on the $C^1$ bound on $V^0$ and on $V$ and on $d$ only. We then get the estimate for $r=2$ and $n=0$ by the maximum principle and using the previous bounds for $D u^0, D u$. 

The estimate on higher order derivatives can be checked in a way similar and we omit the proof. 
\end{proof}

%%%%%%%%%%%%%%%%%%%%%%%%%%%%%%%%%%%%%
%%%%%%%%%%%%%%%%%%%%%%%%%%%%%%%%%%%%%
%%%%%%%%%%%%%%%%%%%%%%%%%%%%%%%%%%%%%
%%%%%%%%
\section{Functions on $\Pw$}\label{sec.diffUm}

\subsection{A criterium of differentiability}

Here we introduce a simple criterium for a map $U$, depending of the measure, to be of class $C^1$. 

\begin{Lemma}\label{lem.condC1} Let $U:\Pw\to \R$ be continuous. For $(s,m,y)\in [0,1]\times \Pw\times \R^d$ we set 
$$
\hat U(s; m,y):= U((1-s)m+s\delta_y).
$$ 
If the map $s\to \hat U(s; m,y)$ has a derivative at $s=0$ and if its derivative at $0$, $\ds \frac{d}{ds}_{|_{s=0}} \hat U:\Pw\times \R^d \to \R$ is continuous and bounded, then $U$ is of class $C^1$ with 
$$
\frac{\delta U}{\delta m}(m,y)= \frac{d}{ds}\hat U(0; m,y).
$$
\end{Lemma}

\begin{proof}
We have  to show that, for any $m_0,m_1\in \Pw$, we have 
\begin{align*}
U(m_1)-U(m_0) & =\int_0^1 \int_{\R^d} \frac{d}{ds}\hat U(0; (1-s)m_0+sm_1,y)(m_1-m_0)(dy).
\end{align*}
 Before starting the proof,  let  us  note that the continuity assumption of $\frac{d}{ds}\hat U$ at $s=0$ implies its continuity at any $s\in [0,1]$, replacing $m$ by  $(1-s)m+s\delta_y$.

Let us start by considering the case where $m_0$ is fixed and $m_1$ is an empirical measure: $m_1= m^N_{\bf y}:=\frac{1}{N} \sum_{k=1}^N \delta_{y_k}$ for some $N\in \N$, $N\geq 1$, $y_k\in \R^d$. The general case will be treated next by approximation. 

All the measures we are going to manipulate in the next lines belong to the set 
$$
K:= \{ \alpha_0m_0+\sum_{k=1}^{N} \alpha_k \delta_{y_k}, \; \alpha_k\geq 0, \sum_{k=0}^N \alpha_k=1\}
$$
which is compact in $\Pw$. So, by continuity of $\frac{d}{ds}\hat U$, if we fix $\ep>0$, there exists $\delta\in (0,1/2)$ such that, if $m',m''\in K$ with $\dw(m,m')<\delta$ and $s\in [0,\delta]$, then 
\be\label{ahkerzbsln}
\sup_k \left| \frac{d}{ds} \hat U(s;m,y_k)- \frac{d}{ds} \hat U(0; m', y_k)\right|\leq \ep. 
\ee
Our first step consists in showing that, for $s>0$ small enough (to be defined below) and for  any $m\in K$, we have 
\begin{align}\label{ehsbfxvnck}
& \left| U((1-s)m+sm^N_{\bf y})-U(m)  - s\int_{\R^d} \frac{d}{ds} \hat U(0; m, y)m^N_{\bf y}(dy) \right| \leq C(\ep s+s^2) ,
 \end{align}
 where $C$ depends on the sup norm of $\frac{d}{ds} \hat U$ on $[0,1]\times K\times \{y_k, \; k=1,\dots, N\}$. 
In order to prove \eqref{ehsbfxvnck}, we define $\alpha_k=\frac{s}{N-(N-k)s}$ for $k=0, \dots, N$ and note that
\be\label{stars}
\prod_{l=k}^N (1-\alpha_l)= 1-\frac{(N+1-k)s}{N}\,. 
\ee
We now define by induction 
\be\label{lkejznredg,fmv}
m_0= m, \qquad m_k= (1-\alpha_k)m_{k-1}+ \alpha_k \delta_{y_k}
\ee
and using \eqref{stars} we get  
\begin{align*}
m_N & = \prod_{k=1}^N (1-\alpha_k) m+   \alpha_n\delta_{y_N}+  \sum_{k=1}^{N-1} \alpha_k\delta_{y_k} \prod_{l=k+1}^N (1-\alpha_l) \\
& = (1-s) m + \sum_{k=1}^N \delta_{y_k}\frac{s}{N-(N-k)s} (1-\frac{(N-k)s}{N}) = (1-s)m+s m^N_{\bf y}.
 \end{align*}
So, by the definition of $m_{k+1}$ in function of $m_k$ in \eqref{lkejznredg,fmv},  
\begin{align*}
& U((1-s)m+sm^N_{\bf y})-U(m)   = \sum_{k=0}^{N-1} U(m_{k+1})-U(m_k) \\ 
& = \sum_{k=0}^{N-1}  \hat U(\alpha_{k+1}; m_k, y_{k+1})-\hat U(0; m_k, y_{k+1}) 
 = \sum_{k=0}^{N-1}  \int_0^{\alpha_{k+1}} \frac{d}{ds} \hat U(\tau; m_k, y_{k+1})d\tau\,.
 \end{align*}
Let us assume that $s\in (0,\delta)$. As $s<1/2$, we have  $\alpha_k\leq 2s/N$ for any $k$, and thus 
$$
\dw(m_k, m) \leq C s
$$
for a constant $C$ which depends on $m_0$ and on the $y_k$ (but not on $m\in K$ nor on $s\in (0,\delta)$). We now require that $s$ is so small that $Cs<\delta$. Then, for any $k$ and any $\tau\in (0,\alpha_k)$,  we have by \eqref{ahkerzbsln}:
$$
\left| \frac{d}{ds} \hat U(\tau; m_k, y_{k+1})-  \frac{d}{ds} \hat U(0; m, y_{k+1})\right| \leq \ep. 
$$
We infer from this that 
\begin{align*}
& \left| U((1-s)m+sm^N_{\bf y})-U(m)  - \sum_{k=0}^{N-1}  \alpha_{k+1} \frac{d}{ds} \hat U(0; m, y_{k+1})\right| \leq C\ep \sum_{k=0}^{N-1} \alpha_{k+1} .
 \end{align*}
As $|\alpha_k- s/N|\leq Cs^2/N$, we conclude that \eqref{ehsbfxvnck} holds. 
 
%Therefore 
%\begin{align*}
%& \left| U((1-s)m+sm^N_{\bf y})-U(m)  - s\int_{\R^d} \frac{d}{ds} \hat U(0; m, y)m^N_{\bf y}(dy) \right| \leq C(\ep s+s^2) .
% \end{align*}
The next step in the proof consists in showing that 
\begin{align}\label{ehsbfxvnckekznrsf}
U(e^{-1}m_0+(1-e^{-1})m^N_{\bf y})-U(m_0)&  = \int_0^{e^{-1}}  \int_{\R^d} \frac{d}{ds} \hat U(0; (1-\tau)m_0+ \tau m^N_{\bf y}, y)m^N_{\bf y}(dy) \frac{d\tau}{1-\tau}. 
 \end{align}
For this, let us now choose $T\in \N$ large and let 
 $$
 m_n= \left(1-\frac{1}{T}\right)^n m_0+ \left(1-\left(1-\frac{1}{T}\right)^n\right) m^N_{\bf y} \qquad n\in \{0, \dots, T\}.
 $$
 We have
 $$
m_{n+1}= \left(1-\frac{1}{T}\right) m_n+ \frac{1}{T} m^N_{\bf y} \qquad n\in \{0, \dots, T\}.
$$
So, by \eqref{ehsbfxvnck}, 
\begin{align*}
& \left| U(m_{T})-U(m_0) - T^{-1}\sum_{n=0}^{T-1}  \int_{\R^d} \frac{d}{ds} \hat U(0; m_n, y)m^N_{\bf y}(dy) \right| \\
& \leq \sum_{n=0}^{T-1} \left| U((1-1/T)m_n+(1/T) m^N_{\bf y})-U(m_n) -   T^{-1} \int_{\R^d} \frac{d}{ds} \hat U(0; m_n, y)m^N_{\bf y}(dy) \right| \\
& \leq C \sum_{n=0}^{T-1} (\ep/T+(1/T)^2) \leq C(\ep+T^{-1}). 
 \end{align*}
We let $T\to+\infty$ and then $\ep\to0$ to conclude by continuity of $U$ and of $\frac{d}{ds}\hat U$ that 
\begin{align*}
U(e^{-1}m_0+(1-e^{-1})m^N_{\bf y})-U(m_0)&  = \int_0^{1} \int_{\R^d} \frac{d}{ds} \hat U(0; e^{-s}m_0+ (1-e^{-s})m^N_{\bf y}, y)m^N_{\bf y}(dy) ds\\
& = \int_0^{e^{-1}}  \int_{\R^d} \frac{d}{ds} \hat U(0; (1-\tau)m_0+ \tau m^N_{\bf y}, y)m^N_{\bf y}(dy) \frac{d\tau}{1-\tau}.
 \end{align*}
 This is \eqref{ehsbfxvnckekznrsf}. 
 
By continuity of $U$ and of $\frac{d}{ds} \hat U$ and by density of the empirical measures, one obtains from \eqref{ehsbfxvnckekznrsf} that, for any measure $m_0,m_1\in \Pw$: 
\be\label{abelsdkjnezdc}
U(e^{-1}m_0+(1-e^{-1})m_1)-U(m_0) = \int_0^{e^{-1}}  \int_{\R^d} \frac{d}{ds} \hat U(0; (1-\tau)m_0+ \tau m_1, y)m_1(dy) \frac{d\tau}{1-\tau}. 
 \ee
 Choosing $m_1=m_0$ then implies the normalization convention 
 $$
  \int_{\R^d} \frac{d}{ds} \hat U(0; m_0, y)m_0(dy)=0  
 $$
  for any $m_0 \in \Pw$. In particular, this yields
 $$
 \int_{\R^d} \frac{d}{ds} \hat U(0; (1-\tau)m_0+ \tau m_1, y)m_1(dy)= (1-\tau) \int_{\R^d} \frac{d}{ds} \hat U(0; (1-\tau)m_0+ \tau m_1, y)(m_1-m_0)(dy)\,.
 $$
 Inserting this relation in \eqref{abelsdkjnezdc} gives the more standard form: 
 \begin{align*}
 U(e^{-1}m_0+(1-e^{-1})m_1)-U(m_0) & = \int_0^{e^{-1}}  \int_{\R^d} \frac{d}{ds} \hat U(0; (1-\tau)m_0+ \tau m_1, y)(m_1-m_0)(dy) d\tau. 
 \end{align*}  
Using again the continuity of $U$ and of $\frac{d}{ds} \hat U$, one easily deduce from this the desired equality.  
\end{proof}

\subsection{Interpolation and Ascoli Theorem in $\Pw$} \label{sec.interpo}

In the proof of Lemma \ref{lem.estiUN}, we have used two interpolation Lemmas. The first one is standard (see, for instance, \cite[Lemma II.3.1]{LSU}): we recall it because we need a specific setting. The second one is an adaptation to $\Pw$ of the same techniques. 

\begin{Lemma}\label{lem.interp0} Let $W:[0,1]\times \R^{d_1} \to \R^{d_2}$ be  Holder continuous in time locally uniformly in space: for any $R>0$, there exists $C_{0,R}>0$ and $\alpha_R>0$ such that 
$$
|W(t,y)-W(s,y)|\leq C_{0,R} |t-s|^\alpha \qquad \forall (s,t,y)\in [0,1]\times [0,1]\times \R^{d_1}\; {\rm with }\; |y|\leq R\;{\rm and}\; |t-s|\leq \alpha_R,
$$
and such that  $D_y W$ is  Holder continuous in space uniformly in time: there exists $C_1>0$ such that 
$$
\left|D_yW(t,y_0)-D_yW(t,y_1)\right|\leq C_1|y_0-y_1|^\delta \qquad \forall (t,y_1,y_2)\in [0,1]\times \R^{d_1}\times \R^{d_1}.
$$
 Then $D_y W$ is  Holder continuous in time locally uniformly in space: 
 \begin{align*}
& \left|D_yW(t,y)-D_yW(s,y)\right| 
\leq C_R \,  |t-s|^{\frac{\alpha  \delta}{(1+\delta)}} \\
& \qquad  \qquad \forall (s,t,y)\in [0,1]\times [0,1]\times \R^{d_1}\; {\rm with}\; |y|\leq R \;{\rm and}\; |t-s|\leq \alpha_R',
\end{align*}
for some constants $C_R>0$ and $\alpha_R'$ only depending on $C_{0,R+1}$, $\alpha_{R+1}$, $C_1, \alpha$ and $\delta$.  
\end{Lemma}

\begin{Remark}{\rm The proof below also shows that, if in addition $W$ is Holder continuous in time  uniformly in space (i.e., $C_{0,R}$ and $\alpha_R$ do not depend on $R$) and if $D_yW$ is bounded, then $D_y W$ is also Holder continuous in time  uniformly in space. 
}\end{Remark}

\begin{proof} Fix $y_0,y_1\in \R^d$ with $|y_0|\leq R$ and $|y_1|\leq R+1$. Let $y_\tau= (1-\tau)y_0+\tau y_1$ for $\tau\in [0,1]$. We have 
\begin{align*}
&\left|\int_0^1 (D_yW(t,y_\tau)-D_yW(s,y_\tau))\cdot (y_1-y_0)d\tau \right|
\\
\m &\qquad = \left| W(t,y_1)- W(t,y_0)- W(s,y_1)+  W(s,y_0) \right|
%\\
%\m & \qquad 
\leq 2C_{0,R+1} \, |t-s|^\alpha\,.
\end{align*}
 So 
\begin{align*}
& \left| (D_yW(t,y_0)-D_yW(s,y_0))\cdot (y_1-y_0)\right|  
%\\
%& \qquad 
\leq \left| \int_0^1 (D_yW(t,y_0)-D_yW(t,y_\tau))\cdot (y_1-y_0)d\tau\right| 
\\
&  
+ \left|\int_0^1 (D_yW(t,y_\tau)-D_yW(s,y_\tau))\cdot (y_1-y_0)d\tau\right|
%\\
%& \qquad\qquad 
+ \left| \int_0^1 (D_yW(s,y_\tau)-D_yW(s,y_0))\cdot (y_1-y_0)d\tau\right|\\
&\qquad \leq 2  C_{0,R+1} |t-s|^{\alpha}  +  2C_1|y_1-y_0|^{1+\delta},
\end{align*}
using also the H\"older continuity of $D_y W$.
Choosing $y_1= y_0+ h v $, with $|v|=1$, we get
$$
 \left |[D_yW(t,y)-D_yW(s,y) ]\cdot v  \right| \leq \frac{2C_{0,R+1}}{|h|} |t-s|^\alpha + 2C_1 |h|^{\delta}\,.
$$
Optimizing with respect to $h\in (0,\alpha_{R+1}]$ and $|v|=1$, we find the result for  $|t-s|\leq \alpha_R'$ for a suitable constant $\alpha_R'$ depending on $C_{0,R+1}$, $\alpha$, $C_1$ and $\delta$. 
\end{proof}

\begin{Lemma}\label{lem.interp} Let $W:[0,1]\times \Pw\to \R^{d_2}$ be  Holder  continuous, locally in time and uniformly in measure: there exists $\alpha\in (0,1]$ and, for any $R>0$ there exists $C_{0,R}>0$ such that
$$
|W(t,m)-W(s,m)|\leq C_{0,R} |t-s|^\alpha \qquad \forall m\in \Pw\; {\rm with}\; M_2(m)\leq R, \; \forall s,t\in [0,1], 
$$
(where $M_2(m)=(\int_{\R^d}|y|^2m(dy))^{1/2}$) and such that  $\frac{\delta W}{\delta m}$ and $D_mW$ are   bounded and  $D_mW$ is Holder continuous with respect to the measure uniformly in time: there exists $\gamma,\delta \in (0,1]$ and  $C_{1}>0$ such that
$$
\left|D_mW(t,m_0,y_0)-D_mW(t,m_1,y_1)\right|\leq C_1\left( \dw^{\gamma}(m_0,m_1) +|y_0-y_1|^\delta\right)
$$
for any $t\in [0,1]$ and any $(m_i,y_i)\in \Pw\times \R^d$. 
 Then $D_m W$ is  Holder continuous in time locally uniformly in $(m,y)\in \Pw\times \R^d$: for any $R>0$ , there exists a  constant $C_R>0$, depending on $R$, $\|D_mW\|_\infty$, $C_{0,R+1}$, $C_1$, $\alpha,\gamma$ and $\delta$, such that 
 \begin{align*}
& \left|D_mW(t,m,y)-D_mW(s,m,y)\right| 
\leq C_R |t-s|^{\alpha\gamma/((2+\gamma)(1+\delta))}  , 
\end{align*}
for any $s,t\in [0,1]$ and any $(m,y)\in \Pw\times \R^d$ with $|y|\leq R$ and  $M_2(m)\leq R$.
\end{Lemma}

\begin{proof} Let $R\geq 1$. Fix  $m_0,m_1\in \Pw$ with $M_2(m_i)\leq R$ and set $m_\tau= (1-\tau)m_0+\tau m_1$. Then 
\begin{align*}
& \left|\int_0^1\int_{\R^d} \left(\frac{\delta W}{\delta m}(t, m_\tau, y)-\frac{\delta W}{\delta m}(s, m_\tau, y)\right) (m_1-m_0)(dy) d\tau\right|\\
\m &\qquad  =\left| W(t,m_1)-W(t,m_0) - W(s,m_1)+W(s,m_0)\right|
\\
\m & \qquad \leq 2 C_{0,R} |t-s|^\alpha\,.
\end{align*} 
As 
\begin{align*}
& \left| \int_{\R^d} \left(\frac{\delta W}{\delta m}(t, m_0, y)-\frac{\delta W}{\delta m}(s, m_0, y)\right) (m_1-m_0)(dy) \right| \\
& \qquad \leq \left| \int_0^1 \int_{\R^d} \left(\frac{\delta W}{\delta m}(t, m_\tau, y)-\frac{\delta W}{\delta m}(s, m_\tau, y)\right) (m_1-m_0)(dy) d\tau\right| \\
& \qquad \qquad + \left|  \int_0^1\int_{\R^d} \left(\frac{\delta W}{\delta m}(t, m_\tau, y)-\frac{\delta W}{\delta m}(t, m_0, y)\right) (m_1-m_0)(dy) d\tau\right|\\
& \qquad \qquad +  \left| \int_0^1\int_{\R^d} \left(\frac{\delta W}{\delta m}(s, m_\tau, y)-\frac{\delta W}{\delta m}(s, m_0, y)\right) (m_1-m_0)(dy) d\tau\right| \,,
%& \qquad \leq 2C_0|t-s|\alpha+ 4C_1 \dw^\gamma(m_0,m_1).
\end{align*}
we obtain, by our Holder continuity assumption on $D_mW$: 
\begin{align*}
& \left| \int_{\R^d} \left(\frac{\delta W}{\delta m}(t, m_0, y)-\frac{\delta W}{\delta m}(s, m_0, y)\right) (m_1-m_0)(dy) \right| \\
& \qquad \leq 2 C_0 |t-s|^\alpha 
%\\
%& \qquad \qquad 
+ \sup_{\tau\in[0,1]}\left\|   D_mW(t, m_\tau, \cdot)-D_mW(t, m_0, \cdot)\right\|_\infty\dk(m_0,m_1)\\
& \qquad \qquad + \sup_{\tau\in[0,1]} \left\|  D_mW(s, m_\tau, \cdot)-D_mW(s, m_0, \cdot)\right\|_\infty\dk(m_0,m_1) \\
& \qquad \leq 2C_0|t-s|^\alpha+ 2C_1 \dw^{\gamma}(m_0,m_1)\dk(m_0,m_1).
\end{align*}
For any  $y_0\in\R^d$ with $|y_0|\leq R$, let  $m_1=(1-\theta)m_0+ \theta \delta_{y_0}$ for some $\theta\in (0,1]$ to be chosen below.
Note that
$$
\dk(m_1,m_0)\leq \theta \int_{\R^d} |y_0-x|m_0(dx)\leq  \theta (|y_0|+(M_2(m_0))^{1/2} ) \leq 2\theta R,
$$
(since $R\geq 1$) while 
$$
\dw (m_1,m_0) \leq (\theta \int_{\R^d} |y_0-x|^2m_0(dx))^{1/2}\leq  (2\theta)^{1/2}  (|y_0|^2+ M_2^2(m_0))^{1/2}\leq 2\theta^{1/2}R.
$$
We get, by the convention on the derivative and our previous  estimates: 
\begin{align*}
& \left| \frac{\delta W}{\delta m}(t, m_0, y_0)-\frac{\delta W}{\delta m}(s, m_0, y_0)\right|  
%\\
%&\qquad  
=   \frac{1}{\theta} \left| \int_{\R^d} \left(\frac{\delta W}{\delta m}(t, m_0, y)-\frac{\delta W}{\delta m}(s, m_0, y)\right) (m_1-m_0)(dy) \right|\\
&\qquad  \leq \frac{1}{\theta}\left[ 2C_{0,R}|t-s|^\alpha+ cC_1 R^{1+\gamma}\theta^{1+\gamma/2}\right]\,,
\end{align*}
where $c$ is universal. If $|t-s|$ is small enough such that $C_{0,R}|t-s|^\alpha/(cC_1R^{1+\gamma})\leq 1$, then we choose $\theta^{1+\gamma/2} := C_{0,R}|t-s|^\alpha/(cC_1R^{1+\gamma})$ and obtain 
\begin{align}\label{liaekzlred}
& \left| \frac{\delta W}{\delta m}(t, m_0, y_0)-\frac{\delta W}{\delta m}(s, m_0, y_0)\right|  
%\\
%&\qquad  
\leq c C_{0,R}^{\gamma/(2+\gamma)}C_1^{1/(1+\gamma/2)}  R^{2(1+\gamma)/(2+\gamma)} |t-s|^{\alpha\gamma/(2+\gamma)}\,,
\end{align}
where $c$ is another universal constant. 

To show the regularity in time of $D_mW$, we just need to apply Lemma \ref{lem.interp0} to $\delta W/\delta m$ since, by \eqref{liaekzlred}, $\delta W/\delta m$ is  locally Holder in time locally uniformly in space (the constant depending also on the measure) and  $D_y\delta W/\delta m=D_mW$ is globally bounded and  Holder in $y$ uniformly in time by assumption. We can  remove the smallness restriction on $|t-s|$ by using the fact that $D_mW$ is globally bounded.
\end{proof}

In the proof of Theorem \ref{theo.ShortTime} we also used the following version of Arzela-Ascoli Theorem. 

\begin{Lemma}\label{lem.Arzela-Ascoli}  Let $(X,d)$ be a locally compact space and $W^N:X\times \Pw\to \R$ be a family of uniformly bounded and  locally uniformly continuous maps: there exists $x_0\in X$ such that, for any $R>0$, there exists a continuous nondecreasing modulus $\omega_R: [0,+\infty)\to [0,+\infty)$ with $\omega_R(0)=0$ such that 
\be\label{hyp.reguWN}
|W^N(x,m)-W^N(x',m')|\leq \omega_R(d(x,x')+ \dw(m,m')),
\ee
for any $x,x'\in X$ and $m,m'\in \Pw$ with $d(x,x_0)\leq R$, $d(x',x_0)\leq R$, $M_2(m)\leq R$, $M_2(m')\leq R$. 

Then there exists a continuous map $W :X\times \Pw\to \R$ and a subsequence (denoted in the same way) such that $(W^N)$ converges  to $W$ pointwisely in $m$ and locally uniformly in $x$: for any $R>0$ and any $m\in \Pw$, 
\be\label{Cvpointwise+unif}
\lim_{N\to+\infty} \sup_{d(x,x_0)\leq R} |W^N(x,m)-W(x,m)| = 0\,.
\ee
\end{Lemma}

The only (very small) issue in the result is that $\Pw$ is not locally compact, so that the standard Arzela-Ascoli Theorem cannot be applied. 

\begin{proof} Let $D$ be an enumerable dense family of $X\times \Pw$. By a diagonal argument we can find a subsequence (denoted in the same way) such that, for any $(x,m)\in D$, $(W^N(x,m))$ converges to some $W(x,m)$. Let us note that, by our regularity assumption \eqref{hyp.reguWN} and using the fact that $X\times \Pw$ is complete, $W$ can be extended to the whole space $X\times \Pw$ into a continuous map which satisfies 
\be\label{hyp.reguWNbis}
|W(x,m)-W(x',m')|\leq \omega_R(d(x,x')+ \dw(m,m')),
\ee
for any $x,x'\in X$ and $m,m'\in \Pw$ with $d(x,x_0)\leq R$, $d(x',x_0)\leq R$, $M_2(m)\leq R$, $M_2(m')\leq R$. 

We claim that, for any $(x,m)\in X\times \Pw$, $(W^N(x,m))$ converges to  $W(x,m)$. Indeed, fix $\ep>0$, $R= 2(1+d(x,x_0)+M_2(m))$.  Then there is $(x',m')\in D$ such that $d(x',x_0)\leq R$, $M_2(m')\leq R$ and $\omega_R((d(x,x')+ \dw(m,m'))\leq \ep/3$. Let also $N_0$ be so large that 
$|W^N(x',m')-W(x',m')|\leq \ep/3$ for $N\geq N_0$. Then, for $N\geq N_0$, we have  
\begin{align*}
& |W^N(x,m)-W(x,m)|  \\ 
& \qquad \leq |W^N(x,m)-W^N(x',m')|+ |W^N(x',m')-W(x',m')|+ |W(x',m')-W(x,m)| \leq \ep\,, 
\end{align*}
where we used \eqref{hyp.reguWN} and \eqref{hyp.reguWNbis} in the last inequality. 

It remains to show that \eqref{Cvpointwise+unif} holds. Fix $\ep>0$ and let $\eta>0$ be such that $\omega(\eta)\leq \ep/3$. As $X$ is locally compact, we can find $x_1, \dots x_n$ such that any point $x\in B_X(x_0,R)$ is at a distance at most $\eta$ from one of the $(x_i)_{i=1, \dots, n}$. Let $N_0$ be so large  that $|W^N(x_i,m)-W(x_i,m)|\leq \ep/3$ for any $i=1,\dots, n$. Then, for any $x\in B_X(x_0,R)$ and any $N\geq N_0$, we have (for $i$ such that $d(x,x_i)\leq \eta$, so that $\omega_R(d(x,x_i))\leq \ep/3$): 
\begin{align*}
& |W^N(x,m)-W(x,m)|  \\ 
& \qquad \leq |W^N(x,m)-W^N(x_i,m)|+ |W^N(x_i,m)-W(x_i,m)|+ |W(x_i,m)-W(x,m)| \leq \ep\,, 
\end{align*}
where we used again \eqref{hyp.reguWN} and \eqref{hyp.reguWNbis} in the last inequality. This shows  \eqref{Cvpointwise+unif}. 
\end{proof}

%\newpage


\begin{thebibliography}{abc99xyz}
\bibitem{ALLRS} Achdou, Y., Lasry, J.-M., Lions, P.-L., Rostand, A., Scheinkman, J.,  
{\it Models for the Economy of Oil}, in preparation. 

\bibitem{Ah} Ahuja, S. (2016). {\it Wellposedness of mean field games with common noise under a weak monotonicity condition.} SIAM Journal on Control and Optimization, 54(1), 30-48.

\bibitem{AKR} Albeverio, S., Kondratiev, Y. G., and Röckner, M. (1998). {\it Analysis and geometry on configuration spaces.} Journal of functional analysis, 154(2), 444-500.

\bibitem{AGS}
Ambrosio, L.,  Gigli, N., Savaré, G.
{\sc Gradient flows in metric spaces and in the space of probability measures.}
 Second edition. Lectures in Mathematics ETH Zürich. Birkhäuser Verlag, Basel, 2008.
 
 \bibitem{BCY15} Bensoussan, A., Chau, M. H. M., Yam, S. C. P. (2015). {\it Mean field Stackelberg games: Aggregation of delayed instructions.} SIAM Journal on Control and Optimization, 53(4), 2237-2266.

\bibitem{BCY16} Bensoussan, A., Chau, M. H. M., Yam, S. C. P. (2016). {\it Mean field games with a dominating player.} Applied Mathematics \& Optimization, 74(1), 91-128.

%U. Bessi. Existence of solutions of the master equation in the smooth case. SIAM J. Math. Anal,48(1):204?228, 2016.

\bibitem{BLP14} Buckdahn, R., Li, J., Peng, S. (2014). {\it Nonlinear stochastic differential games involving a major player and a large number of collectively acting minor agents.} SIAM Journal on Control and Optimization, 52(1), 451-492.

\bibitem{CK14} Caines, P. E., Kizilkale, A. C. (2014). {\it Mean field estimation for partially observed LQG systems with major and minor agents.} IFAC Proceedings Volumes, 47(3), 8705-8709.

\bibitem{CDLL} Cardaliaguet P., Delarue F.,  Lasry J.-M., Lions P.-L. (2019). {\sc The Master Equation and the Convergence Problem in Mean Field Games.} (AMS-201) (Vol. 381). Princeton University Press.

\bibitem{CCP} Cardaliaguet P., Cirant M., Porretta A. 
{\it Remarks on Nash equilibria in mean field game models with a major player.} Preprint. 

\bibitem{CaDebook} Carmona R., Delarue F. {\sc Probabilistic Theory of Mean Field Games with Applications I \& II}. 2017 - Springer Verlag.
DOI 10.1007/978-3-319-58920-6

\bibitem{CDLcn} Carmona, R., Delarue, F., and Lacker, D. (2016). {\it Mean field games with common noise}. The Annals of Probability, 44(6), 3740-3803.

\bibitem{CW16} Carmona, R., Wang, P. (2016). Finite state mean field games with major and minor players. arXiv preprint arXiv:1610.05408.

\bibitem{CW17} Carmona, R.,  Wang, P. (2017). An alternative approach to mean field game with major and minor players, and applications to herders impacts. Applied Mathematics \& Optimization, 76(1), 5-27.

\bibitem{CZ16} Carmona, R.,  Zhu, X. (2016). {\it A probabilistic approach to mean field games with major and minor players.} The Annals of Applied Probability, 26(3), 1535-1580.



\bibitem{CgCrDe} Chassagneux, J. F., Crisan, D., and Delarue, F. (2014). {\it Classical solutions to the master equation for large population equilibria. }arXiv preprint arXiv:1411.3009.
 
\bibitem{GS14-2}
Gangbo, W., Swiech A. (2015) {\it Existence of a solution to an equation arising from the theory of mean field games}.  Journal of Differential Equations, 259(11), 6573-6643.


\bibitem{GT}
Gilbarg, D. and Trudinger, N.S. {\sc
Elliptic partial differential equations of second order.}
Springer, 2015.

\bibitem{HCMieeeAC06} Huang, M., Malham\'e, R.P., Caines, P.E.  (2006). 
{\it Large population stochastic dynamic games: closed-loop McKean-Vlasov systems and the Nash certainty equivalence
principle.} Communication in information and systems. Vol. 6, No. 3, pp. 221-252. 

\bibitem{H10} Huang, M. (2010). {\it Large-population LQG games involving a major player: the Nash certainty equivalence principle.} SIAM Journal on Control and Optimization, 48(5), 3318-3353.

 
%\bibitem{IL}   Ishii, H., Lions, P.-L. {\it Viscosity Solutions of Fully Nonlinear Second-Order Elliptic Partial Differential Equations}.  J. Diff. Eq. 83 (1990), 26-78.

\bibitem{LW} Lacker, D. and Webster, K. (2015). {\it Translation invariant mean field games with common noise.} Electronic Communications in Probability, 20.

\bibitem{LSU}
Lady\v{z}enskaja O.A., Solonnikov V.A. and Ural'ceva N.N. {\sc
Linear and quasilinear equations of parabolic type.}
 Translations of Mathematical Monographs, Vol. 23 American Mathematical Society, Providence, R.I. 1967.
  
\bibitem{LL06cr1} Lasry, J.-M., Lions, P.-L. {\it Jeux \`a champ moyen. I. Le cas stationnaire.}
C. R. Math. Acad. Sci. Paris  343  (2006),  no. 9, 619-625.

\bibitem{LL06cr2} Lasry, J.-M., Lions, P.-L. {\it Jeux \`a champ moyen. II. Horizon fini et contr\^ole optimal.}
C. R. Math. Acad. Sci. Paris  343  (2006),  no. 10, 679-684.

\bibitem{LL07mf} Lasry, J.-M., Lions, P.-L. {\it Mean field games.}  Jpn. J. Math.  2  (2007),  no. 1, 229--260.

\bibitem{LL18} Lasry, J. M., Lions, P. L. (2018). {\it Mean-field games with a major player.} C. R. Math. Acad. Sci. Paris. 

\bibitem{LLperso} Lions, P.L.  Cours au Coll\`{e}ge de France. 
www.college-de-france.fr.

\bibitem{May} Mayorga S. {\it Short time solution to the master equation of a first order mean field game}. arXiv:1811.08964, 2018

%\bibitem{MZ} Mou, C., and Zhang, J. (2019). Weak Solutions of Mean Field Game Master Equations. arXiv preprint arXiv:1903.09907.

\bibitem{LLperso2} Lions P.-L. {\it Estimées nouvelles pour les équations quasilinéaires.} Seminar in Applied
Mathematics at the Collège de France. http://www.college-de-france.fr/site/pierrelouis-lions/seminar-2014-11-14-11h15.htm, 2014.

%\bibitem{ORS} Overbeck, L., Rockner, M., Schmuland, B. (1995). {\it An analytic approach to Fleming-Viot processes with interactive selection.} The Annals of Probability, 23(1), 1-36.
%
\bibitem{RaRu98}
  Rachev, S.T.,  and Schendorf, L. R. {\sc Mass Transportation problems.} Vol. I: Theory; Vol. II : Applications. Springer-Verlag, 1998.
 
\bibitem{villani}
Villani, C. {\sc Optimal transport, old and new.} Springer-Verlag, 2008. 

\end{thebibliography}
\end{document}